%% file: main-inv-perm-rev0.tex
\documentclass[12pt]{article}
\usepackage{preprint-layout-light}
\usepackage{preprint-notation}

\DeclareMathOperator*{\argmin}{arg\,min}

\newcommand{\data}{\ensuremath{y_L^\dagger}}
\newcommand{\datah}{\ensuremath{y_{L,h}^\dagger}}

\title{Latent Inversion of Material Coefficients from Boundary Data via Finite Tests and Neural Surrogates}

\date{\today}

\author{
Erik Burman,
Mats G. Larson,
Karl Larsson,
Jonatan Vallin
}
\date{}

\begin{document}

\maketitle

\begin{abstract}
We reconstruct a spatially varying material coefficient in a scalar elliptic equation from finitely many boundary excitations, each producing a full Dirichlet trace. To mitigate the ill-posedness and the cost of repeated PDE solves, we restrict the coefficient to a low-dimensional family, specified analytically or learned from samples, and solve the inverse problem in its latent coordinates. We consider a \(C^1\) parametrization with \(m\) latent coordinates and full-rank derivative at a reference point. If the continuous linearized Neumann-to-Dirichlet map is injective on the corresponding tangent space, at most \(m\) excitations suffice for local injectivity and Lipschitz stability. Convergence of the coefficient sensitivities then transfers this stability to conforming finite element discretizations. For sufficiently fine meshes, the stability constant and neighborhood can be chosen independently of the mesh size. Under a local residual-comparison condition, uniform accuracy of the surrogate forward map yields coefficient-error bounds separating representation error, data noise, finite element error, and surrogate error. Derivative accuracy additionally preserves the surrogate's own local stability. All stability statements are local to a reference coefficient. Two-dimensional numerical experiments combine analytic and learned representations of inclusions and crack-like coefficients with neural forward surrogates. They illustrate latent-space reconstruction, reduced online cost, and further improvement from optional FEM-based refinement.
\end{abstract}

\input{sec-intro-rev0.tex}

\input{sec-method-rev0.tex}

\input{sec-analysis-rev0.tex}

\input{sec-analysis-discrete-rev0.tex}
\input{sec-material-rev0.tex}
\input{sec-numerical-experiments-rev0.tex}
\input{sec-conclusions-rev0.tex}
\input{sec-appendix-rev0.tex}

\bigskip
\paragraph{Acknowledgement.} MGL, KL, and JV were supported in part by the Swedish Research
Council grants Nos.\ 2021-04925, 2025-05562;  Knut and Alice Wallenberg Foundation grant No.\ 
KAW 2025.0277; the Swedish Research Programme Essence.  EB was supported in part by EPSRC grants Nos.\
EP/V050400/1,  EP/T033126/1,  EP/X042650/1.

During the preparation of this manuscript, the authors used OpenAI ChatGPT and Codex with the GPT-5.5 and GPT-5.6 Sol models to assist with drafting and revising text, improving mathematical exposition, checking notation and internal consistency, editing LaTeX, and developing computational code. These tools were not used to fabricate or directly alter research data or numerical results. All AI-assisted material, including mathematical statements and computational code, was critically reviewed and validated by the authors, who take full responsibility for the accuracy, originality, and integrity of the manuscript.

\bibliographystyle{habbrv}
\footnotesize{
\bibliography{perm_biblio}
}

\bigskip
\bigskip
\noindent
\footnotesize {\bf Authors' addresses:}

\smallskip
\noindent
Erik Burman,  \quad \hfill \addressuclshort\\
{\tt e.burman@ucl.ac.uk}

\smallskip
\noindent
Mats G. Larson,  \quad \hfill \addressumushort\\
{\tt mats.larson@umu.se}

\smallskip
\noindent
Karl Larsson, \quad \hfill \addressumushort\\
{\tt karl.larsson@umu.se}

\smallskip
\noindent
Jonatan Vallin, \quad \hfill \addressumushort\\
{\tt jonatan.vallin@umu.se}

\end{document}

%% file: sec-intro-rev0.tex
\section{Introduction}
\label{sec:introduction}

\paragraph{Background.}
Recovering an unknown parameter in a partial differential equation from
measurements is a classical inverse problem. A prominent example is the
Calder\'on problem, in which the conductivity coefficient of a scalar elliptic
equation is reconstructed from boundary measurements. Applications include
medical imaging and the geosciences.

We consider the reconstruction of a coefficient
\(a:\Omega\to\mathbb{R}\) from measurements of the solution to
\begin{equation}
-\nabla \cdot \big( a \nabla u \big) = 0 \quad \text{in } \Omega\label{eq:scalar_elliptic}
\end{equation}
where \(\Omega\subset\mathbb{R}^d\) is a bounded domain. In many applications,
only boundary measurements of \(u\) are available, but even when \(u\) is known
throughout \(\Omega\), the recovery of \(a\) reduces to a linear but ill-posed problem.

In the standard formulation, the data are encoded by the
Dirichlet-to-Neumann map, or equivalently the Neumann-to-Dirichlet (NtD) map \(\mathcal N_a\), which we use below.
More precisely, for each compatible Neumann datum 
\(g_N\in H^{-1/2}(\partial\Omega)\), we assume that the corresponding 
Dirichlet trace \(u|_{\partial \Omega}=\mathcal N_a g_N\) of the solution to 
\eqref{eq:scalar_elliptic} is known.

To formulate the inverse problem, let \(\mathcal A_{\mathrm{ad}}\) denote the
admissible coefficient class and let
\(a^\dagger\in\mathcal A_{\mathrm{ad}}\) be the unknown exact coefficient. The
Calder\'on problem is to recover \(a^\dagger\) from the complete NtD operator
\(\mathcal N_{a^\dagger}\). Under suitable assumptions on the coefficient, the
domain, and the dimension, this operator uniquely determines \(a^\dagger\).
Equivalently, the recovery problem can be written as
\begin{equation}
a^\star\in\argmin_{a\in\mathcal A_{\mathrm{ad}}}
\|\mathcal N_a-\mathcal N_{a^\dagger}\|
\quad\text{where}\quad
\|\mathcal N_a-\mathcal N_{a^\dagger}\|
=
\;\,
\sup_{\mathclap{\substack{
0\ne g_N\\
\langle g_N,1\rangle_{\partial\Omega}=0
}}}
\;\,
\frac{
\|(\mathcal N_a-\mathcal N_{a^\dagger})g_N\|_{H^{1/2}(\partial\Omega)}
}{
\|g_N\|_{H^{-1/2}(\partial\Omega)}
}\label{eq:data_fitting}
\end{equation}
where \(a^\star\) denotes an optimizer, and \(a^\star=a^\dagger\) under the
preceding uniqueness assumptions.

Despite this uniqueness, the reconstruction is severely unstable.
Under appropriate a priori smoothness bounds on the coefficients and
assumptions on the domain and dimension, the discrepancy
\(\|\mathcal N_a-\mathcal N_{a^\dagger}\|\leq\varepsilon\) yields a
conditional logarithmic stability estimate of the form \cite{Ale88}
\begin{equation}
\|a-a^\dagger\|_{L^\infty(\Omega)}\leq C|\log(\varepsilon)|^{-s},
\quad 0<\varepsilon<\varepsilon_0<1
\quad \text{for some }s\in(0,1)\label{eq:introduction-logarithmic-stability}
\end{equation}
The constants depend on the a priori coefficient class. This instability
limits reconstruction over general infinite-dimensional classes, whereas
additional structural restrictions can permit stronger stability estimates.

Related approximation barriers have been rigorously quantified for unique
continuation \cite{BNO25}, which plays a central role in the analysis of the
Calder\'on problem. These results illustrate that discretization and
regularization must contend with the underlying instability rather than
eliminate it. Access to full operator data (the entire NtD map) is a further
practical limitation, since such data are rarely available in applications.
Even when full operator data are
available, solving \eqref{eq:data_fitting} numerically may require many
iterations, each involving repeated PDE solves. These limitations motivate
reductions in both the coefficient space and the number of boundary experiments.

\paragraph{Summary of Main Ideas.}
\label{par:main-idea}
We develop a framework for reconstructing an elliptic material coefficient from
finitely many boundary excitations, each producing a full Dirichlet trace. To
establish local stability, we first assume that the exact coefficient belongs
to a low-dimensional latent family
\(\mathcal M_{\mathrm{lat}}=\Phi(\mathcal Z)\subset\mathcal A_{\mathrm{ad}}\),
prescribed analytically or learned from coefficient samples, with
\(a^\dagger=\Phi(z^\dagger)\). We assume that \(\Phi\) is \(C^1\) as a map
into \(L^\infty(\Omega)\) and that \(D\Phi(z^\dagger)\) has full column rank.
If the admissible Neumann excitations separate
every nonzero perturbation in the latent tangent space at \(z^\dagger\), then
at most \(m\) can be selected so that the derivative of the latent finite-test
map has full column rank. Writing \(\varepsilon_L\) for the aggregate
\(H^{1/2}(\partial\Omega)\)-discrepancy between the selected traces, this yields
local injectivity and the Lipschitz stability estimate
\begin{equation}
\|\Phi(z)-\Phi(z^\dagger)\|_{L^\infty(\Omega)}\leq C\varepsilon_L
\quad\text{for }z\in\mathcal Z\text{ sufficiently close to }z^\dagger\label{eq:introduction-local-lipschitz-stability}
\end{equation}
Under suitable regularity and dimensional assumptions, we verify the required
separation property for the admissible experiment class using complex
geometrical optics solutions. Corollary~\ref{cor:constant-reference-two-dimensional-family}
gives an elementary two-dimensional example for affine families near a constant
reference coefficient.

Theorems~\ref{thm:finite-tests-from-linearized-injectivity}
and~\ref{thm:local-injectivity-from-finite-tests} establish the test count and
continuous local stability. Convergence of the coefficient sensitivities then
transfers this stability to sufficiently fine conforming finite element
discretizations, with a stability constant and neighborhood independent of the mesh size
(Theorem~\ref{thm:local-stability-discrete-finite-test-map}). To reduce the
cost of repeated discrete evaluations, we train a neural surrogate offline.
Uniform value accuracy suffices for a coefficient-error bound separating data
noise, finite element error, and surrogate error under residual comparison
(Corollary~\ref{cor:material-image-error-surrogate-residual-minimizers}).
Derivative accuracy additionally preserves the surrogate's own injectivity
and local stability
(Theorem~\ref{thm:stability-under-small-c1-perturbations}). These approximation
conditions are explicit hypotheses and are not certified by the training
procedure. Corollary~\ref{cor:l2-observation-stability} gives the same local
conclusions for the boundary \(L^2\)-norm used in the numerical objectives.

The stability theory is local around \(z^\dagger\) within the latent family.
For a true coefficient outside that family,
Remark~\ref{rem:representation-mismatch-error} adds a representation-error term
relative to a comparison coefficient in the same local stability set. Numerical experiments
use analytically parametrized elliptic inclusions and learned representations of
inclusions and crack-like materials. They illustrate latent-space reconstruction,
reduced online cost, and further improvement from optional FEM-based refinement
in the same latent space.

\paragraph{Previous Work.}
The mathematical study of the inverse conductivity problem began with Calder\'on
\cite{Calderon1980}. Its applications include electrical impedance tomography
(EIT) \cite{Cheney1999, Bor02, Hyv04} and related optical tomography problems
\cite{Arr99,Arr11,ArridgeSchotland2009}. Foundational work focused on uniqueness and 
identifiability, notably the results of Kohn and Vogelius 
\cite{KohnVogelius1984,KV85} and the breakthrough global 
uniqueness theorem of Sylvester and Uhlmann \cite{SU87}. 
A constructive reconstruction procedure in two dimensions was later
provided by Nachman \cite{Nachman1996}. Stability estimates of logarithmic type were obtained in \cite{Ale88}, with an extension to Lipschitz boundaries in \cite{Ale90}. The arguments in \cite{Man01} show that logarithmic stability cannot be improved.

Two main classes of computational methods have emerged.
Direct reconstruction techniques based on complex
geometrical optics (CGO) solutions and the D-bar method provide explicit 
inversion procedures in two dimensions \cite{Nachman1996,Knudsen2009}. 
Iterative approaches based on PDE-constrained
optimization have become the standard in practical EIT applications. 
These methods typically rely on variational formulations, such as the 
Kohn--Vogelius functional \cite{KV87, KM90} and Tikhonov-type 
regularization \cite{JM12, KJ15,HKR20,HK21}. Recently, machine-learning-based regularization has received increasing attention; see \cite{AMOS19} and references therein.

Finite element approximations and regularization methods for
EIT have been proposed in several works, including the variational and 
discrete approaches of \cite{Knowles1998,Know01, BGZ03, MMM04, GehreJinLu2014,Hinze2018}. 
The analysis of discretization effects and convergence 
of numerical schemes for the inverse conductivity problem has been 
studied in \cite{LR08 , GehreJinLu2014, Rond16, JXZ17, FR24}. Computational methods using neural networks to approximate unknowns have been considered in \cite{CJSZ23, CJQZ24}. For an overview of machine learning approaches to EIT, see \cite{denker2025deeplearningbasedreconstruction}.

Lipschitz stability can typically be recovered when the coefficient is restricted to a finite-dimensional space \cite{AV05, Sinc07,BCH11,Bour13,BFMRV14,AHGS17, RS22}. For piecewise constant coefficient classes, the stability constant can deteriorate exponentially as the number of subdomains increases \cite{Rond06}. More recently, Lipschitz stability has also been established under structural or geometric assumptions on the variation of the coefficient \cite{GHH26}.

Finite element approaches for inverse problems exploiting finite dimensionality have been considered in \cite{Harr19, Harr23, BH25}, with particular emphasis on monotonicity methods. Variational finite element methods for reconstructing piecewise constant coefficients on polygonal partitions were studied in \cite{BMPS18}. For inverse Robin problems with finite-dimensional coefficient spaces, finite element methods based on stability from unique continuation were analyzed in \cite{BCJZ25,BKO25}. In the context of unique continuation, collective data were used in \cite{BJL25} to isolate a low-dimensional manifold of solutions.

Despite these advances, the computational solution of the Calder\'on 
problem remains challenging, especially in the presence of partial or 
limited data. When the coefficient space is finite-dimensional, finitely many measurements of the Neumann-to-Dirichlet map are known to suffice to determine the coefficient \cite{AS19calderon, AS22independent, AS22, ASS23}. The Calder\'on-specific results in \cite{AS19calderon, AS22independent} also address reconstruction and, in the latter case, boundary inputs independent of the unknown and stability with respect to noise and model mismatch. A complementary recent result proves almost-sure identifiability of a \(d\)-dimensional analytic model class from \(2d+1\) random scalar measurements under full-data injectivity and suitable sampling assumptions \cite{APSS26}.

This finite-dimensional viewpoint motivates the use of machine learning to
construct low-dimensional latent representations of the coefficient space and
to perform the reconstruction in latent rather than ambient coordinates
\cite{AHSS24, BLLL25}. Such approaches require repeated evaluations of the
forward map and are therefore closely connected with the approximation of
high-dimensional parametric PDEs; see \cite{CD15}. In particular, we use a neural
surrogate operator, an approach related to the developments in
\cite{Dung2022,DNP23}.

The finite-dimensional stability principle is established in prior work,
including the manifold formulation in \cite{ASS23}. The latent finite element
approach of \cite{BLLL25} addresses unique continuation through a learned
representation of boundary data and the associated solution operator.
The surrogate reconstruction framework of
\cite{burman2026solvinginverseparametrizedproblems} treats parametrized inverse
problems, with quantitative photoacoustic tomography as a numerical example.
Here the unknown is the leading conductivity coefficient, and the observations
are Neumann-to-Dirichlet boundary responses. The main contribution is the
transfer of local finite-test stability through the continuous, finite element,
and surrogate maps by controlling coefficient sensitivities. The analysis also
separates the value accuracy needed for reconstruction-error control from the
derivative accuracy needed for stability of the surrogate itself.

\paragraph{Paper Outline.}
\label{par:outline-of-the-paper}
Section~\ref{sec:problem-formulation-finite-element-method-and-neural-surrogate-approximation}
introduces the forward problem, finite element method, and neural surrogate.
Section~\ref{sec:continuous-theory-for-latent-inversion-and-compressed-observations}
develops the continuous inverse theory, and
Section~\ref{sec:discrete-theory-for-the-latent-forward-map} its discrete and
surrogate counterparts. Section~\ref{sec:material-representations} discusses
material representations, Section~\ref{sec:numerical-experiments} presents the
numerical experiments, and Section~\ref{sec:conclusions} states the conclusions.
The appendix verifies the separation condition using CGO solutions.

%% file: sec-method-rev0.tex
\section{Mathematical and Computational Framework}
\label{sec:problem-formulation-finite-element-method-and-neural-surrogate-approximation}

\subsection{Continuous Forward and Inverse Problems}
\label{subsec:continuous-problem}

\paragraph{Domain and Material Coefficient.}
\label{par:domain-and-material-coefficient}
Let \(\Omega\subset \mathbb{R}^d\), \(d\in\{2,3\}\), be a bounded, connected Lipschitz domain with boundary \(\partial\Omega\). The material coefficient is denoted by \(a:\Omega\to\mathbb{R}\). We work with the admissible coefficient class
\begin{equation}
\mathcal{A}_{\mathrm{ad}}
=
\bigl\{
 a\in L^\infty(\Omega):0<a_{\min}\le a(x)\le a_{\max}\ \text{a.e. in }\Omega
\bigr\}
\label{eq:admissible-coefficient-class}
\end{equation}
where \(a_{\min}\) and \(a_{\max}\) are fixed positive constants. The coefficient \(a\) describes the material distribution and is the unknown in the inverse problem.
Whenever differentiability with respect to the coefficient is invoked, the
forward maps are understood as restrictions of maps defined on an open
\(L^\infty(\Omega)\)-neighborhood of \(\mathcal A_{\mathrm{ad}}\) consisting of
uniformly positive coefficients.

\paragraph{Boundary Value Problem.}
\label{par:neumann-boundary-value-problem}
We denote the space of compatible Neumann data by
\begin{equation}
H^{-1/2}_{\diamond}(\partial\Omega)
=
\bigl\{
 g\in H^{-1/2}(\partial\Omega):\langle g,1\rangle_{\partial\Omega}=0
\bigr\}\label{eq:h-minus-one-half-diamond}
\end{equation}
and the boundary-mean-zero trace space by
\begin{equation}
H^{1/2}_{\diamond}(\partial\Omega)
=
\bigl\{
v\in H^{1/2}(\partial\Omega):(v,1)_{\partial\Omega}=0
\bigr\}\label{eq:h-plus-one-half-diamond}
\end{equation}
Here and below, \(\langle\cdot,\cdot\rangle_{\partial\Omega}\) denotes the
duality pairing between \(H^{-1/2}(\partial\Omega)\) and
\(H^{1/2}(\partial\Omega)\).
Given a Neumann boundary datum \(g_N\in H^{-1/2}_{\diamond}(\partial\Omega)\), we consider the boundary value problem
\begin{align}
-\nabla\cdot(a\nabla u) &= 0 \quad \text{in }\Omega
\label{eq:strong-neumann-problem-pde}\\
a\nabla u\cdot n &= g_N \quad \text{on }\partial\Omega
\label{eq:strong-neumann-problem-bc}\\
(\gamma u,1)_{\partial\Omega} &= 0\label{eq:strong-neumann-problem-normalization}
\end{align}
where \(n\) denotes the unit outward normal on \(\partial\Omega\), \(\gamma\)
denotes the trace operator, and the boundary condition is understood in the
weak conormal sense through the weak formulation below. The compatibility condition
\(\langle g_N,1\rangle_{\partial\Omega}=0\) in
\eqref{eq:h-minus-one-half-diamond}, together with the normalization
\eqref{eq:strong-neumann-problem-normalization}, ensures that the solution \(u\)
is uniquely determined by the material coefficient \(a\) and the Neumann datum
\(g_N\). The boundary-mean normalization selects one representative of the
solution modulo constants, and all continuous and discrete NtD operators below
use this representative.

\paragraph{Weak Formulation.}
\label{par:weak-formulation-continuous-problem}
We introduce the boundary-mean-zero space
\begin{equation}
V=\bigl\{v\in H^1(\Omega):(\gamma v,1)_{\partial\Omega}=0\bigr\}\label{eq:mean-zero-space}
\end{equation}
For a fixed coefficient \(a\in \mathcal{A}_{\mathrm{ad}}\) and Neumann datum \(g_N\in H^{-1/2}_{\diamond}(\partial\Omega)\), we define the bilinear form \(A_a(\cdot,\cdot):V\times V\to\mathbb{R}\) and the linear functional \(F_{g_N}(\cdot):V\to\mathbb{R}\) by
\begin{align}
A_a(u,v) &= (a\nabla u,\nabla v)_{\Omega}
\label{eq:bilinear-form-definition}\\
F_{g_N}(v) &= \langle g_N,\gamma v\rangle_{\partial\Omega}
\label{eq:linear-form-definition}
\end{align}
The weak formulation of \eqref{eq:strong-neumann-problem-pde}--\eqref{eq:strong-neumann-problem-normalization} reads: given \(a \in \mathcal{A}_{\mathrm{ad}}\) and \(g_N \in H^{-1/2}_{\diamond}(\partial\Omega)\), find \(u(a;g_N)  \in V\) such that
\begin{equation}
A_a(u,v)=F_{g_N}(v)\quad \forall v\in V
\label{eq:weak-formulation-continuous}
\end{equation}

\paragraph{Finite Boundary Experiments.}
\label{par:finite-boundary-experiments}
We fix a finite family of admissible Neumann data,
\begin{equation}
\mathbf g_N
=
\bigl(g_N^1,\dots,g_N^L\bigr)
\in
\bigl(H^{-1/2}_{\diamond}(\partial\Omega)\bigr)^L
\label{eq:finite-family-neumann-data}
\end{equation}
These excitations specify the boundary experiments used in the inverse problem.
Theorems~\ref{thm:finite-tests-from-linearized-injectivity}
and~\ref{thm:local-injectivity-from-finite-tests} give sufficient local
conditions under which such a family can be selected so that the corresponding
latent forward map is locally stable near the reference point.

\paragraph{Boundary Solution Operator.}
\label{par:boundary-solution-operator}
The solution trace defines the Neumann-to-Dirichlet operator
\begin{equation}
\mathcal{N}_a:H^{-1/2}_{\diamond}(\partial\Omega)\to H^{1/2}_{\diamond}(\partial\Omega),
\qquad
\mathcal{N}_a g_N=u(a;g_N)|_{\partial\Omega}
\label{eq:boundary-solution-operator}
\end{equation}
For the prescribed boundary experiments, we collect the corresponding boundary responses in
\begin{equation}
N_L(a)=\bigl(\mathcal{N}_a g_N^1,\dots,\mathcal{N}_a g_N^L\bigr) \in Y_L
\label{eq:finite-test-boundary-response-map}
\end{equation}
where the finite-test observation space
\(Y_L=\bigl(H^{1/2}(\partial\Omega)\bigr)^L\) contains full boundary
traces and is therefore infinite-dimensional. We equip it with the product norm inherited from
\(H^{1/2}(\partial\Omega)\),
\(\|y\|_{Y_L}^2=\sum_{\ell=1}^L\|y^\ell\|_{H^{1/2}(\partial\Omega)}^2\) for
\(y=(y^1,\dots,y^L)\). The PDE responses have zero boundary mean in each
component, while the observation space also accommodates surrogate traces
with nonzero means. Figure~\ref{fig:NtD-map} illustrates the
Neumann-to-Dirichlet map.

\begin{figure}
    \centering
\includegraphics[width=0.8\textwidth]{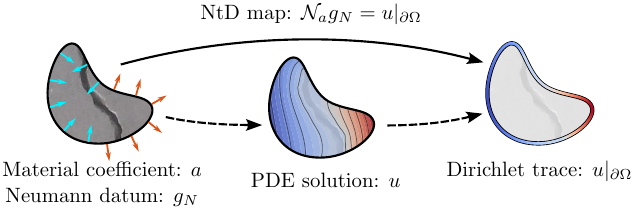}
\caption{Schematic of the Neumann-to-Dirichlet map for a fixed material coefficient \(a\). Given a compatible Neumann datum \(g_N\), the solution \(u=u(a;g_N)\) of the elliptic problem determines the Dirichlet trace \(\mathcal N_a g_N=u|_{\partial\Omega}\).}
\label{fig:NtD-map}
\end{figure}

\paragraph{Inverse Problem.}
\label{par:inverse-problem}
Let \(a^\dagger\in\mathcal A_{\mathrm{ad}}\) denote the unknown exact coefficient.
For the prescribed boundary experiments, the corresponding measured
Dirichlet data are
\begin{equation}
g_D^\ell
=
\mathcal N_{a^\dagger}g_N^\ell,
\qquad
\ell=1,\dots,L
\label{eq:finite-measurements-reference-coefficient}
\end{equation}
We collect these measurements in the data vector
\begin{equation}
\data
=
\bigl(
g_D^1,
\dots,
g_D^L
\bigr)
\in
Y_L
\label{eq:finite-measurement-vector}
\end{equation}
which satisfies \(\data=N_L(a^\dagger)\).
The inverse problem is to determine \(a^\dagger\), or an approximation to it, from the data \(\data\).
A standard formulation is the data-fitting problem: find \(a_{\mathcal A}^\star\in\mathcal{A}_{\mathrm{ad}}\) such that
\begin{equation}
a_{\mathcal A}^\star
\in
\operatorname*{arg\,min}_{a\in\mathcal{A}_{\mathrm{ad}}}
\frac12\|N_L(a)-\data\|_{Y_L}^2 + \mathcal{R}_{\mathcal A}(a)
\label{eq:continuous-inverse-problem-data-fitting}
\end{equation}
where \(\mathcal{R}_{\mathcal A}(a)\) is a regularization term. The optimization is posed over the full admissible class \(\mathcal A_{\mathrm{ad}}\). We next restrict the search to a low-dimensional latent coefficient family to incorporate prior structural information and reduce the number of unknowns.

\paragraph{Latent Representation.}
\label{par:latent-representation}
We describe the material coefficient through a latent variable
\(z\in\mathcal{Z}\subset\mathbb{R}^m\), where \(m\) is chosen to be small,
and a representation map
\begin{equation}
\Phi:\mathcal{Z}\to \mathcal{A}_{\mathrm{ad}}
\label{eq:representation-map}
\end{equation}
so that the material coefficients are given by
\begin{equation}
a=\Phi(z)
\label{eq:coefficient-from-representation-map}
\end{equation}
The image of \(\Phi\) defines the latent coefficient family
\begin{equation}
\mathcal{M}_{\mathrm{lat}}
=
\Phi(\mathcal{Z})
\subset
\mathcal{A}_{\mathrm{ad}}
\label{eq:latent-material-manifold}
\end{equation}
The representation map \(\Phi\) may be specified analytically or learned offline from material samples. During inversion, \(\Phi\) is fixed and known. Figure~\ref{fig:coeff-map} illustrates this map.
Whenever differentiability is invoked, \(C^k(\mathcal Z;X)\) means that the
map under consideration admits a \(C^k\)-extension to an open neighborhood of
\(\mathcal Z\), and reference latent variables in the local analysis are
assumed to be interior points of \(\mathcal Z\).

\begin{figure}
    \centering
    \includegraphics[width=0.7\textwidth]{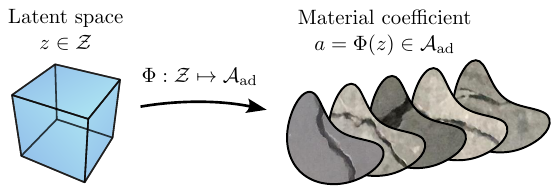}
    \caption{Schematic of the latent representation. The representation map \(\Phi:\mathcal Z\to\mathcal A_{\mathrm{ad}}\) maps a latent coordinate \(z\) to the material coefficient \(a=\Phi(z)\). Its image \(\mathcal M_{\mathrm{lat}}=\Phi(\mathcal Z)\) is the low-dimensional coefficient family used in the reduced inverse problem.}
    \label{fig:coeff-map}
    \end{figure}

\paragraph{Reduced Inverse Problem.}
\label{par:reduced-inverse-problem}
Restricting the admissible coefficients to
\(\mathcal{M}_{\mathrm{lat}}\subset\mathcal{A}_{\mathrm{ad}}\) yields the mapping
\begin{equation}
N_L^\Phi(z)
=
N_L(\Phi(z))
=
\bigl(
\mathcal{N}_{\Phi(z)}g_N^1,
\dots,
\mathcal{N}_{\Phi(z)}g_N^L
\bigr)
\in
Y_L
\label{eq:continuous-latent-finite-test-map}
\end{equation}
Under the exact latent-model assumption, there exists
\(z^\dagger\in\mathcal Z\) such that \(a^\dagger=\Phi(z^\dagger)\), and hence
\(\data=N_L^\Phi(z^\dagger)\). If this assumption does not hold, the reduced
problem instead seeks a best fit within \(\mathcal M_{\mathrm{lat}}\).
The reduced inverse problem is to determine a latent variable
\(z^\star\in\mathcal{Z}\) such that
\begin{equation}
z^\star
\in
\operatorname*{arg\,min}_{z\in\mathcal{Z}} 
\frac12
\|N_L^\Phi(z)-\data\|_{Y_L}^2 + \mathcal{R}_{\mathcal Z}(z)
\label{eq:continuous-reduced-inverse-problem}
\end{equation}
where \(\mathcal{R}_{\mathcal Z}(z)\) is a regularization term. The corresponding approximation of the material coefficient is then given by
\(a_{\mathcal Z}^\star=\Phi(z^\star)\).
The stability analysis below concerns the forward maps and is independent of a
particular choice of regularization.

\subsection{Finite Element Method}
\label{subsec:finite-element-method}

Let \(\mathcal{T}_h\) be a conforming partition of \(\Omega\), and let \(V_h\subset H^1(\Omega)\) be a standard conforming finite element space associated with \(\mathcal{T}_h\). We define the discrete boundary-mean-zero space by
\begin{equation}
V_h^0=\bigl\{v_h\in V_h:(\gamma v_h,1)_{\partial\Omega}=0\bigr\}\label{eq:discrete-mean-zero-space}
\end{equation}
For \(g_N\in H^{-1/2}_{\diamond}(\partial\Omega)\), the finite element
approximation of the Neumann problem reads: find \(u_h(a;g_N)\in V_h^0\) such
that
\begin{equation}
A_a\bigl(u_h(a;g_N),v_h\bigr)=\langle g_N,\gamma v_h\rangle_{\partial\Omega}
\quad \forall v_h\in V_h^0
\label{eq:discrete-weak-formulation}
\end{equation}
For the fixed boundary experiments, we set
\(u_h^\ell(a)=u_h(a;g_N^\ell)\), \(\ell=1,\dots,L\).
The corresponding discrete boundary operator is defined by
\begin{equation}
\mathcal{N}_{a,h}:H^{-1/2}_{\diamond}(\partial\Omega)\to W_h,
\qquad
\mathcal{N}_{a,h}g_N=u_h(a;g_N)|_{\partial\Omega}
\label{eq:discrete-boundary-solution-operator}
\end{equation}
where
\begin{equation}
W_h=\gamma(V_h)\subset H^{1/2}(\partial\Omega)\label{eq:discrete-trace-space-method}
\end{equation}
The discrete NtD responses belong to the subspace \(\gamma(V_h^0)\).
The discrete finite observation space is
\begin{equation}
Y_{L,h}=(W_h)^L\subset Y_L
\label{eq:discrete-finite-observation-space}
\end{equation}
and we equip \(Y_{L,h}\) with the norm inherited from \(Y_L\).
For the selected family of tests \(\{g_N^\ell\}_{\ell=1}^L\), we write
\begin{equation}
N_{L,h}(a)=\bigl(\mathcal{N}_{a,h}g_N^1,\dots,\mathcal{N}_{a,h}g_N^L\bigr)
\in Y_{L,h}
\label{eq:finite-test-discrete-boundary-response-map}
\end{equation}
The discrete counterpart of the exact data vector is
\begin{equation}
\datah=N_{L,h}(a^\dagger)\in Y_{L,h}\label{eq:discrete-data-vector}
\end{equation}
This is an idealized model-generated datum. Unlike \(\data\), it is not an
independently measured quantity because its definition uses \(a^\dagger\). We
use it as the noiseless datum in the synthetic experiments. For measured data,
one would instead compare with \(\data\) in \(Y_L\), or first specify a
projection or discretization of the measurements into \(Y_{L,h}\).

\subsection{Neural Surrogate Approximation}
\label{subsec:neural-surrogate-approximation}

\paragraph{Latent-to-Observation Workflow.}
\label{par:latent-to-observation-workflow}
The reconstruction workflow maps latent variables to boundary responses and
then to a data-misfit objective.
At the continuous level one uses the Neumann-to-Dirichlet map
\(\mathcal{N}_{\Phi(z)}\), while at the discrete level this map is replaced by
\(\mathcal{N}_{\Phi(z),h}\). The neural surrogate introduced below replaces the
map \(N_{L,h}^\Phi\) by its learned approximation \(N_{L,h,M}^\Phi\), while
keeping the inverse variable \(z\) in the same latent space. These relations are
illustrated in Figure~\ref{fig:NN-map}.

\begin{figure}
    \centering
\includegraphics[width=0.9\textwidth]{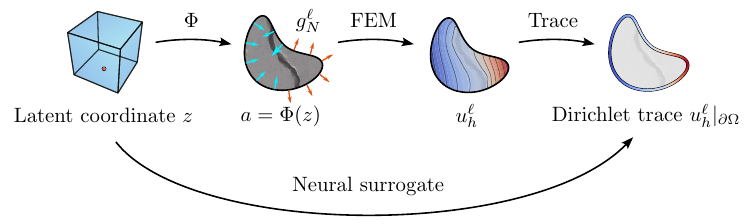}
\caption{Schematic of the discrete finite-test map and its neural surrogate. The
upper path shows the response for a generic boundary experiment \(\ell\), mapping
the latent coordinate \(z\) and prescribed Neumann datum \(g_N^\ell\) to the
Dirichlet trace \(u_h^\ell|_{\partial\Omega}\). The neural surrogate approximates
the vector-valued map \(N_{L,h}^\Phi\), which collects these responses for all
\(\ell=1,\dots,L\).}
\label{fig:NN-map}
\end{figure}

\paragraph{Latent-Space Surrogates.}
\label{par:training-data-and-latent-space-surrogates}
For the prescribed Neumann data \(g_N^1,\dots,g_N^L\), we train a surrogate network to approximate the discrete finite-test response map
\begin{equation}
N_{L,h}^\Phi(z)
=
N_{L,h}(\Phi(z))
=
\bigl(\mathcal{N}_{\Phi(z),h}g_N^1,\dots,\mathcal{N}_{\Phi(z),h}g_N^L\bigr)
\in Y_{L,h}
\label{eq:training-target-finite-test-map}
\end{equation}
where \(Y_{L,h}\) is defined in \eqref{eq:discrete-finite-observation-space}. For latent training points \(z_i\in\mathcal{Z}\), the corresponding training targets are \(N_{L,h}^\Phi(z_i)\). Since \(Y_{L,h}\) is finite-dimensional, we fix a coordinate isomorphism
\begin{equation}
\Lambda_h:Y_{L,h}\to \mathbb{R}^{P_h},
\qquad P_h=\dim Y_{L,h}
\label{eq:method-coordinate-isomorphism-discrete-output}
\end{equation}
and define the coordinate representation of this map by
\begin{equation}
\widehat{N}_{L,h}^\Phi=\Lambda_h\circ N_{L,h}^\Phi:\mathcal{Z}\to \mathbb{R}^{P_h}
\label{eq:method-coordinate-representation-discrete-forward-map}
\end{equation}
The neural-network surrogate is trained directly on the latent variable and approximates the coordinate map \(\widehat{N}_{L,h}^\Phi\).

\paragraph{Two-Layer Ansatz.}
\label{par:two-layer-ansatz}

For the \(C^1\)-approximation hypothesis used in the stability analysis, we consider vector-valued two-layer neural networks. For a fixed activation function \(\sigma:\mathbb{R}\to \mathbb{R}\) and width \(M\), the surrogate network \(\widehat{N}_{L,h,M}^\Phi:\mathcal{Z}\to \mathbb{R}^{P_h}\) is given by
\begin{equation}
\widehat{N}_{L,h,M}^\Phi(z)=\sum_{r=1}^M \boldsymbol{\beta}_{r,h}\,\sigma({\boldsymbol w}_{r,h}\cdot z+b_{r,h})
\label{eq:two-layer-network-discrete-forward-map}
\end{equation}
with parameters
\begin{equation}
\boldsymbol{\beta}_{r,h}\in \mathbb{R}^{P_h},
\quad
{\boldsymbol w}_{r,h}\in \mathbb{R}^m,
\quad
b_{r,h}\in \mathbb{R} \label{eq:network-parameters}
\end{equation}
The associated surrogate map is then defined by
\begin{equation}
 N_{L,h,M}^\Phi=\Lambda_h^{-1}\circ \widehat{N}_{L,h,M}^\Phi :\mathcal{Z}\to Y_{L,h}
\label{eq:network-surrogate-discrete-forward-map}
\end{equation}
and is used to approximate the discrete finite-test map \(N_{L,h}^\Phi\).

\paragraph{Surrogate-Based Inverse Problem.}
\label{par:surrogate-based-inverse-problem}
For the idealized discrete data \(\datah\) defined in
\eqref{eq:discrete-data-vector}, replacing the discrete finite-test map by its
neural surrogate gives the computational inverse problem: find
\(z_{h,M}^\star\in\mathcal Z\) such that
\begin{equation}
z_{h,M}^\star
\in
\operatorname*{arg\,min}_{z\in\mathcal Z}
\frac12\|N_{L,h,M}^\Phi(z)-\datah\|_{Y_L}^2+\mathcal R_{\mathcal Z}(z)\label{eq:surrogate-based-inverse-problem}
\end{equation}
The corresponding reconstructed coefficient is
\(a_{h,M}^\star=\Phi(z_{h,M}^\star)\). Section~\ref{sec:numerical-experiments}
uses the unregularized version of this objective with a discrete
\(L^2(\partial\Omega)^L\)-misfit.

%% file: sec-analysis-rev0.tex
\section[Continuous Theory for Latent Inversion with Finite Tests]{Continuous Theory for Latent\\ Inversion with Finite Tests}
\label{sec:continuous-theory-for-latent-inversion-and-compressed-observations}

\paragraph{Overview.}

This section establishes local stability for the exact latent forward map before
introducing the finite element discretization and neural-network surrogate. We
first differentiate the Neumann-to-Dirichlet map with respect to the coefficient
and compose the resulting sensitivity with the latent map \(\Phi\). At the
reference point \(z^\dagger\), a separation condition for products of gradients
then yields injectivity of the linearized map on the finite-dimensional latent
tangent space \(T^\Phi_{z^\dagger}=\operatorname{range}D\Phi(z^\dagger)\). A
dimension argument selects at most \(m\) Neumann tests whose cumulative Gram
matrix is positive definite, which gives local Lipschitz stability in both the
latent and coefficient variables. Appendix~\ref{sec:cgo-verification-of-the-separation-condition}
provides a CGO-based sufficient condition for separation, while
Section~\ref{sec:discrete-theory-for-the-latent-forward-map} transfers the
finite-test stability to the discrete and surrogate maps.

\paragraph{Standing Assumptions.}

Let \(z^\dagger\in\mathcal Z\) denote the latent variable corresponding to the
target coefficient \(a^\dagger=\Phi(z^\dagger)\). The analysis below is local
about \(z^\dagger\). We use the following assumptions:
\begin{itemize}
\item \proofparagraph{A1:} The latent map \(\Phi\) satisfies
\begin{equation}
\Phi\in C^1\bigl(\mathcal{Z};L^\infty(\Omega)\bigr),
\qquad
\Phi(\mathcal{Z})\subset \mathcal{A}_{\mathrm{ad}}
\label{eq:assumption-latent-map}
\end{equation}

\item \proofparagraph{A2:} The derivative of the latent map has full column rank at \(z^\dagger\),
\begin{equation}
\operatorname{rank} D\Phi(z^\dagger)=m
\label{eq:assumption-dg-full-rank}
\end{equation}

\end{itemize}
We also fix a nonempty class
\(\mathcal{G}\subset H^{-1/2}_{\diamond}(\partial\Omega)\) of admissible
Neumann tests. The richness of this class is imposed separately through the
separation condition in Theorem~\ref{thm:linearized-injectivity-on-the-tangent-space}.

\subsection{Differentiability of the Continuous NtD Map}
\label{subsec:differentiability-of-the-continuous-neumann-to-dirichlet-map}

For each \(g_N\in H^{-1/2}_{\diamond}(\partial\Omega)\), define the solution
map \(S_{g_N}:\mathcal{A}_{\mathrm{ad}}\to V\) by
\begin{equation}
S_{g_N}(a)=u(a;g_N)\label{eq:solution-map-definition-fixed-neumann-data}
\end{equation}
\begin{prop}[Differentiability of the continuous boundary map]
\label{prop:differentiability-of-the-boundary-map}
The solution map \(S_{g_N}\) is the restriction to
\(\mathcal A_{\mathrm{ad}}\) of a continuously Fr\'echet differentiable map
defined on an open \(L^\infty(\Omega)\)-neighborhood of
\(\mathcal A_{\mathrm{ad}}\) consisting of uniformly positive coefficients.
For \(a\in\mathcal A_{\mathrm{ad}}\) and
\(\delta a\in L^\infty(\Omega)\), its derivative
\begin{equation}
w(a,\delta a;g_N)=DS_{g_N}(a)[\delta a]\label{eq:directional-derivative-solution-map}
\end{equation}
is the unique solution of
\begin{equation}
A_a\bigl(w(a,\delta a;g_N),v\bigr)
=
-(\delta a \nabla u(a;g_N),\nabla v)_{\Omega}
\quad \forall v\in V\label{eq:sensitivity-problem-continuous}
\end{equation}
Moreover, the boundary map
\begin{equation}
\mathcal{B}_{g_N}(a)=\mathcal{N}_a g_N\label{eq:boundary-map-fixed-neumann-data}
\end{equation}
is continuously Fr\'echet differentiable as a map from
\(\mathcal{A}_{\mathrm{ad}}\) into
\(H^{1/2}_{\diamond}(\partial\Omega)\), with derivative
\begin{equation}
D\mathcal{B}_{g_N}(a)[\delta a]
=
w(a,\delta a;g_N)|_{\partial\Omega}\label{eq:derivative-of-boundary-map}
\end{equation}
\end{prop}

\begin{proof}
Well-posedness follows from the Lax--Milgram lemma and the boundary-mean
Poincar\'e inequality. Using the coefficient-neighborhood
convention in Section~\ref{subsec:continuous-problem}, consider a perturbation
\(b\in L^\infty(\Omega)\) such that \(a+b\) remains in the open uniformly
positive coefficient neighborhood. Set \(u=S_{g_N}(a)\) and
\(u_b=S_{g_N}(a+b)\). Subtracting the two state equations and applying the
uniform stability estimate gives
\begin{equation}
\|u_b-u\|_{H^1(\Omega)}
\lesssim
\|b\|_{L^\infty(\Omega)}\|g_N\|_{H^{-1/2}(\partial\Omega)}
\label{eq:continuous-solution-perturbation-estimate}
\end{equation}
Let \(w=w(a,b;g_N)\) solve \eqref{eq:sensitivity-problem-continuous} and set
\(r_b=u_b-u-w\). Another subtraction gives
\begin{equation}
A_a(r_b,v)
=
-(b\nabla(u_b-u),\nabla v)_\Omega
\quad \forall v\in V
\label{eq:continuous-solution-derivative-remainder-equation}
\end{equation}
and therefore
\begin{equation}
\|r_b\|_{H^1(\Omega)}
\lesssim
\|b\|_{L^\infty(\Omega)}^2\|g_N\|_{H^{-1/2}(\partial\Omega)}
\label{eq:continuous-solution-derivative-remainder-estimate}
\end{equation}
This proves the asserted Fr\'echet derivative. To verify its continuity, let
\(a_1\) and \(a_2\) belong to the same uniformly positive coefficient
neighborhood, and set \(u_i=S_{g_N}(a_i)\) and
\(w_i=DS_{g_N}(a_i)[\delta a]\) for \(i=1,2\). Subtracting the two
sensitivity equations yields
\begin{equation}
A_{a_1}(w_1-w_2,v)
=
-(\delta a\nabla(u_1-u_2),\nabla v)_\Omega
-((a_1-a_2)\nabla w_2,\nabla v)_\Omega
\quad \forall v\in V
\label{eq:continuous-sensitivity-perturbation-equation}
\end{equation}
The uniform stability estimate, together with
\eqref{eq:continuous-solution-perturbation-estimate}, therefore gives
\begin{equation}
\|DS_{g_N}(a_1)-DS_{g_N}(a_2)\|_{\mathcal L(L^\infty(\Omega),V)}
\lesssim
\|a_1-a_2\|_{L^\infty(\Omega)}
\|g_N\|_{H^{-1/2}(\partial\Omega)}
\label{eq:continuous-solution-derivative-continuity}
\end{equation}
Thus \(S_{g_N}\) is continuously Fr\'echet differentiable. Finally, the trace
map \(V\to H^{1/2}_{\diamond}(\partial\Omega)\) is bounded, which proves the
claims for \(\mathcal B_{g_N}\).
\end{proof}

\subsection{Differentiability of the Latent Forward Map}
\label{subsec:latent-reduced-forward-map}

\begin{prop}[Differentiability of the latent reduced map]
\label{prop:differentiability-of-the-latent-reduced-map}
Under Assumption (A1), the map \(N_L^\Phi:\mathcal Z\to Y_L\) defined in \eqref{eq:continuous-latent-finite-test-map} is of class \(C^1\). For \(\eta\in \mathbb{R}^m\), its derivative is given by
\begin{equation}
DN_L^\Phi(z)\eta
=
\bigl(
w^1(z;\eta)|_{\partial\Omega},\dots,w^L(z;\eta)|_{\partial\Omega}
\bigr)
\label{eq:derivative-of-the-latent-reduced-map}
\end{equation}
where, for each \(\ell=1,\dots,L\), \(w^\ell(z;\eta)\in V\) solves
\begin{equation}
A_{\Phi(z)}(w^\ell(z;\eta),v)
=
-
\bigl(
D\Phi(z)\eta\,\nabla u^\ell(z),
\nabla v
\bigr)_{\Omega}
\quad \forall v\in V
\label{eq:sensitivity-problem-latent-direction}
\end{equation}
and
\begin{equation}
u^\ell(z)=u\bigl(\Phi(z);g_N^\ell\bigr)
\label{eq:background-state-latent-direction}
\end{equation}
\end{prop}

\begin{proof}
Proposition~\ref{prop:differentiability-of-the-boundary-map}, the chain rule,
and Assumption (A1) show that each component
\(z\mapsto\mathcal B_{g_N^\ell}(\Phi(z))\) is of class \(C^1\). Collecting the
components gives the stated regularity of \(N_L^\Phi\), and applying the chain
rule componentwise gives \eqref{eq:derivative-of-the-latent-reduced-map}--\eqref{eq:background-state-latent-direction}.
\end{proof}

\subsection{Linearized Injectivity and Finite Neumann Tests}
\label{subsec:linearized-injectivity-and-finite-neumann-tests}

Assumption (A2) makes \(D\Phi(z^\dagger)\) injective. We call
\begin{equation}
T^\Phi_{z^\dagger}
:=
\operatorname{range}D\Phi(z^\dagger)\label{eq:latent-tangent-space-at-reference-point}
\end{equation}
the latent tangent space at \(z^\dagger\). It is associated with the local
parametrization by \(\Phi\); no globally embedded-manifold structure is assumed
for the full image \(\mathcal M_{\mathrm{lat}}\).

\paragraph{Linearized Responses.}
\label{par:linearized-responses}
Let \(a^\dagger=\Phi(z^\dagger)\) be the reference coefficient fixed above.
We define the family of admissible background states by
\begin{equation}
\mathcal{U}(a^\dagger)
=
\bigl\{
u(a^\dagger;g_N):g_N\in \mathcal{G}
\bigr\}
\label{eq:background-state-family}
\end{equation}
For \(\delta a\in T^\Phi_{z^\dagger}\) and \(g_N\in \mathcal{G}\), let \(u=u(a^\dagger;g_N)\in V\) denote the corresponding background solution and let \(w=w(\delta a;g_N)\in V\) solve
\begin{equation}
A_{a^\dagger}(w,v) 
=
-
(\delta a \nabla u,\nabla v)_{\Omega}
\quad \forall v\in V
\label{eq:sensitivity-problem-reference-coefficient}
\end{equation}
The linearized Neumann-to-Dirichlet map at \(a^\dagger\) is then defined by
\begin{equation}
D\mathcal{B}_{g_N}(a^\dagger)[\delta a]
=
w(\delta a;g_N)|_{\partial\Omega}
\label{eq:linearized-neumann-to-dirichlet-map}
\end{equation}

\begin{thm}[Linearized injectivity on the tangent space]
\label{thm:linearized-injectivity-on-the-tangent-space}
Assume that the family of products of gradients of background solutions separates the tangent space in the sense that
\begin{equation}
\delta a\in T^\Phi_{z^\dagger}
\quad\text{and}\quad
(\delta a \nabla u,\nabla v)_{\Omega}=0
\quad \forall u,v\in \mathcal{U}(a^\dagger)
\quad\Longrightarrow\quad
\delta a=0
\label{eq:separation-condition-tangent-space}
\end{equation}

Then the linearized Neumann-to-Dirichlet map is injective on \(T^\Phi_{z^\dagger}\), that is,
\begin{equation}
\delta a\in T^\Phi_{z^\dagger}
\quad\text{and}\quad
D\mathcal{B}_{g_N}(a^\dagger)[\delta a]=0
\quad \forall g_N\in \mathcal{G}
\quad\Longrightarrow\quad
\delta a=0
\label{eq:linearized-injectivity-tangent-space}
\end{equation}
\end{thm}

\begin{proof}
Let \(\delta a\in T^\Phi_{z^\dagger}\) have vanishing linearized data for
every \(g_N\in\mathcal G\). For \(u=u(a^\dagger;g_N)\), the associated
sensitivity \(w\) then satisfies \(\gamma w=0\). If
\(v=u(a^\dagger;h_N)\) with \(h_N\in\mathcal G\), symmetry and the weak
state and sensitivity equations give
\begin{equation}
-(\delta a\nabla u,\nabla v)_\Omega
=A_{a^\dagger}(w,v)
=A_{a^\dagger}(v,w)
=\langle h_N,\gamma w\rangle_{\partial\Omega}
=0
\label{eq:orthogonality-identity-tangent-space}
\end{equation}
The separation condition \eqref{eq:separation-condition-tangent-space}
therefore gives \(\delta a=0\).
\end{proof}

A sufficient full-data criterion for \eqref{eq:separation-condition-tangent-space}
is proved in Appendix~\ref{sec:cgo-verification-of-the-separation-condition}
using CGO solutions and the arguments from \cite{SU87}. It applies to smooth
strictly positive reference coefficients in dimension \(d\ge 3\) when
\(\mathcal G=H^{-1/2}_\diamond(\partial\Omega)\), or when \(\mathcal G\)
contains the real and imaginary parts of the conormal traces of the
CGO-generated conductivity solutions.
For a constant reference coefficient,
Corollary~\ref{cor:constant-reference-two-dimensional-family} verifies separation
in two dimensions using harmonic exponentials.

\paragraph{Selection of Finitely Many Neumann Tests.}
\label{par:selection-of-finitely-many-neumann-tests}
Let \(\{e_j\}_{j=1}^m\) denote the canonical basis of \(\mathbb{R}^m\). For each \(g_N\in \mathcal{G}\), we define the linearized responses by

\begin{equation}
    q_{j}(g_N)=\partial_{e_j}\bigl(\mathcal{N}_{\Phi(z)}g_N\bigr)\big|_{z=z^{\dagger}}\in H^{1/2}_{\diamond}(\partial\Omega)
    \label{eq:linearized-responses}
\end{equation}
Define the tangent basis directions by
\begin{equation}
\delta a_j = D\Phi(z^\dagger)e_j,
\quad j=1,\dots,m
\label{eq:tangent-basis-directions-continuous-theory}
\end{equation}
The chain rule then gives
\begin{equation}
q_j(g_N)
=
D\mathcal{B}_{g_N}(a^\dagger)[\delta a_j]
\in H^{1/2}_{\diamond}(\partial\Omega),
\quad j=1,\dots,m
\label{eq:linearized-responses-tangent-basis-continuous}
\end{equation}
Next, we define the associated Gram matrix
\begin{equation}
M(g_N)_{ij}
=
\bigl(
q_i(g_N),
q_j(g_N)
\bigr)_{H^{1/2}(\partial\Omega)},
\quad i,j=1,\dots,m
\label{eq:single-test-gram-matrix-continuous}
\end{equation}
For a finite family
\(\mathbf g_N=(g_N^1,\dots,g_N^L)\in \mathcal{G}^L\), we set
\begin{equation}
M_L
=
\sum_{\ell=1}^L M(g_N^\ell)
\label{eq:cumulative-gram-matrix-continuous}
\end{equation}
For \(\eta=(\eta_1,\dots,\eta_m)^\top\in \mathbb{R}^m\), let
\begin{equation}
\delta a_\eta
=
\sum_{j=1}^m \eta_j \delta a_j
\label{eq:latent-tangent-direction-continuous}
\end{equation}
Then
\begin{equation}
\eta^\top M_L \eta
=
\sum_{\ell=1}^L
\bigl\|
D\mathcal{B}_{g_N^\ell}(a^\dagger)[\delta a_\eta]
\bigr\|_{H^{1/2}(\partial\Omega)}^2
=
\bigl\|DN_L^\Phi(z^\dagger)\eta\bigr\|_{Y_L}^2
\label{eq:gram-identity-continuous}
\end{equation}

\begin{thm}[Finite tests from linearized injectivity]
\label{thm:finite-tests-from-linearized-injectivity}
Assume (A1), (A2), and \eqref{eq:linearized-injectivity-tangent-space}. Then
there exist \(1\le L\le m\)
tests \(g_N^1,\dots,g_N^L\in \mathcal{G}\) such that the corresponding
cumulative Gram matrix \(M_L\) is positive definite. Equivalently,
\(DN_L^\Phi(z^\dagger)\) has full column rank.
\end{thm}

\begin{proof}
For each \(g_N\in\mathcal G\), define the linear map
\begin{equation}
A_{g_N}\eta=D\mathcal B_{g_N}(a^\dagger)[D\Phi(z^\dagger)\eta]
\quad \eta\in\mathbb R^m
\label{eq:linearized-test-map}
\end{equation}
Assumption (A2) and \eqref{eq:linearized-injectivity-tangent-space} give
\(\bigcap_{g_N\in\mathcal G}\ker A_{g_N}=\{0\}\).
Set \(W_0=\mathbb R^m\). Whenever \(W_{\ell-1}\ne\{0\}\), choose a nonzero
\(\eta\in W_{\ell-1}\) and a test \(g_N^\ell\in\mathcal G\) for which
\(A_{g_N^\ell}\eta\ne0\), and set
\begin{equation}
W_\ell=W_{\ell-1}\cap\ker A_{g_N^\ell}
\quad\text{so that}\quad
\dim W_\ell<\dim W_{\ell-1}
\label{eq:finite-test-kernel-reduction}
\end{equation}
Thus \(W_L=\{0\}\) after at most \(m\) steps. For the selected tests,
\eqref{eq:gram-identity-continuous} gives
\(\eta^\top M_L\eta=\sum_{\ell=1}^L\|A_{g_N^\ell}\eta\|_{H^{1/2}(\partial\Omega)}^2>0\)
for every nonzero \(\eta\), proving the claim.
\end{proof}

The bound \(L\le m\) counts Neumann excitations, each producing a full
Dirichlet trace, rather than scalar measurements. The selection depends on the
reference point and does not certify any prescribed family of tests or a
single family valid throughout the latent model.

\begin{thm}[Local stability from finite tests]
\label{thm:local-injectivity-from-finite-tests}
Under Assumption (A1), assume that \(M_L\) is positive definite. Then there
exist a convex neighborhood \(U\subset \mathcal{Z}\) of \(z^\dagger\) and
constants \(c_L,C_{\Phi,L}>0\) such that
\begin{equation}
\|z_1-z_2\|
\le
c_L
\|N_L^\Phi(z_1)-N_L^\Phi(z_2)\|_{Y_L}
\quad \forall z_1,z_2\in U
\label{eq:local-stability-finite-tests}
\end{equation}
and
\begin{equation}
\|\Phi(z_1)-\Phi(z_2)\|_{L^\infty(\Omega)}
\le
C_{\Phi,L}
\|N_L^\Phi(z_1)-N_L^\Phi(z_2)\|_{Y_L}
\quad \forall z_1,z_2\in U
\label{eq:local-coefficient-stability-continuous-finite-tests}
\end{equation}
In particular, \(N_L^\Phi\) is injective on \(U\).
\end{thm}

\begin{proof}
By \eqref{eq:gram-identity-continuous}, positive definiteness of \(M_L\)
implies that the derivative \(DN_L^\Phi(z^\dagger)\) has full column rank.
Hence there exists \(c_0>0\) such that
\begin{equation}
\|DN_L^\Phi(z^\dagger)\eta\|_{Y_L}
\ge
c_0 \|\eta\|
\quad \forall \eta\in \mathbb{R}^m
\label{eq:uniform-lower-bound-derivative}
\end{equation}
Since \(DN_L^\Phi\) is continuous, there exists a convex neighborhood \(U\) of \(z^\dagger\) such that
\begin{equation}
\|DN_L^\Phi(z)-DN_L^\Phi(z^\dagger)\|_{\mathcal{L}(\mathbb{R}^m,Y_L)}
\le
\frac{c_0}{2}
\quad \forall z\in U
\label{eq:derivative-close-to-reference}
\end{equation}
Let \(z_1,z_2\in U\), and set
\begin{equation}
z(t)=z_2+t(z_1-z_2)
\quad t\in[0,1]
\label{eq:segment-between-latent-points}
\end{equation}
Then
\begin{equation}
N_L^\Phi(z_1)-N_L^\Phi(z_2)
=
DN_L^\Phi(z^\dagger)(z_1-z_2)
+
\int_0^1
\bigl(
DN_L^\Phi(z(t))-DN_L^\Phi(z^\dagger)
\bigr)
(z_1-z_2)\,dt
\label{eq:mean-value-expansion-finite-tests}
\end{equation}
Using \eqref{eq:uniform-lower-bound-derivative} and \eqref{eq:derivative-close-to-reference}, we obtain
\begin{equation}
\|N_L^\Phi(z_1)-N_L^\Phi(z_2)\|_{Y_L}
\ge
\frac{c_0}{2}\|z_1-z_2\|
\label{eq:lower-bound-forward-difference}
\end{equation}
Thus \eqref{eq:local-stability-finite-tests} holds with \(c_L=2/c_0\).
After shrinking \(U\) if necessary, continuity of \(D\Phi\) gives a constant
\(L_{\Phi,U}>0\) such that
\begin{equation}
\|\Phi(z_1)-\Phi(z_2)\|_{L^\infty(\Omega)}
\le
L_{\Phi,U}\|z_1-z_2\|
\quad \forall z_1,z_2\in U\label{eq:local-lipschitz-bound-latent-map-continuous}
\end{equation}
Combining this estimate with \eqref{eq:local-stability-finite-tests} proves
\eqref{eq:local-coefficient-stability-continuous-finite-tests} with
\(C_{\Phi,L}=L_{\Phi,U}c_L\). Injectivity on \(U\) follows immediately from
\eqref{eq:local-stability-finite-tests}. \qedhere
\end{proof}

\paragraph{Scope of the Continuous Theory.}
\label{par:scope-of-the-continuous-theory}
The finite-test selection in
Theorem~\ref{thm:finite-tests-from-linearized-injectivity} is existential and
depends on the reference point, in particular on
\(a^\dagger=\Phi(z^\dagger)\) and the tangent space
\(T^\Phi_{z^\dagger}\). The theorem therefore does not provide a single family
of tests that works uniformly over \(\mathcal M_{\mathrm{lat}}\). The boundary
experiments used in Section~\ref{sec:numerical-experiments} are prescribed
computationally and are not obtained from this existence argument. The
stability analysis is local within the latent model about
\(a^\dagger=\Phi(z^\dagger)\).
Section~\ref{sec:discrete-theory-for-the-latent-forward-map} transfers this
stability to the finite element and neural surrogate approximations, and
Remark~\ref{rem:representation-mismatch-error} quantifies representation mismatch
relative to a comparison coefficient in the local stability set.

%% file: sec-analysis-discrete-rev0.tex
\section[Discrete and Surrogate Theory for Latent Forward Maps]{Discrete and Surrogate Theory\\ for Latent Forward Maps}
\label{sec:discrete-theory-for-the-latent-forward-map}

\paragraph{Overview.}
We transfer the continuous local stability to the finite element map through
convergence of the coefficient sensitivities
(Theorem~\ref{thm:local-stability-discrete-finite-test-map}). We then distinguish
two consequences of surrogate accuracy: derivative accuracy preserves the surrogate's own
local injectivity and stability
(Theorem~\ref{thm:stability-under-small-c1-perturbations}), while uniform value
accuracy suffices for a reconstruction-error estimate under residual comparison
(Corollary~\ref{cor:material-image-error-surrogate-residual-minimizers}).
Corollary~\ref{cor:l2-observation-stability} gives the corresponding conclusions
for the \(L^2\) boundary norm used in the computations.

\subsection{Discrete Well-Posedness and Differentiability}
\label{subsec:uniform-discrete-well-posedness}

For a fixed \(F\in(V_h^0)'\), independent of \(z\), consider
\begin{equation}
A_{\Phi(z)}(u_h(z),v_h)=F(v_h)
\quad \forall v_h\in V_h^0
\label{eq:discrete-state-problem-uniform-well-posedness}
\end{equation}
For the \(\ell\)th boundary experiment, we take \(F=F_{g_N^\ell}|_{V_h^0}\) and
write \(u_h^\ell(z)=u_h(\Phi(z);g_N^\ell)\).

\begin{prop}[Discrete well-posedness and differentiability]
\label{prop:c1-regularity-discrete-forward-map}
Assume (A1). The discrete state problem has a unique solution, and the state
map \(z\mapsto u_h(z)\in V_h^0\) and finite-test map
\(N_{L,h}^\Phi:\mathcal Z\to Y_{L,h}\) are of class \(C^1\). For each
\(\eta\in\mathbb R^m\), the sensitivity \(w_h(z;\eta)=Du_h(z)\eta\) satisfies
\begin{equation}
A_{\Phi(z)}(w_h(z;\eta),v_h)
=-(D\Phi(z)\eta\,\nabla u_h(z),\nabla v_h)_\Omega
\quad \forall v_h\in V_h^0
\label{eq:discrete-first-derivative-problem}
\end{equation}
The uniform bounds are
\begin{align}
\|u_h(z)\|_{H^1(\Omega)}
&\lesssim \|F\|_{(V_h^0)'}
\label{eq:discrete-uniform-stability}\\
\|w_h(z;\eta)\|_{H^1(\Omega)}
&\lesssim \|D\Phi(z)\eta\|_{L^\infty(\Omega)}\|u_h(z)\|_{H^1(\Omega)}
\label{eq:discrete-first-derivative-stability}
\end{align}
with constants independent of \(z\) and \(h\). In particular,
\begin{equation}
DN_{L,h}^\Phi(z)\eta
=\bigl(w_h^1(z;\eta)|_{\partial\Omega},\dots,w_h^L(z;\eta)|_{\partial\Omega}\bigr)
\label{eq:discrete-forward-map-derivative}
\end{equation}
where
\begin{equation}
A_{\Phi(z)}(w_h^\ell(z;\eta),v_h)
=-(D\Phi(z)\eta\,\nabla u_h^\ell(z),\nabla v_h)_\Omega
\quad \forall v_h\in V_h^0
\label{eq:discrete-forward-map-sensitivity-problem}
\end{equation}
\end{prop}

\begin{proof}
The bounds \(a_{\min}\le\Phi(z)\le a_{\max}\) and the Poincar\'e
inequality on the boundary-mean-zero space give continuity and coercivity of
\(A_{\Phi(z)}\), uniformly in \(z\) and \(h\). Lax--Milgram therefore gives
existence, uniqueness, and \eqref{eq:discrete-uniform-stability}. In any basis
of \(V_h^0\), the state coefficients solve \(K_h(z)U(z)=f_h\), where
\(K_h\) is \(C^1\) and invertible. Thus \(U=K_h^{-1}f_h\) is \(C^1\),
and differentiation gives \eqref{eq:discrete-first-derivative-problem}.
Testing with \(w_h\) and using coercivity yields
\eqref{eq:discrete-first-derivative-stability}. Finally, boundedness of the
trace operator gives the forward-map assertions.
\end{proof}

The same matrix argument gives \(C^k\)-regularity when \(\Phi\) is \(C^k\).
For data \(y\in Y_L\) and a regularization term \(\mathcal R\in C^1(\mathcal Z)\),
we define the objective, following \cite{BMPS18}, by
\begin{equation}
J_{L,h}(z)=\frac12\|N_{L,h}^\Phi(z)-y\|_{Y_L}^2+\mathcal R(z)
\label{eq:discrete-objective-functional-finite-tests}
\end{equation}
It is of class \(C^1\), with derivative
\begin{equation}
DJ_{L,h}(z)\eta
=\bigl(N_{L,h}^\Phi(z)-y,DN_{L,h}^\Phi(z)\eta\bigr)_{Y_L}
+D\mathcal R(z)\eta
\label{eq:discrete-objective-functional-derivative}
\end{equation}

\subsection{Derivative Error Estimates}
\label{subsec:derivative-error-estimates}

We now compare the continuous and discrete finite-test maps and their first
derivatives for the Neumann data fixed in
\eqref{eq:finite-family-neumann-data}. For each \(\ell=1,\dots,L\) and
\(z\in\mathcal Z\), we use the continuous and discrete states \(u^\ell(z)\) and
\(u_h^\ell(z)\) from \eqref{eq:background-state-latent-direction} and
\eqref{eq:discrete-weak-formulation}, respectively. Their sensitivities
\(w^\ell(z;\eta)\) and \(w_h^\ell(z;\eta)\) are characterized by
\eqref{eq:sensitivity-problem-latent-direction} and
\eqref{eq:discrete-forward-map-sensitivity-problem}. We first record the
standard best-approximation estimate for the discrete state.

\begin{prop}[Best-approximation estimate for the discrete state]
\label{prop:state-error-estimate}
Assume that \(\Phi(\mathcal{Z})\subset \mathcal{A}_{\mathrm{ad}}\). Then, for every \(z\in\mathcal{Z}\) and every \(\ell=1,\dots,L\),
\begin{equation}
\|u^\ell(z)-u_h^\ell(z)\|_{H^1(\Omega)}\lesssim \inf_{v_h\in V_h^0}\|u^\ell(z)-v_h\|_{H^1(\Omega)}
\label{eq:state-error-best-approximation}
\end{equation}
with a constant independent of \(z\), \(\ell\), and \(h\).
\end{prop}

\begin{proof}
The continuous and discrete weak problems give
\(A_{\Phi(z)}(u^\ell(z)-u_h^\ell(z),v_h)=0\) for all \(v_h\in V_h^0\).
Uniform continuity and coercivity of \(A_{\Phi(z)}\), together with the
Poincar\'e inequality, allow us to apply C\'ea's lemma and obtain
\eqref{eq:state-error-best-approximation}.
\end{proof}

We next compare the continuous and discrete sensitivities.

\begin{prop}[Best-approximation estimate for the discrete sensitivity]
\label{prop:first-derivative-error-estimate}
Assume that \(\Phi\in C^1(\mathcal{Z};L^{\infty}(\Omega))\) and that \(\Phi(\mathcal{Z})\subset \mathcal{A}_{\mathrm{ad}}\). Then, for every \(z\in\mathcal{Z}\), every \(\ell=1,\dots,L\), and every \(\eta\in\mathbb{R}^m\),
\begin{align}
\|w^\ell(z;\eta)-w_h^\ell(z;\eta)\|_{H^1(\Omega)}
&\lesssim
\inf_{v_h\in V_h^0}\|w^\ell(z;\eta)-v_h\|_{H^1(\Omega)}
\label{eq:first-derivative-error-best-approximation}\\
&\quad+
\|D\Phi(z)\eta\|_{L^{\infty}(\Omega)}\|u^\ell(z)-u_h^\ell(z)\|_{H^1(\Omega)}
\notag
\end{align}
with a constant independent of \(z\), \(\ell\), \(h\), and \(\eta\).
\end{prop}

\begin{proof}
Subtracting the sensitivity problems
\eqref{eq:discrete-forward-map-sensitivity-problem} and
\eqref{eq:sensitivity-problem-latent-direction} and testing with
\(v_h\in V_h^0\), we obtain
\begin{equation}
A_{\Phi(z)}(w^\ell(z;\eta)-w_h^\ell(z;\eta),v_h)=-(D\Phi(z)\eta\,\nabla (u^\ell(z)-u_h^\ell(z)),\nabla v_h)_{\Omega}\quad \forall v_h\in V_h^0
\label{eq:first-derivative-error-galerkin-identity}
\end{equation}

Let \(v_h\in V_h^0\) be arbitrary and set
\begin{equation}
\chi_h=v_h-w_h^\ell(z;\eta)
\label{eq:first-derivative-error-chi-definition}
\end{equation}
Using \eqref{eq:first-derivative-error-galerkin-identity}, we obtain
\begin{align}
A_{\Phi(z)}(\chi_h,\chi_h)
&=A_{\Phi(z)}(v_h-w^\ell(z;\eta),\chi_h)-(D\Phi(z)\eta\,\nabla (u^\ell(z)-u_h^\ell(z)),\nabla \chi_h)_{\Omega}
\label{eq:first-derivative-error-cea-step-1}
\end{align}
Hence, by uniform coercivity and continuity,
\begin{align}
\|\chi_h\|_{H^1(\Omega)}^2
&\lesssim \|v_h-w^\ell(z;\eta)\|_{H^1(\Omega)}\|\chi_h\|_{H^1(\Omega)}
\label{eq:first-derivative-error-cea-step-2}\\
&\quad+\|D\Phi(z)\eta\|_{L^{\infty}(\Omega)}\|u^\ell(z)-u_h^\ell(z)\|_{H^1(\Omega)}\|\chi_h\|_{H^1(\Omega)}
\notag
\end{align}
Therefore
\begin{equation}
\|v_h-w_h^\ell(z;\eta)\|_{H^1(\Omega)}
\lesssim
\|w^\ell(z;\eta)-v_h\|_{H^1(\Omega)}
+
\|D\Phi(z)\eta\|_{L^{\infty}(\Omega)}\|u^\ell(z)-u_h^\ell(z)\|_{H^1(\Omega)}
\label{eq:first-derivative-error-cea-step-4}
\end{equation}
and the triangle inequality gives
\begin{equation}
\|w^\ell(z;\eta)-w_h^\ell(z;\eta)\|_{H^1(\Omega)}
\le
\|w^\ell(z;\eta)-v_h\|_{H^1(\Omega)}+\|v_h-w_h^\ell(z;\eta)\|_{H^1(\Omega)}
\label{eq:first-derivative-error-cea-step-5}
\end{equation}
Taking the infimum in \eqref{eq:first-derivative-error-cea-step-4} and using \eqref{eq:first-derivative-error-cea-step-5} gives \eqref{eq:first-derivative-error-best-approximation}.
\end{proof}

We now transfer the previous estimates to the continuous and discrete finite-test maps \(N_L^\Phi\) and \(N_{L,h}^\Phi\), defined in \eqref{eq:continuous-latent-finite-test-map} and \eqref{eq:training-target-finite-test-map}, respectively.

\begin{cor}[Error estimates for the discrete finite-test map]
\label{cor:forward-map-and-derivative-error-estimates}
Under the assumptions of Proposition~\ref{prop:first-derivative-error-estimate}, the following estimates hold for every \(z\in\mathcal{Z}\) and every \(\eta\in\mathbb{R}^m\):
\begin{align}
\|N_L^\Phi(z)-N_{L,h}^\Phi(z)\|_{Y_L}
&\lesssim
\Biggl(
\sum_{\ell=1}^L
\inf_{v_h\in V_h^0}\|u^\ell(z)-v_h\|_{H^1(\Omega)}^2
\Biggr)^{1/2}
\label{eq:forward-map-error-estimate}\\
\|DN_L^\Phi(z)\eta-DN_{L,h}^\Phi(z)\eta\|_{Y_L}
&\lesssim
\Biggl(
\sum_{\ell=1}^L
\Bigl(
\inf_{v_h\in V_h^0}\|w^\ell(z;\eta)-v_h\|_{H^1(\Omega)}
\Bigr)^2
\Biggr)^{1/2}
\label{eq:forward-map-derivative-error-estimate-1}\\
&\quad+
\|D\Phi(z)\eta\|_{L^{\infty}(\Omega)}
\Biggl(
\sum_{\ell=1}^L
\|u^\ell(z)-u_h^\ell(z)\|_{H^1(\Omega)}^2
\Biggr)^{1/2}
\notag
\end{align}
\end{cor}

\begin{proof}
By the boundedness of the trace operator \(H^1(\Omega)\to H^{1/2}(\partial\Omega)\),
\begin{equation}
\|u^\ell(z)|_{\partial\Omega}-u_h^\ell(z)|_{\partial\Omega}\|_{H^{1/2}(\partial\Omega)}
\lesssim
\|u^\ell(z)-u_h^\ell(z)\|_{H^1(\Omega)}
\label{eq:forward-map-error-trace-step}
\end{equation}
and similarly
\begin{equation}
\|w^\ell(z;\eta)|_{\partial\Omega}-w_h^\ell(z;\eta)|_{\partial\Omega}\|_{H^{1/2}(\partial\Omega)}
\lesssim
\|w^\ell(z;\eta)-w_h^\ell(z;\eta)\|_{H^1(\Omega)}
\label{eq:forward-map-derivative-error-trace-step}
\end{equation}
Summing over \(\ell=1,\dots,L\) and applying Propositions~\ref{prop:state-error-estimate} and~\ref{prop:first-derivative-error-estimate} gives \eqref{eq:forward-map-error-estimate} and \eqref{eq:forward-map-derivative-error-estimate-1}.
\end{proof}

\begin{cor}[Operator-norm convergence of the discrete derivative]
\label{cor:operator-norm-convergence-discrete-derivative}
Assume the hypotheses of Proposition~\ref{prop:first-derivative-error-estimate}, and let \(z^{\dagger}\in\mathcal Z\). Suppose that the family \(\{V_h^0\}_h\) satisfies the standard approximation property
\begin{equation}
\lim_{h\to 0}\inf_{v_h\in V_h^0}\|v-v_h\|_{H^1(\Omega)}=0\quad \forall v\in V
\label{eq:standard-approximation-property-discrete-spaces}
\end{equation}
Then
\begin{equation}
\|DN_L^\Phi(z^{\dagger})-DN_{L,h}^\Phi(z^{\dagger})\|_{\mathcal L(\mathbb R^m,Y_L)}\to 0\quad \text{as }h\to 0
\label{eq:operator-norm-convergence-discrete-derivative}
\end{equation}
\end{cor}

\begin{proof}
Set \(T_h=DN_L^\Phi(z^\dagger)-DN_{L,h}^\Phi(z^\dagger)\).
The approximation property and
Corollary~\ref{cor:forward-map-and-derivative-error-estimates} give
\(\|T_he_j\|_{Y_L}\to0\) for each canonical basis vector \(e_j\).
By the Cauchy--Schwarz inequality,
\begin{equation}
\|T_h\|_{\mathcal L(\mathbb R^m,Y_L)}
\le\biggl(\sum_{j=1}^m\|T_he_j\|_{Y_L}^2\biggr)^{1/2}\to0
\label{eq:discrete-derivative-error-basis-bound}
\end{equation}
\end{proof}

\begin{rem}[Convergence rates from approximation properties]
\label{rem:rates-from-approximation-properties}
The best-approximation form of these estimates leaves the finite element degree and solution regularity unspecified. Combined with standard interpolation estimates, they yield convergence rates for the finite-test map and its first derivative. Corollary~\ref{cor:operator-norm-convergence-discrete-derivative} supplies the operator convergence used below to show that the discrete Gram matrix remains positive definite for sufficiently small \(h\).
\end{rem}

\subsection{Positivity of the Discrete Gram Matrix and Local Stability}
\label{subsec:positivity-of-the-discrete-gram-matrix-and-local-stability}

We now return to the reference point \(z^{\dagger}\in \mathcal{Z}\). Under Assumption (A1), suppose that the finite family of Neumann tests \(g_N^1,\dots,g_N^L\) has been chosen so that the continuous cumulative Gram matrix \(M_L\), defined in \eqref{eq:cumulative-gram-matrix-continuous}, is positive definite. The Gram identity \eqref{eq:gram-identity-continuous} then gives, with \(c_0=\sqrt{\lambda_{\min}(M_L)}>0\),
\begin{equation}
\|DN_L^\Phi(z^{\dagger})\eta\|_{Y_L}\ge c_0\|\eta\|\quad \forall \eta\in \mathbb{R}^m
\label{eq:continuous-lower-bound-reference-point}
\end{equation}
This is precisely the full-rank condition obtained in the continuous theory.

\paragraph{The Discrete Gram Matrix.}
\label{par:the-discrete-gram-matrix}
Let \(\{e_j\}_{j=1}^m\) denote the canonical basis of \(\mathbb{R}^m\). For each \(\ell=1,\dots,L\) and \(j=1,\dots,m\), we define the discrete linearized responses by
\begin{equation}
q_{j,h}^\ell=\partial_{e_j}\bigl(\mathcal{N}_{\Phi(z),h}g_N^\ell\bigr)\big|_{z=z^{\dagger}}\in H^{1/2}_{\diamond}(\partial\Omega)
\label{eq:discrete-linearized-responses}
\end{equation}
and the associated discrete cumulative Gram matrix by
\begin{equation}
[M_{L,h}]_{ij}=\sum_{\ell=1}^L (q_{i,h}^\ell,q_{j,h}^\ell)_{H^{1/2}(\partial\Omega)}
\label{eq:discrete-cumulative-gram-matrix}
\end{equation}
Then, for every \(\eta=(\eta_1,\dots,\eta_m)^{\top}\in \mathbb{R}^m\),
\begin{equation}
\eta^{\top}M_{L,h}\eta=\|DN_{L,h}^\Phi(z^{\dagger})\eta\|_{Y_L}^2
\label{eq:discrete-gram-identity}
\end{equation}
where \(N_{L,h}^\Phi\) is the discrete finite-test forward map from \eqref{eq:training-target-finite-test-map}.

\begin{prop}[Discrete linearized injectivity]
\label{prop:positivity-of-the-discrete-gram-matrix}
Assume that \eqref{eq:continuous-lower-bound-reference-point} holds and that there exists a sequence \(\varepsilon_h\to 0\) such that
\begin{equation}
\|DN_L^\Phi(z^{\dagger})-DN_{L,h}^\Phi(z^{\dagger})\|_{\mathcal L(\mathbb{R}^m,Y_L)}\le \varepsilon_h
\label{eq:derivative-perturbation-at-reference-point}
\end{equation}
Then, for all sufficiently small \(h\), the discrete derivative \(DN_{L,h}^\Phi(z^{\dagger})\) is injective and satisfies
\begin{equation}
\|DN_{L,h}^\Phi(z^{\dagger})\eta\|_{Y_L}\ge \frac{c_0}{2}\|\eta\|\quad \forall \eta\in \mathbb{R}^m
\label{eq:discrete-lower-bound-reference-point}
\end{equation}
In particular, the discrete Gram matrix \(M_{L,h}\) is positive definite for these mesh sizes.
\end{prop}

\begin{proof}
For sufficiently small \(h\), we have \(\varepsilon_h\le c_0/2\). The
triangle inequality then gives
\begin{equation}
\|DN_{L,h}^\Phi(z^\dagger)\eta\|_{Y_L}
\ge(c_0-\varepsilon_h)\|\eta\|
\ge\frac{c_0}{2}\|\eta\|
\label{eq:discrete-lower-bound-step-3}
\end{equation}
The Gram identity \eqref{eq:discrete-gram-identity} implies
\(\eta^\top M_{L,h}\eta\ge(c_0^2/4)\|\eta\|^2\), proving positive definiteness.
\end{proof}

\begin{rem}[Verification of the derivative perturbation bound]
\label{rem:obtaining-the-derivative-perturbation-bound}
Under the standard approximation property \eqref{eq:standard-approximation-property-discrete-spaces}, Corollary~\ref{cor:operator-norm-convergence-discrete-derivative} verifies \eqref{eq:derivative-perturbation-at-reference-point} with \(\varepsilon_h\to 0\) as \(h\to 0\).
\end{rem}

We now show that the injectivity of the discrete derivative at \(z^{\dagger}\) yields local injectivity and Lipschitz stability of the discrete finite-test map.

\begin{thm}[Local stability of the discrete finite-test map]
\label{thm:local-stability-discrete-finite-test-map}
Under Assumption (A1), suppose that the continuous cumulative Gram matrix \(M_L\) associated with the selected Neumann tests is positive definite and that the family \(\{V_h^0\}_h\) satisfies the standard approximation property \eqref{eq:standard-approximation-property-discrete-spaces}. Set \(c_0=\sqrt{\lambda_{\min}(M_L)}\). Then there exist \(h_0>0\), a convex neighborhood \(U\subset\mathcal Z\) of \(z^\dagger\), and a constant \(c_L^{\mathrm d}=4/c_0\), with \(U\) and \(c_L^{\mathrm d}\) independent of \(h\), such that, for every \(0<h\le h_0\),
\begin{equation}
\|z_1-z_2\|\le c_L^{\mathrm d}\|N_{L,h}^\Phi(z_1)-N_{L,h}^\Phi(z_2)\|_{Y_L}\quad \forall z_1,z_2\in U
\label{eq:local-stability-discrete-finite-test-map}
\end{equation}
In particular, \(N_{L,h}^\Phi\) is injective on \(U\).
\end{thm}

\begin{proof}
Corollary~\ref{cor:operator-norm-convergence-discrete-derivative} and
Proposition~\ref{prop:positivity-of-the-discrete-gram-matrix} give the lower
bound \eqref{eq:discrete-lower-bound-reference-point} for \(0<h\le h_0\).
To obtain a neighborhood independent of \(h\), we first establish uniform
continuity of the discrete derivative at \(z^\dagger\). For
\(z_1,z_2\in\mathcal Z\) and \(\eta\in\mathbb R^m\), write
\(a_i=\Phi(z_i)\), \(b_i=D\Phi(z_i)\eta\),
\(u_{h,i}^\ell=u_h^\ell(z_i)\), and \(w_{h,i}^\ell=w_h^\ell(z_i;\eta)\),
\(i=1,2\). Subtracting the discrete state and sensitivity equations and using
uniform coercivity and the bounds in
Proposition~\ref{prop:c1-regularity-discrete-forward-map} gives
\begin{align}
\|u_{h,1}^\ell-u_{h,2}^\ell\|_{H^1(\Omega)}
&\lesssim\|a_1-a_2\|_{L^\infty(\Omega)}
\|g_N^\ell\|_{H^{-1/2}(\partial\Omega)}
\label{eq:discrete-state-uniform-continuity}\\
\|w_{h,1}^\ell-w_{h,2}^\ell\|_{H^1(\Omega)}
&\lesssim\bigl(\|b_1-b_2\|_{L^\infty(\Omega)}
+\|a_1-a_2\|_{L^\infty(\Omega)}\|b_2\|_{L^\infty(\Omega)}\bigr)
\|g_N^\ell\|_{H^{-1/2}(\partial\Omega)}
\label{eq:discrete-sensitivity-uniform-continuity}
\end{align}
with constants independent of \(h\), \(z_1\), \(z_2\), \(\ell\), and \(\eta\).
Set \(X=\mathcal L(\mathbb R^m,L^\infty(\Omega))\). Taking traces,
summing over the fixed Neumann tests, and taking the supremum over
\(\|\eta\|=1\) yields
\begin{align}
\|DN_{L,h}^\Phi(z)-DN_{L,h}^\Phi(z^\dagger)\|_{\mathcal L(\mathbb R^m,Y_L)}
&\lesssim\|D\Phi(z)-D\Phi(z^\dagger)\|_X
\label{eq:discrete-derivative-uniform-continuity}\\
&\quad+\|\Phi(z)-\Phi(z^\dagger)\|_{L^\infty(\Omega)}
\|D\Phi(z^\dagger)\|_X
\notag
\end{align}
The implied constant is independent of \(h\). Since \(\Phi\) is \(C^1\),
the right-hand side tends to zero as \(z\to z^\dagger\). We can therefore
choose a convex neighborhood \(U\subset\mathcal Z\) of \(z^\dagger\),
independent of \(h\), on which
\begin{equation}
\|DN_{L,h}^\Phi(z)-DN_{L,h}^\Phi(z^\dagger)\|_{\mathcal L(\mathbb R^m,Y_L)}
\le\frac{c_0}{4}
\quad \forall z\in U
\label{eq:discrete-derivative-close-in-neighborhood}
\end{equation}
Integrating the derivative along the segment joining \(z_1,z_2\in U\),
as in the proof of Theorem~\ref{thm:local-injectivity-from-finite-tests}, gives
\begin{equation}
\|N_{L,h}^\Phi(z_1)-N_{L,h}^\Phi(z_2)\|_{Y_L}
\ge\left(\frac{c_0}{2}-\frac{c_0}{4}\right)\|z_1-z_2\|
\label{eq:discrete-local-stability-step-4}
\end{equation}
This proves the estimate with \(c_L^{\mathrm d}=4/c_0\) and hence injectivity.
\end{proof}

The local stability estimate for \(N_{L,h}^\Phi\) also gives an error estimate for the corresponding material coefficients. Fix a compact convex set \(K\subset U\), independent of \(h\), where \(U\) is the neighborhood from Theorem~\ref{thm:local-stability-discrete-finite-test-map}. Since \(\Phi\in C^1(\mathcal Z;L^\infty(\Omega))\), the map \(\Phi\) is Lipschitz on \(K\). We set
\begin{equation}
L_{\Phi,K}\coloneqq \sup_{z\in K}\|D\Phi(z)\|_{\mathcal L(\mathbb R^m,L^\infty(\Omega))} \label{eq:material-image-local-lipschitz-constant}
\end{equation}
Then \(L_{\Phi,K}\) controls the passage from a latent error to a material-image error.

\begin{prop}[Material-image error under residual comparison]
\label{prop:material-image-error-local-discrete-minimizers}
Assume the hypotheses of Theorem~\ref{thm:local-stability-discrete-finite-test-map}, and let \(K\subset U\) be compact and convex. Let \(y\in Y_L\), let \(z_\ast\in K\), and let \(z_{h,\ast}\in K\) satisfy the residual-comparison condition
\begin{equation}
\|N_{L,h}^\Phi(z_{h,\ast})-y\|_{Y_L}\le \|N_{L,h}^\Phi(z_\ast)-y\|_{Y_L} \label{eq:residual-minimality-discrete-latent-point}
\end{equation}
Define the corresponding material images by \(a_\ast=\Phi(z_\ast)\) and \(a_{h,\ast}=\Phi(z_{h,\ast})\). Then
\begin{equation}
\|a_{h,\ast}-a_\ast\|_{L^\infty(\Omega)}
\le 2L_{\Phi,K}c_L^{\mathrm d}\bigl(\|N_L^\Phi(z_\ast)-y\|_{Y_L}+\|N_L^\Phi(z_\ast)-N_{L,h}^\Phi(z_\ast)\|_{Y_L}\bigr) \label{eq:material-image-error-local-discrete-minimizer}
\end{equation}
In particular, if the data are generated by the continuous model at \(z_\ast\), so that \(y=N_L^\Phi(z_\ast)\), then
\begin{align}
\|a_{h,\ast}-a_\ast\|_{L^\infty(\Omega)}
&\le 2L_{\Phi,K}c_L^{\mathrm d}\|N_L^\Phi(z_\ast)-N_{L,h}^\Phi(z_\ast)\|_{Y_L} \label{eq:material-image-error-exact-data-forward-error}\\
&\lesssim L_{\Phi,K}c_L^{\mathrm d}\biggl(\sum_{\ell=1}^L\inf_{v_h\in V_h^0}\|u^\ell(z_\ast)-v_h\|_{H^1(\Omega)}^2\biggr)^{1/2} \label{eq:material-image-error-exact-data-best-approximation}
\end{align}
\end{prop}

\begin{proof}
The triangle inequality and \eqref{eq:residual-minimality-discrete-latent-point}
give
\begin{equation}
\|N_{L,h}^\Phi(z_{h,\ast})-N_{L,h}^\Phi(z_\ast)\|_{Y_L}
\le2\|N_{L,h}^\Phi(z_\ast)-y\|_{Y_L}
\label{eq:material-image-proof-minimality}
\end{equation}
Apply discrete local stability and the Lipschitz bound for \(\Phi\), then
split the right-hand side into data and finite element errors to obtain
\eqref{eq:material-image-error-local-discrete-minimizer}. For exact data,
Corollary~\ref{cor:forward-map-and-derivative-error-estimates} gives
\eqref{eq:material-image-error-exact-data-best-approximation}.
\end{proof}

The comparison condition holds for a global residual minimizer over the local
compact set \(K\). It need not hold for an arbitrary local minimizer or
an iterate of the numerical optimization method.

\begin{rem}[Regularized objectives]
\label{rem:material-image-error-regularized-objectives}
If \(J_{L,h}(z_{h,\ast})\le J_{L,h}(z_\ast)\), the same argument adds
\begin{equation}
L_{\Phi,K}c_L^{\mathrm d}
\bigl(2(\mathcal R(z_\ast)-\mathcal R(z_{h,\ast}))_+\bigr)^{1/2}
\label{eq:regularized-material-image-error-estimate}
\end{equation}
to the right-hand side of
\eqref{eq:material-image-error-local-discrete-minimizer}, where
\(s_+=\max\{s,0\}\). This follows by comparing the squared residuals in
\eqref{eq:discrete-objective-functional-finite-tests} and using
\(\sqrt{r^2+t}\le r+\sqrt t\) for \(r,t\ge0\).
\end{rem}

\subsection{Stability under Neural-Surrogate Approximation}
\label{subsec:neural-network-surrogates-and-stability-under-small-perturbations}

We now replace the discrete finite-test map \(N_{L,h}^\Phi\) by a neural-network surrogate. As described in Section~\ref{sec:problem-formulation-finite-element-method-and-neural-surrogate-approximation}, the network approximates the coordinate representation of this map directly on \(\mathcal{Z}\).
\paragraph{Finite-Dimensional Output Representation.}
\label{par:finite-dimensional-output-representation}	
We first study stability of the surrogate itself by viewing it as a \(C^1\)-perturbation of \(N_{L,h}^\Phi\). Since the discrete output space \(Y_{L,h}\) is finite-dimensional, we can work in coordinates near the reference point \(z^{\dagger}\). The coordinate isomorphism \(\Lambda_h\) from \eqref{eq:method-coordinate-isomorphism-discrete-output} induces a norm equivalence: there exist constants \(c_{\Lambda,h}>0\) and \(C_{\Lambda,h}>0\) such that
\begin{equation}
c_{\Lambda,h}\|\xi\|_{Y_L}\le \|\Lambda_h \xi\|_{\IR^{P_h}}\le C_{\Lambda,h}\|\xi\|_{Y_L}\quad \forall \xi\in Y_{L,h}
\label{eq:norm-equivalence-coordinate-output}
\end{equation}

We now show that the local injectivity and stability of \(N_{L,h}^\Phi\) persist under sufficiently small perturbations in \(C^1\). Let \(U\subset \mathcal{Z}\) be a convex neighborhood of \(z^{\dagger}\) on which the discrete finite-test map satisfies the local stability estimate from Theorem~\ref{thm:local-stability-discrete-finite-test-map}, that is,
\begin{equation}
\|z_1-z_2\|\le c_L^{\mathrm d}\|N_{L,h}^\Phi(z_1)-N_{L,h}^\Phi(z_2)\|_{Y_L}\quad \forall z_1,z_2\in U
\label{eq:discrete-local-stability-on-neighborhood}
\end{equation}
Recall the definition of the neural-network surrogate \(N_{L,h,M}^\Phi\) given in \eqref{eq:two-layer-network-discrete-forward-map} and \eqref{eq:network-surrogate-discrete-forward-map}.
\begin{thm}[\boldmath Stability under small \(C^1\) perturbations]
\label{thm:stability-under-small-c1-perturbations}
Let \( N_{L,h,M}^\Phi \in C^1(U;Y_{L,h})\) satisfy
\begin{equation}
\sup_{z\in U}\|D(N_{L,h}^\Phi-N_{L,h,M}^\Phi)(z)\|_{\mathcal L(\mathbb{R}^m,Y_L)}\le \frac{1}{2c_L^{\mathrm d}}
\label{eq:small-c1-perturbation-assumption}
\end{equation}
Then
\begin{equation}
\|z_1-z_2\|\le 2c_L^{\mathrm d}\|N_{L,h,M}^\Phi(z_1)-N_{L,h,M}^\Phi (z_2)\|_{Y_L}\quad \forall z_1,z_2\in U
\label{eq:local-stability-network-surrogate}
\end{equation}
In particular, \(N_{L,h,M}^\Phi\) is injective on \(U\).
\end{thm}

\begin{proof}
Set \(E_h=N_{L,h}^\Phi-N_{L,h,M}^\Phi\). Convexity of \(U\) and
\eqref{eq:small-c1-perturbation-assumption} give
\(\|E_h(z_1)-E_h(z_2)\|_{Y_L}\le(2c_L^{\mathrm d})^{-1}\|z_1-z_2\|\)
by integration along the line segment. Thus
\eqref{eq:discrete-local-stability-on-neighborhood} and the triangle inequality yield
\begin{align}
\|N_{L,h,M}^\Phi(z_1)-N_{L,h,M}^\Phi(z_2)\|_{Y_L}
&\ge \|N_{L,h}^\Phi(z_1)-N_{L,h}^\Phi(z_2)\|_{Y_L}
-\|E_h(z_1)-E_h(z_2)\|_{Y_L}
\label{eq:network-surrogate-stability-step-1}\\
&\ge\frac{1}{2c_L^{\mathrm d}}\|z_1-z_2\|
\label{eq:network-surrogate-stability-step-3}
\end{align}
This proves the estimate and injectivity.
\end{proof}

\paragraph{\boldmath \(C^1\)-Approximation Hypothesis.}
\label{par:c1-approximation-hypothesis}
To connect the discrete theory with neural-network approximation, we formulate the approximation step in terms of the coordinate map \(\widehat{N}_{L,h}^\Phi\). For each fixed \(h\), let \(K\subset U\) be a compact convex set containing \(z^{\dagger}\). Assume that there exists a sequence of two-layer networks \(\widehat{N}_{L,h,M}^\Phi \) such that
\begin{equation}
\|\widehat{N}_{L,h}^\Phi-\widehat{N}_{L,h,M}^\Phi\|_{C^1(K;\mathbb{R}^{P_h})}\to 0\quad \text{as }M\to \infty
\label{eq:c1-approximation-coordinate-map}
\end{equation}
Then the associated surrogates \(N_{L,h,M}^\Phi\) defined by \eqref{eq:network-surrogate-discrete-forward-map} satisfy
\begin{equation}
\|N_{L,h}^\Phi-N_{L,h,M}^\Phi\|_{C^1(K;Y_L)}\to 0\quad \text{as }M\to \infty
\label{eq:c1-approximation-forward-map}
\end{equation}
in the sense induced by the norm equivalence \eqref{eq:norm-equivalence-coordinate-output}.

\begin{prop}[\boldmath Stability under \(C^1\)-convergent neural approximation]
\label{prop:c1-approximation-implies-stable-surrogate}
For a fixed \(h\), assume that \eqref{eq:c1-approximation-coordinate-map} holds. Then, for all sufficiently large widths \(M\),
\begin{equation}
\sup_{z\in K}\|D(N_{L,h}^\Phi-N_{L,h,M}^\Phi)(z)\|_{\mathcal L(\mathbb{R}^m,Y_L)}\le \frac{1}{2c_L^{\mathrm d}}
\label{eq:large-width-derivative-threshold}
\end{equation}
and therefore the neural-network surrogate \(N_{L,h,M}^\Phi\) is injective on \(K\) and satisfies
\begin{equation}
\|z_1-z_2\|\le 2c_L^{\mathrm d}\|N_{L,h,M}^\Phi(z_1)-N_{L,h,M}^\Phi(z_2)\|_{Y_L}\quad \forall z_1,z_2\in K
\label{eq:stable-surrogate-from-c1-approximation}
\end{equation}
\end{prop}

\begin{proof}
The convergence \eqref{eq:c1-approximation-coordinate-map} and the norm equivalence \eqref{eq:norm-equivalence-coordinate-output} imply \eqref{eq:c1-approximation-forward-map}. In particular, \eqref{eq:large-width-derivative-threshold} holds for all sufficiently large \(M\). The claim then follows from Theorem~\ref{thm:stability-under-small-c1-perturbations}.
\end{proof}

The hypothesis \eqref{eq:c1-approximation-coordinate-map} is not established
here for the trained networks. Barron-space estimates could provide one route
to its verification for suitable coordinate maps and activations, but
\(C^1\)-regularity alone does not imply such estimates.

\begin{cor}[Material-image error under surrogate residual comparison]
\label{cor:material-image-error-surrogate-residual-minimizers}
Assume the hypotheses of
Theorem~\ref{thm:local-stability-discrete-finite-test-map}, fix
\(0<h\le h_0\), and let \(K\subset U\) be compact and convex.
Let \(N_{L,h,M}^\Phi:K\to Y_{L,h}\) be a continuous surrogate,
let \(y\in Y_L\), and suppose \(z_\ast,z_{h,M,\ast}\in K\) satisfy
\begin{equation}
\|N_{L,h,M}^\Phi(z_{h,M,\ast})-y\|_{Y_L}
\le\|N_{L,h,M}^\Phi(z_\ast)-y\|_{Y_L}
\label{eq:residual-minimality-surrogate-latent-point}
\end{equation}
Set
\begin{equation}
\delta_{h,M}=\sup_{z\in K}
\|N_{L,h}^\Phi(z)-N_{L,h,M}^\Phi(z)\|_{Y_L}
\label{eq:uniform-surrogate-forward-error}
\end{equation}
Then
\begin{equation}
\|\Phi(z_{h,M,\ast})-\Phi(z_\ast)\|_{L^\infty(\Omega)}
\le2L_{\Phi,K}c_L^{\mathrm d}
\bigl(\|N_{L,h}^\Phi(z_\ast)-y\|_{Y_L}+\delta_{h,M}\bigr)
\label{eq:material-image-error-surrogate-residual-minimizer}
\end{equation}
For data \(y=N_L^\Phi(z_\ast)+e\) with \(\|e\|_{Y_L}\le\varepsilon\),
this gives
\begin{equation}
\|\Phi(z_{h,M,\ast})-\Phi(z_\ast)\|_{L^\infty(\Omega)}
\le2L_{\Phi,K}c_L^{\mathrm d}\bigl(\varepsilon
+\|N_L^\Phi(z_\ast)-N_{L,h}^\Phi(z_\ast)\|_{Y_L}
+\delta_{h,M}\bigr)
\label{eq:material-image-error-surrogate-noisy-data}
\end{equation}
In particular, for exact continuous data,
\begin{equation}
\|\Phi(z_{h,M,\ast})-\Phi(z_\ast)\|_{L^\infty(\Omega)}
\le2L_{\Phi,K}c_L^{\mathrm d}
\bigl(\|N_L^\Phi(z_\ast)-N_{L,h}^\Phi(z_\ast)\|_{Y_L}+\delta_{h,M}\bigr)
\label{eq:material-image-error-surrogate-continuous-data}
\end{equation}
\end{cor}

\begin{proof}
Write \(F=N_{L,h}^\Phi\), \(S=N_{L,h,M}^\Phi\), and
\(\widehat z=z_{h,M,\ast}\). The uniform error and residual comparison give
\begin{align}
\|F(\widehat z)-F(z_\ast)\|_{Y_L}
&\le\delta_{h,M}+\|S(\widehat z)-y\|_{Y_L}+\|y-F(z_\ast)\|_{Y_L}
\label{eq:surrogate-material-error-proof-triangle}\\
&\le2\delta_{h,M}+2\|F(z_\ast)-y\|_{Y_L}
\label{eq:surrogate-material-error-proof-forward-split}
\end{align}
Apply the local stability of \(F\) and the Lipschitz continuity of \(\Phi\)
on \(K\). The remaining estimates follow from the triangle inequality.
\end{proof}

This corollary requires only uniform value accuracy of the surrogate. No
derivative accuracy or surrogate injectivity is needed. Since the surrogate is
continuous, its residual attains a global minimum on \(K\). Every such minimizer
satisfies \eqref{eq:residual-minimality-surrogate-latent-point}. The result does
not assert convergence of an optimization algorithm or that its iterates remain in
\(K\). Training and validation mean squared errors do not by themselves
certify the uniform bound \(\delta_{h,M}\).

\begin{rem}[Representation mismatch]
\label{rem:representation-mismatch-error}
Retain the hypotheses and notation of
Corollary~\ref{cor:material-image-error-surrogate-residual-minimizers}, with
\(a^\dagger=\Phi(z^\dagger)\) as the reference coefficient for local
stability. Allow the true coefficient to be an arbitrary
\(\widetilde a\in\mathcal A_{\mathrm{ad}}\), with data
\(y=N_L(\widetilde a)+e\), \(\|e\|_{Y_L}\le\varepsilon\). For
\(z_\ast,z_{h,M,\ast}\in K\) satisfying
\eqref{eq:residual-minimality-surrogate-latent-point}, set
\begin{equation}
e_{\mathrm{rep}}=\|\widetilde a-\Phi(z_\ast)\|_{L^\infty(\Omega)},
\qquad
E_h=\|N_L^\Phi(z_\ast)-N_{L,h}^\Phi(z_\ast)\|_{Y_L}
\label{eq:representation-and-discretization-errors}
\end{equation}
Let \(C_F\) be a Lipschitz constant of
\(N_L:\mathcal A_{\mathrm{ad}}\to Y_L\), obtained by the state subtraction
in \eqref{eq:continuous-solution-perturbation-estimate} and the trace bound
for the fixed tests. Then
\begin{equation}
\|\Phi(z_{h,M,\ast})-\widetilde a\|_{L^\infty(\Omega)}
\le\bigl(1+2L_{\Phi,K}c_L^{\mathrm d}C_F\bigr)e_{\mathrm{rep}}
+2L_{\Phi,K}c_L^{\mathrm d}
\bigl(\varepsilon+E_h+\delta_{h,M}\bigr)
\label{eq:material-image-error-with-representation-mismatch}
\end{equation}
Indeed,
\(\|N_{L,h}^\Phi(z_\ast)-y\|_{Y_L}\le E_h+C_Fe_{\mathrm{rep}}+\varepsilon\).
Apply \eqref{eq:material-image-error-surrogate-residual-minimizer} and add
\(e_{\mathrm{rep}}\) by the triangle inequality. The estimate remains local:
both latent points must lie in \(K\), and residual comparison is still required.
\end{rem}

\subsection[Stability in the Boundary L2-Norm]{Stability in the Boundary \(L^2\)-Norm}
\label{subsec:l2-observation-stability}

The boundary misfit used in the numerical experiments can be treated directly
in the observation space
\begin{equation}
Y_L^{(0)}=\bigl(L^2(\partial\Omega)\bigr)^L
\label{eq:l2-observation-space}
\end{equation}
This space also includes residuals with a nonzero boundary mean.
Using the continuous responses from \eqref{eq:linearized-responses}, define
\begin{equation}
[M_L^{(0)}]_{ij}
=\sum_{\ell=1}^L(q_i(g_N^\ell),q_j(g_N^\ell))_{L^2(\partial\Omega)},
\quad i,j=1,\dots,m
\label{eq:l2-continuous-gram-matrix}
\end{equation}

\begin{cor}[Local stability with \(L^2\) boundary observations]
\label{cor:l2-observation-stability}
Under the hypotheses of
Theorem~\ref{thm:local-stability-discrete-finite-test-map}, the matrix
\(M_L^{(0)}\) is positive definite. Set
\(c_0^{(0)}=\sqrt{\lambda_{\min}(M_L^{(0)})}\).
The continuous local stability estimate holds in \(Y_L^{(0)}\) on a
sufficiently small neighborhood of \(z^\dagger\). Moreover, there exist
\(h_0^{(0)}>0\) and a convex neighborhood \(U^{(0)}\) of \(z^\dagger\),
independent of \(h\), such that, for every \(0<h\le h_0^{(0)}\) and
\(z_1,z_2\in U^{(0)}\),
\begin{equation}
\|z_1-z_2\|
\le c_L^{\mathrm d,(0)}
\|N_{L,h}^\Phi(z_1)-N_{L,h}^\Phi(z_2)\|_{Y_L^{(0)}},
\quad c_L^{\mathrm d,(0)}=\frac{4}{c_0^{(0)}}
\label{eq:l2-discrete-local-stability}
\end{equation}
The constant is independent of \(h\). The material-image estimates of
Proposition~\ref{prop:material-image-error-local-discrete-minimizers},
Corollary~\ref{cor:material-image-error-surrogate-residual-minimizers}, and
Remark~\ref{rem:representation-mismatch-error} hold
with \(Y_L\), \(c_L^{\mathrm d}\), and \(K\subset U\) replaced by
\(Y_L^{(0)}\), \(c_L^{\mathrm d,(0)}\), and a compact convex subset of
\(U^{(0)}\), respectively. A \(C^1\) surrogate itself inherits local stability if
its derivative error in \(\mathcal L(\mathbb R^m,Y_L^{(0)})\) is at most
\(1/(2c_L^{\mathrm d,(0)})\) on that set.
For the mismatch estimate, \(C_F\) is taken in the \(Y_L^{(0)}\)-norm.
\end{cor}

\begin{proof}
If \(\eta^\top M_L^{(0)}\eta=0\), then
\(\sum_{j=1}^m\eta_jq_j(g_N^\ell)=0\) in \(L^2(\partial\Omega)\) for every
\(\ell\). These traces belong to \(H^{1/2}(\partial\Omega)\), so they also
vanish in that space. Positive definiteness of \(M_L\) gives \(\eta=0\).
The corresponding Gram identity therefore yields a linearized lower bound
with constant \(c_0^{(0)}>0\). The continuous local stability proof applies
in \(Y_L^{(0)}\), since the inclusion \(Y_L\hookrightarrow Y_L^{(0)}\) is
bounded. This inclusion also transfers the derivative convergence in
Corollary~\ref{cor:operator-norm-convergence-discrete-derivative} and the
uniform continuity estimate \eqref{eq:discrete-derivative-uniform-continuity}
to \(Y_L^{(0)}\). Repeating the proof of
Theorem~\ref{thm:local-stability-discrete-finite-test-map} therefore gives
\eqref{eq:l2-discrete-local-stability} on a neighborhood independent of \(h\).
The residual-comparison and surrogate
arguments then apply in the same norm.
\end{proof}

The stability constant and neighborhood are independent of \(h\) for
sufficiently fine meshes, but depend on the reference point, the latent map,
and the selected tests. The argument uses injectivity on the
finite-dimensional tangent space; it does not require an inverse estimate
relating the two norms on the full discrete trace space.

\paragraph{Scope of the Discrete and Surrogate Theory.}
\label{par:scope-of-the-discrete-and-surrogate-theory}
The finite element conclusions require the standard approximation property and hold for sufficiently small \(h\). Both the stability constant and neighborhood can be chosen independently of \(h\), but depend on the reference point, the latent map, and the selected Neumann tests. The best-approximation estimates yield explicit rates only when combined with additional solution regularity and interpolation estimates. For each fixed \(h\), stability of the surrogate itself is conditional on the coordinate-space \(C^1\)-approximation property \eqref{eq:c1-approximation-coordinate-map}; the analysis does not establish this property for a particular network architecture or training procedure, nor does it provide a mesh-uniform width bound or a joint relation between \(M\) and \(h\). The stability conclusions are local about \(z^\dagger\) within the latent model. Remark~\ref{rem:representation-mismatch-error} allows a true coefficient outside that model by including its discrepancy from a comparison coefficient in the same local stability set.

%% file: sec-material-rev0.tex
\section{Material Representation Maps}
\label{sec:material-representations}

We now describe three choices of representation map for the material coefficient. We use the latent variable \(z\in\mathcal{Z}\subset\mathbb{R}^m\) and representation map \(\Phi:\mathcal{Z}\to\mathcal{A}_{\mathrm{ad}}\) introduced in \eqref{eq:representation-map}, with \(a=\Phi(z)\) as in \eqref{eq:coefficient-from-representation-map}. The resulting continuous, discrete, and surrogate finite-test maps are \(N_L^\Phi\), \(N_{L,h}^\Phi\), and \(N_{L,h,M}^\Phi\), respectively. The representation-dependent hypotheses in the local stability results of Sections~\ref{sec:continuous-theory-for-latent-inversion-and-compressed-observations} and~\ref{sec:discrete-theory-for-the-latent-forward-map} are that \(\Phi\) satisfies Assumptions (A1) and (A2) and that the selected Neumann tests yield a positive-definite cumulative Gram matrix. Thus, admissibility, \(C^1\)-regularity, and full rank at the reference point must be verified for each particular representation.

\paragraph{Analytical Parametrization.}
\label{par:analytical-parametrization}
An analytical parametrization describes the material by a small number of explicit, interpretable parameters \(\mu\in\mathcal{Z}\subset\mathbb{R}^m\). We identify the latent variable with these parameters, so that \(z=\mu\), and write the representation map as \(\Phi(\mu)\).

Typical examples include piecewise constant coefficients, affine expansions, and constitutive laws involving a few scalar parameters. For smooth parametrizations, admissibility, regularity, and full rank can often be checked explicitly. Parametrizations that move sharp interfaces, however, are generally not \(C^1\) as maps into \(L^\infty(\Omega)\). In particular, the binary analytical ellipse model in Section~\ref{sec:numerical-experiments} is used as a numerical illustration outside Assumption (A1), as discussed there.

\paragraph{Linear Representation.}
\label{par:linear-latent-representation}
As an intermediate model between analytical parametrizations and nonlinear learned descriptions, consider
\begin{equation}
\Phi(z)=\bar a+\sum_{j=1}^m z_j\psi_j
\qquad
z\in\mathcal{Z}\subset\mathbb{R}^m
\label{eq:linear-latent-representation}
\end{equation}
where \(\bar a,\psi_1,\ldots,\psi_m\in L^\infty(\Omega)\). The latent set \(\mathcal{Z}\) must be chosen so that \(\Phi(\mathcal{Z})\subset\mathcal{A}_{\mathrm{ad}}\). This model includes reduced-basis, dictionary, and principal-component-type representations. Its \(C^1\)-regularity is immediate, and Assumption (A2) holds precisely when the modes \(\psi_1,\ldots,\psi_m\) are linearly independent.

\begin{cor}[Affine families near a constant reference in two dimensions]
\label{cor:constant-reference-two-dimensional-family}
Let \(d=2\), let \(\bar a=a_0\) be constant with
\(a_{\min}<a_0<a_{\max}\), and let \(\psi_1,\ldots,\psi_m\) be linearly
independent in \(L^\infty(\Omega)\). For \(\mathcal Z\) a sufficiently small
open ball about zero, \eqref{eq:linear-latent-representation} satisfies
(A1), (A2), and the separation condition
\eqref{eq:separation-condition-tangent-space} at \(z^\dagger=0\) with
\(\mathcal G=H^{-1/2}_\diamond(\partial\Omega)\). Consequently, at most
\(m\) Neumann excitations yield local Lipschitz stability near zero. The
discrete stability conclusions hold under the approximation assumption of
Theorem~\ref{thm:local-stability-discrete-finite-test-map}.
\end{cor}

\begin{proof}
Admissibility follows by taking \(\mathcal Z\) small, and linear independence
gives (A2). To verify separation, we use the constant-reference argument of
Calder\'on \cite{Calderon1980}. For \(0\ne k\in\mathbb R^2\), choose
\(\xi\in\mathbb R^2\) with \(\xi\cdot k=0\) and \(|\xi|=|k|\). The
functions \(u_\pm(x)=e^{(ik\pm\xi)\cdot x/2}\) are harmonic and satisfy
\begin{equation}
\nabla u_+\cdot\nabla u_-=-\frac{|k|^2}{2}e^{ik\cdot x}
\label{eq:constant-reference-harmonic-products}
\end{equation}
Their conormal data \(a_0\nabla u_\pm\cdot n\) have zero
integral, and subtracting the boundary means makes their real and imaginary
parts admissible background states. If \(\delta a\) satisfies the
orthogonality in \eqref{eq:separation-condition-tangent-space}, bilinear
extension therefore gives
\begin{equation}
\int_\Omega\delta a(x)e^{ik\cdot x}\,dx=0
\quad \forall k\in\mathbb R^2\setminus\{0\}
\label{eq:constant-reference-fourier-separation}
\end{equation}
The zero extension of \(\delta a\) belongs to \(L^1(\mathbb R^2)\), so its
Fourier transform is continuous and vanishes everywhere. Thus \(\delta a=0\).
Theorems~\ref{thm:linearized-injectivity-on-the-tangent-space},
\ref{thm:finite-tests-from-linearized-injectivity}, and
\ref{thm:local-injectivity-from-finite-tests} give the remaining claims.
\end{proof}

This example includes piecewise constant modes on a fixed partition. The
selected excitations depend on the family and reference coefficient; the
corollary does not verify the prescribed tests or learned representations used
in the numerical experiments.

\paragraph{Nonlinear Decoder Representation.}
\label{par:nonlinear-decoder-representation}
In the learned setting, a decoder \(D_\vartheta:\mathcal{Z}\to L^\infty(\Omega)\) generates a normalized coefficient field. We set
\begin{align}
\widehat a(z)&=D_\vartheta(z)
\qquad
\text{with}
\qquad
0\leq \widehat a(z)\leq 1\quad\text{a.e. in }\Omega
\label{eq:normalized-nonlinear-decoder-output}\\
\Phi(z)&=a_{\min}+(a_{\max}-a_{\min})\widehat a(z)
\label{eq:nonlinear-decoder-representation}
\end{align}
The decoder is obtained in an offline stage from a family of material samples and is then held fixed during inversion. The bounded output in \eqref{eq:normalized-nonlinear-decoder-output} enforces \(\Phi(\mathcal{Z})\subset\mathcal{A}_{\mathrm{ad}}\). If \(D_\vartheta\) is \(C^1\), then \(\Phi\) satisfies Assumption (A1), while Assumption (A2) additionally requires \(D\Phi(z^\dagger)\) to have full column rank. These properties are not automatic consequences of training.

\paragraph{Connection to the Numerical Experiments.}
Section~\ref{sec:numerical-experiments} uses the analytical and nonlinear-decoder representations and discusses their relationship to the regularity assumptions. The linear representation above provides a direct example of the abstract framework.

%% file: sec-numerical-experiments-rev0.tex
\section{Numerical Experiments}
\label{sec:numerical-experiments}
We illustrate the proposed method by reconstructing a material coefficient \(a:\Omega\to\mathbb{R}\) from finitely many Neumann boundary excitations, each producing a full discrete Dirichlet trace. All experiments were implemented in JAX \cite{bradbury2018jax}, and the machine learning models were trained on an NVIDIA RTX PRO 2000 Blackwell Generation GPU.

\paragraph{Scope of the Numerical Experiments.}
\label{par:numerical-experiment-scope}
The theoretical results are conditional on local differentiability and nondegeneracy assumptions. The experiments below are intended to illustrate the computational reconstruction pipeline and its cost, rather than to numerically verify the separation condition or its sufficient CGO criterion. In particular, all computations are two-dimensional and use pixelwise piecewise constant coefficient fields, so the smooth, \(d\ge 3\) CGO result in Appendix~\ref{sec:cgo-verification-of-the-separation-condition} is not invoked in the numerical examples.

The examples cover different regularity regimes. In the learned ellipse experiment, the ELU decoder and surrogate define \(C^1\) maps between finite-dimensional spaces and are therefore compatible with the regularity framework, although the derivative nondegeneracy condition and the required \(C^1\)-surrogate accuracy are not verified numerically. The binary analytical ellipse map is not \(C^1\) as a map into \(L^\infty(\Omega)\), while the ReLU-based decoder and surrogate used for the crack model are not globally \(C^1\). These cases are included as empirical nonsmooth tests outside the global assumptions of the analysis. Moreover, targets drawn from the underlying coefficient families need not lie exactly in the image of the trained decoder, so the learned experiments may also contain the representation error discussed in the introduction.

\subsection{Problem Setup}
\label{subsec:problem-setup}
\paragraph{Discretization.}
Throughout this section, we consider the two-dimensional square domain \(\Omega=[-1,1]^2\subset\mathbb{R}^2\). For a fixed integer \(p\geq 1\), we partition \(\Omega\) into a uniform
\(p\times p\) Cartesian grid. Denoting the resulting collection of square
cells by \(\mathcal T_p=\{Q_{ij}\}_{i,j=1}^p\), each cell \(Q_{ij}\) has side
length \(2/p\). We represent the material coefficients as
piecewise constant functions on the partition \(\mathcal T_p\). We therefore
introduce the admissible pixel coefficient set
\begin{equation}
X_p
=
\left\{
a\in \mathcal{A}_{\mathrm{ad}}
:
a|_{Q_{ij}}
\text{ is constant for every }
Q_{ij}\in\mathcal T_p
\right\}\label{eq:numerics-pixel-coefficient-space}
\end{equation}
Every coefficient \(a\in X_p\) is uniquely determined by its cell values
\(a_{ij}=a|_{Q_{ij}}\), \(i,j=1,\ldots,p\), and may therefore be identified
with a vector in \([a_{\min},a_{\max}]^{p^2}\). Equivalently, we may regard \(a\) as a
\(p\times p\) pixel image. The admissible coefficient families considered in the numerical experiments form structured subsets of \(X_p\) that the representation map \(\Phi\) is intended to capture. Throughout the experiments, we set \(p=64\), so the ambient pixel space has dimension \(p^2=4\,096\).

Given a coefficient \(a\) and boundary datum \(g_N\), the forward Neumann problem is discretized by the Galerkin finite element method described 
in Section~\ref{subsec:finite-element-method}. We construct a regular triangulation \(\mathcal{T}_h\) 
of \(\Omega\) by diagonally splitting each square cell in \(\mathcal T_p\) into two triangles. For \(p=64\), the square-grid spacing is \(1/32\), and the resulting right-triangular elements have diameter \(\sqrt{2}/32\). We use continuous, piecewise linear finite elements on \(\mathcal T_h\). The boundary-mean-zero constraint defining \(V_h^0\) in \eqref{eq:discrete-mean-zero-space} is enforced with a Lagrange multiplier.

\paragraph{Boundary Measurements.}
In the numerical experiments, the boundary data consist of finitely many
Neumann patterns \((g_N^\ell)_{\ell=1}^L\) imposed on
\(\partial\Omega\), together with the Dirichlet traces of the corresponding discrete finite element solutions \(u_h^\ell\). Since \(\Omega=[-1,1]^2\), we parametrize \(\partial\Omega\)
counterclockwise by arclength \(s\in[0,8]\) starting at the corner \((-1,-1)\). For \(\ell=1,\ldots,L\), we define the Neumann patterns by
\begin{equation}
g_N^{\ell}(s)=
\begin{cases}
\displaystyle
\sin\!\left(\frac{2\pi \lceil \ell/2\rceil s}{8}\right)
& \text{if } \ell \text{ is odd} \\[1ex]
\displaystyle
\cos\!\left(\frac{2\pi \lceil \ell/2\rceil s}{8}\right)
& \text{if } \ell \text{ is even}
\end{cases}\label{eq:numerics-neumann-modes}
\end{equation}
These \(L\) Fourier modes are smooth, orthogonal in \(L^2(\partial\Omega)\), have zero mean over \(\partial\Omega\), and probe low-to-moderate frequencies.

For each excitation, the discrete Dirichlet trace is stored as the vector of
FEM solution values at all boundary nodes. Collecting these vectors gives the
nodal coordinate map \(\Lambda_h\) from
\eqref{eq:method-coordinate-isomorphism-discrete-output}. The surrogate predicts
all these nodal values without imposing a zero boundary mean on its output.

\paragraph{Dataset Independence.}
Within each coefficient family, the coefficient samples used for the VAE
dataset, surrogate-data construction, and reconstruction targets are independent
draws from the same coefficient model, generated using distinct random seeds.
The VAE and surrogate have separate training and validation splits, with sizes
specified below. Reconstruction targets are not selected from any training or
validation set.

\subsection{Variational Autoencoder for Material Representation}
\label{subsec:numerics-vae}
We construct a material representation map
\begin{equation}
\Phi:\mathcal Z\to X_p \subset\mathcal A_{\mathrm{ad}}\label{eq:numerics-material-representation-map}
\end{equation}
by training a variational autoencoder on coefficient samples. A variational autoencoder (VAE) is a generative latent-variable model that learns a probabilistic low-dimensional representation of a dataset. It consists of an encoder network \(E_{\upsilon}\) and a decoder network \(D_\vartheta\) with trainable parameters \(\upsilon\) and \(\vartheta\), respectively. For the VAE, we use normalized pixel images in the set
\(\widehat X_p\simeq[0,1]^{p^2}\). In the VAE formulas below, \(a\) denotes
a normalized input image; the positive offset defining the physical coefficient
is introduced in \eqref{eq:numerics-representation-map}. Given such an input
\(a\), the encoder parametrizes a probability distribution on \(\mathcal{Z}=\mathbb{R}^m\) by
\begin{equation}
q_\upsilon(z\mid a)=\mathcal N(\mu_\upsilon(a),\operatorname{diag}(\sigma_\upsilon^2(a)))\label{eq:numerics-latent-distribution}
\end{equation}
A latent variable is sampled using the reparameterization
\begin{equation}
z=\mu_\upsilon(a)+\sigma_\upsilon(a)\odot\varepsilon ,
\qquad
\varepsilon\sim\mathcal N(0,I)\label{eq:numerics-vae-reparameterization}
\end{equation}
and the decoder maps this latent variable back to a reconstruction
\(\widehat{a}=D_\vartheta(z)\).
We set the latent dimension \(m\) much smaller than the ambient pixel dimension \(4\,096\). The encoder and decoder are trained so that the reconstruction \(\widehat{a}\) is close to the input \(a\), while the latent distribution \(q_\upsilon(z\mid a)\) remains close to the standard Gaussian distribution in the low-dimensional latent space \(\mathcal{Z}\). In our setting, the VAE can be viewed as the composition
\begin{equation}
\widehat X_p \stackrel{E_\upsilon}{\longrightarrow} \mathcal{Z} \stackrel{D_\vartheta}{\longrightarrow} \widehat X_p\label{eq:numerics-vae-composition}
\end{equation}
by identifying \(E_{\upsilon}(a)=\mu_\upsilon(a)\). Samples are generated by drawing latent vectors from the standard Gaussian distribution and mapping them to \(\widehat X_p\) with the decoder \(D_\vartheta\). Figure~\ref{fig:vae-with-samples} illustrates the encoding, sampling, and decoding steps.

\begin{figure}
    \centering
    \includegraphics[width=0.75\textwidth]{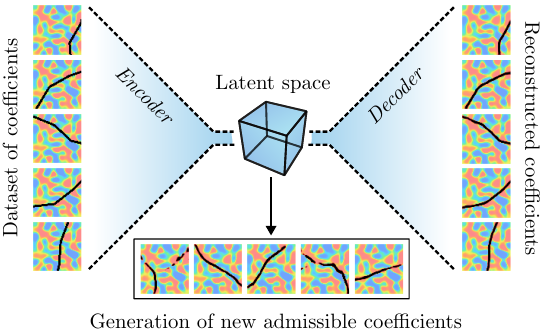}
    \caption{Schematic of the variational autoencoder (VAE). The encoder \(E_\upsilon\) maps a normalized input image \(a\in\widehat X_p\) to the parameters of a distribution on \(\mathcal Z\), from which a latent representation \(z\) is sampled. The decoder \(D_\vartheta\) maps \(z\) to a reconstructed normalized image \(\widehat a\in\widehat X_p\). New normalized images are generated by sampling from the latent prior and applying the decoder.}
    \label{fig:vae-with-samples}
\end{figure}

\paragraph{Admissibility.}
Throughout the numerical experiments, we take \(a_{\max}=a_{\min}+1\) and design the VAE so that the decoder outputs a pixel image \(\widehat a=D_\vartheta(z)\) with pixel values in the interval \([0,1]\). We then define the physical representation map by
\begin{equation}
    \Phi(z)
    =
    a_{\min}
    +
    D_\vartheta(z)\label{eq:numerics-representation-map}
\end{equation}
where \(a_{\min}>0\) is fixed. Consequently, \(a_{\min}\leq\Phi(z)\leq a_{\max}\) almost everywhere in \(\Omega\), so \(a=\Phi(z)\) belongs to the pixel coefficient set \(X_p\), and hence to the full admissible class \(\mathcal{A}_{\mathrm{ad}}\). Accordingly, we use datasets of pixel images with values in \([0,1]\) to train the VAEs.

\begin{rem}
    Throughout the numerical experiments, the physical coefficient used in the forward and inverse problems includes the positive offset \(a_{\min}\), so that
    \begin{equation}
    a = a_{\min}+\widehat a\label{eq:numerics-physical-coefficient-offset}
    \end{equation}
    where \(\widehat a\) denotes the coefficient field represented by either an analytical parametrization or, in the learned setting, by the decoder output. To use a common visualization scale, all plots of coefficient fields display only \(\widehat a=a-a_{\min}\), so the constant offset \(a_{\min}\) is omitted. Thus, the values shown in the figures lie in \([0,1]\), while the corresponding physical coefficients used in the computations are obtained by adding \(a_{\min}\).
    \end{rem}

\paragraph{Architecture.}
The encoder consists of four convolutional blocks with \(4\times4\) kernels and stride two, reducing the spatial resolution according to
\begin{equation}
64\times64
\;\rightarrow\;
32\times32
\;\rightarrow\;
16\times16
\;\rightarrow\;
8\times8
\;\rightarrow\;
4\times4\label{eq:numerics-vae-encoder-resolution}
\end{equation}
while increasing the number of feature channels from \(1\) to
\(32\), \(64\), \(128\), and \(256\), respectively. Each convolutional layer
is followed by an activation function. The resulting \(4\times4\times256\) feature
tensor is flattened and processed by two fully connected layers with
hidden dimensions \(512\) and \(256\). Two linear output heads then produce
the mean vector \(\mu_\upsilon(a)\in\mathbb R^m\) and log-variance vector
\(\log\sigma_\upsilon^2(a)\in\mathbb R^m\), defining the Gaussian latent distribution \(q_\upsilon(z|a)\) in \eqref{eq:numerics-latent-distribution}.

The decoder is constructed symmetrically. Starting from a latent vector
\(z\in\mathbb R^m\), two fully connected layers expand the representation to a
\(4\times4\times256\) feature tensor. Four decoder blocks then progressively
increase the spatial resolution by nearest-neighbor upsampling followed by
\(3\times3\) convolutions, with channel dimensions
\(256\rightarrow128\rightarrow64\rightarrow32\rightarrow16\). An activation function is applied after each decoder convolution.
A final \(3\times3\) convolution maps the \(16\)-channel feature map to a single output channel, and a sigmoid activation restricts the reconstructed pixel values to the interval \([0,1]\). In the numerical experiments, we use the rectified linear unit (ReLU) and the exponential linear unit (ELU) as activation functions in the convolutional and fully connected layers.

\paragraph{Training Objective.}
The VAE is trained by minimizing an objective consisting of three terms: a pixelwise reconstruction error, an edge-aware
gradient matching term that promotes accurate recovery of sharp interfaces,
and a Kullback--Leibler regularization term that encourages the latent
distribution to remain close to the standard Gaussian distribution. More precisely, if
\(\{a_i\}_{i=1}^{N_{\mathrm{tr}}}\) is a training set of images and \(\{\widehat{a}_i\}_{i=1}^{N_{\mathrm{tr}}}\) are the reconstructions,
the loss functional takes the form
\begin{equation}
     L^{\mathrm{VAE}}(\upsilon,\vartheta)
    =\frac{1}{N_{\mathrm{tr}}}\sum_{i=1}^{N_{\mathrm{tr}}}
    \bigg(\mathcal L_{\mathrm{rec}}(a_i,\widehat{a}_i)
    +
    \lambda^{\mathrm{edge}}
    \mathcal L_{\mathrm{edge}}(a_i,\widehat{a}_i)
    +
    \beta
    \mathcal L_{\mathrm{KL}}(a_i)\bigg)\label{eq:numerics-vae-loss}
\end{equation}
where \(\lambda^{\mathrm{edge}}\) and \(\beta\) are positive weights. The three terms are defined by
\begin{align}
    \mathcal L_{\mathrm{rec}}(a,\widehat{a})
    &=\frac{1}{p^2}
    \|a-\widehat{a}\|_2^2\label{eq:numerics-vae-reconstruction-loss}\\
    \mathcal L_{\mathrm{edge}}(a,\widehat{a})
    &=\frac{1}{p^2}\bigg(
    \|\Delta_x a-\Delta_x \widehat{a}\|_2^2 + \|\Delta_y a-\Delta_y \widehat{a}\|_2^2\bigg)\label{eq:numerics-vae-edge-loss}\\
    \mathcal L_{\mathrm{KL}}(a)
    &=
    D_{\mathrm{KL}}
    \!\left(
    q_\upsilon(z\mid a)\,\|\,\mathcal N(0,I)
    \right)\label{eq:numerics-vae-kl-loss}
\end{align}
where the horizontal and vertical forward differences on the pixel grid \(\mathcal T_p\) are defined by
\((\Delta_x a)_{i,j}=a_{i,j+1}-a_{i,j}\) and
\((\Delta_y a)_{i,j}=a_{i+1,j}-a_{i,j}\), respectively. Here, \(D_{\mathrm{KL}}\) is the Kullback--Leibler divergence between the two Gaussian distributions in \eqref{eq:numerics-vae-kl-loss}.

\subsection{Neural Forward Surrogate}
\label{subsec:numerics-surrogate}
The purpose of the surrogate is to avoid repeated expensive finite element solves during
online inversion. The finite-test latent forward map \(N^{\Phi}_{L,h}:\mathcal Z\to Y_{L,h}\) and its coordinate representation \(\widehat N^{\Phi}_{L,h}:\mathcal Z\to\mathbb R^{P_h}\) are defined in \eqref{eq:training-target-finite-test-map} and \eqref{eq:method-coordinate-representation-discrete-forward-map}, respectively. The surrogate architecture is given in \eqref{eq:two-layer-network-discrete-forward-map}. In the experiments, we use either ReLU or ELU activation functions.

\paragraph{Training Data.}
Assume that a VAE has been pretrained for a coefficient family and that \(\Phi\) is defined by \eqref{eq:numerics-representation-map}. Given a sample \(\{a_i\}_{i=1}^{M_{\mathrm{tr}}}\) of normalized coefficients from this family, we use the encoder \(E_\upsilon\) to obtain latent vectors \(\{z_i\}_{i=1}^{M_{\mathrm{tr}}}\). For each \(z_i\), we compute the target output \(y_i=\widehat N^\Phi_{L,h}(z_i)\). The surrogate training set is therefore \(\{(z_i,y_i)\}_{i=1}^{M_{\mathrm{tr}}}\).

We use an overparameterized surrogate with more hidden units than training pairs. Specifically, we set
\begin{equation} 
M=4M_{\mathrm{tr}}\label{eq:numerics-overparametrization}
\end{equation}

Given the training data \(\{(z_i,y_i)\}_{i=1}^{M_{\mathrm{tr}}}\), the network parameters are determined by minimizing the mean squared prediction error over the training set. Specifically, we consider
\begin{align}
\theta^\star
&\in\operatorname*{arg\,min}_{\theta}L^{\mathrm{Net}}(\theta)\label{eq:numerics-nn-minimizer}\\
L^{\mathrm{Net}}(\theta)
&=\frac{1}{P_h M_{\mathrm{tr}}}\sum_{i=1}^{M_{\mathrm{tr}}}
\|\widehat{N}^{\Phi}_{L,h,M}(z_i)-y_i\|_{\mathbb{R}^{P_h}}^2\label{eq:numerics-nn-least-squares-problem}
\end{align}
Here, \(\theta\) collects the trainable parameters of \(\widehat{N}^{\Phi}_{L,h,M}\), and \(\theta^\star\) denotes a minimizer. This is the coordinate form of training against the discrete boundary response map \(N^{\Phi}_{L,h}\). Since the surrogate approximates boundary responses, we assess its accuracy primarily through the boundary-data misfit.

When an analytical parametrization \(a=\Phi(\mu)\) is available, we instead sample the parameter space and map each parameter \(\mu_i\) to \(X_p\) using \(\Phi\). We then compute \(y_i=\widehat N^{\Phi}_{L,h}(\mu_i)\) and train the surrogate on the resulting parameter--response pairs.

\subsection{Inverse Problem Objectives}
\label{subsec:numerics-inverse-formulations}

Let \(a^\dagger\) denote the target coefficient. We use the corresponding
noiseless discrete boundary observations \(\datah\) defined in
\eqref{eq:discrete-data-vector}. Corollary~\ref{cor:l2-observation-stability} establishes local
stability directly in the boundary \(L^2(\partial\Omega)^L\)-norm, with a
constant independent of \(h\) on sufficiently fine meshes under the stated
hypotheses. This avoids the mesh-dependent conversion from the
\(H^{1/2}\)-norm. The required nondegeneracy and approximation conditions are
not certified for the numerical tests below. We approximate the target
coefficient by minimizing the surrogate objective
\begin{equation}
J^{\mathrm{Net}}(z)
=
\frac12
\left\|
N^{\Phi}_{L,h,M}(z)-\datah
\right\|_{L^2(\partial\Omega)^L}^2\label{eq:numerics-boundary-objective-surrogate}
\end{equation}
using gradient descent with backtracking line search. In accordance with \eqref{eq:discrete-objective-functional-finite-tests}, we also introduce the loss

\begin{equation}
    J^{\mathrm{FEM}}(z)
    =
    \frac12
    \left\|
    N^{\Phi}_{L,h}(z)-\datah
    \right\|_{L^2(\partial\Omega)^L}^2\label{eq:numerics-boundary-objective-fem}
\end{equation}
for reconstruction with the FEM forward map. We omit regularization in both objectives. To minimize the FEM objective, we differentiate the discrete boundary response map with respect to the latent variable \(z\), using JAX automatic differentiation through the finite element solver.

Both boundary objectives are evaluated using the boundary mass matrix
\((M_{\partial})_{ij}=(\varphi_i,\varphi_j)_{\partial\Omega}\), where
\(\{\varphi_i\}\) is the nodal basis of the piecewise linear trace space.
Writing \(r_\star^\ell(z)\) for the vector of predicted minus observed
boundary nodal values in experiment \(\ell\), we have
\begin{equation}
J^\star(z)=\frac12\sum_{\ell=1}^L
r_\star^\ell(z)^\top M_{\partial}r_\star^\ell(z)
\qquad \star\in\{\mathrm{Net},\mathrm{FEM}\}
\label{eq:numerics-boundary-mass-matrix-objective}
\end{equation}
Each quadratic form is the squared boundary \(L^2\)-norm of the corresponding
piecewise linear residual, including any error in the surrogate's predicted
boundary mean.

We consider two families of piecewise constant coefficients. The first consists of elliptic inclusions; the second combines a smooth random background field with randomly generated localized crack-like defects.

\subsection{Elliptic Inclusions}
\label{subsec:elliptic-inclusions}

The ellipse model is a low-complexity class of piecewise constant
coefficients on the uniform pixel grid \(\mathcal{T}_p\) of \(\Omega\). The elliptic inclusions are parametrized by
\(
\mu=(c_x,c_y,r_1,r_2,\alpha),
\)
where \((c_x,c_y)\) is the ellipse center, \(r_1\) and \(r_2\) denote the
lengths of the principal semi-axes, and \(\alpha\) specifies the angle of
rotation relative to the coordinate axes. More precisely, the elliptic inclusions are defined by
\begin{equation}
E(\mu)
=
\left\{
(x,y):
\frac{\xi_1(x,y)^2}{r_1^2}
+
\frac{\xi_2(x,y)^2}{r_2^2}
\le 1
\right\}\label{eq:numerics-ellipse-set}
\end{equation}
where
\(\xi_1(x,y)
=(x-c_x)\cos\alpha+(y-c_y)\sin\alpha\)
and
\(\xi_2(x,y)
=-(x-c_x)\sin\alpha+(y-c_y)\cos\alpha\).
For a parameter vector \(\mu\in \IR^5\) with positive semi-axes, we obtain the corresponding physical coefficient field \(a\in X_p\) by defining its value on each pixel \(Q_{ij}\) according to
\begin{equation}
a|_{Q_{ij}}=\Phi(\mu)|_{Q_{ij}}=a_{\min}+\begin{cases} 1 & \text{if } c_{ij}\in E(\mu) \\ 0 & \text{otherwise} \end{cases}\label{eq:numerics-ellipse-coefficient-field}
\end{equation}
where \(c_{ij}\) is the center of \(Q_{ij}\). In all experiments, we use \(a_{\min}=0.05\). Recall that the constant offset \(a_{\min}\) is omitted from the coefficient plots.

\paragraph{Sampling.}
We draw \(r_1\) and \(r_2\) independently and uniformly from \([0.12,0.45]\),
without an ordering constraint, and independently draw
\(\alpha\) uniformly from \([-\pi,\pi]\). Conditional on the semi-axes and
angle, the center is sampled uniformly from the rectangle of centers for which
the entire rotated ellipse is contained in \(\Omega\). The same parameter law
is used for the coefficient samples underlying the VAE and surrogate datasets,
the reconstruction targets, and the analytical initial guesses.

Examples of coefficient fields from this family are shown in Figure~\ref{fig:numerics-ellipses-examples}. 

This family serves two purposes. First, it provides a transparent benchmark in
which the relevant degrees of freedom are geometrically interpretable.
Second, it allows us to compare the latent reconstruction framework with a
classical finite-dimensional analytical parametrization. In this case the representation map \(\Phi\)
can either be defined analytically by \eqref{eq:numerics-ellipse-coefficient-field} or learned from realizations of the coefficient family. In all experiments involving this family, we use \(L=6\) boundary experiments with Neumann data defined by \eqref{eq:numerics-neumann-modes}.

\begin{figure}
    \centering
    \includegraphics[width=0.75\linewidth]{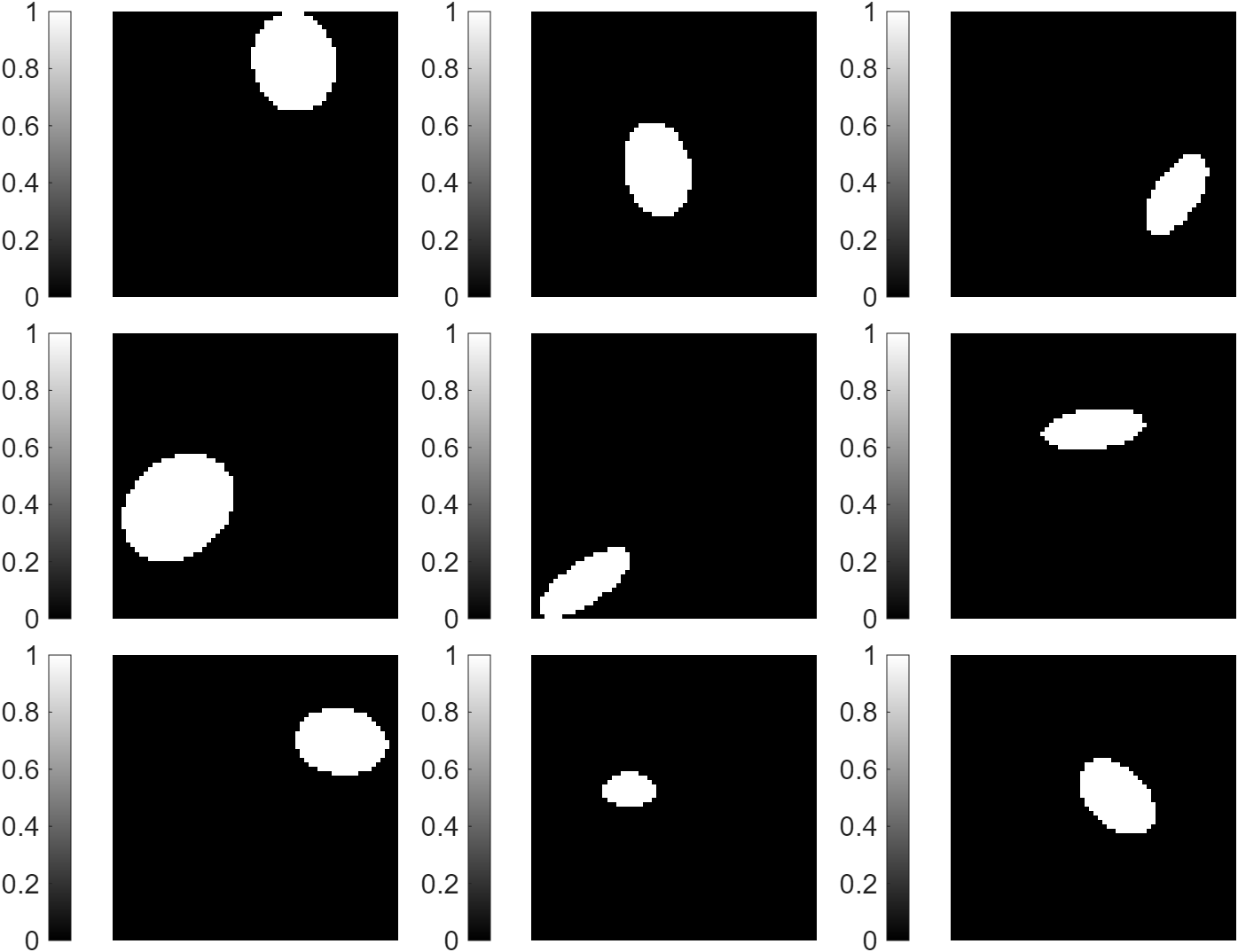}
    \caption{Examples of normalized coefficient contrasts \(\widehat a=a-a_{\min}\) from the elliptic-inclusion family.}
    \label{fig:numerics-ellipses-examples}
\end{figure}

\subsubsection{Analytical Parametrization}
In the first experiment, we train our surrogate network to map ellipse parameters \(\mu\) to the boundary responses corresponding to the associated coefficient fields given by \eqref{eq:numerics-ellipse-coefficient-field}. Thus, in this analytical model, the latent space is the ellipse parameter space, a subset of \(\IR^5\), and no VAE is used.

\paragraph{Training Details (Surrogate Network).}
To train the network, we generated \(5\,000\) ellipse parameter vectors and the corresponding boundary responses. We used \(M_{\mathrm{tr}}=4\,500\) pairs for training and held out \(500\) pairs for validation. In this test, we used the ELU activation function, and \eqref{eq:numerics-overparametrization} gives \(M=18\,000\) hidden units. We trained the network with batches of size \(100\) for \(10\,000\) epochs using the Adam optimizer with an initial learning rate of \(10^{-3}\), reduced by a factor of \(0.75\) every \(2\,000\) epochs. The training and validation loss curves are shown in Figure~\ref{fig:ellipse-analytical-network-loss}.

\begin{figure}
    \centering
    \includegraphics[width=0.4\linewidth]{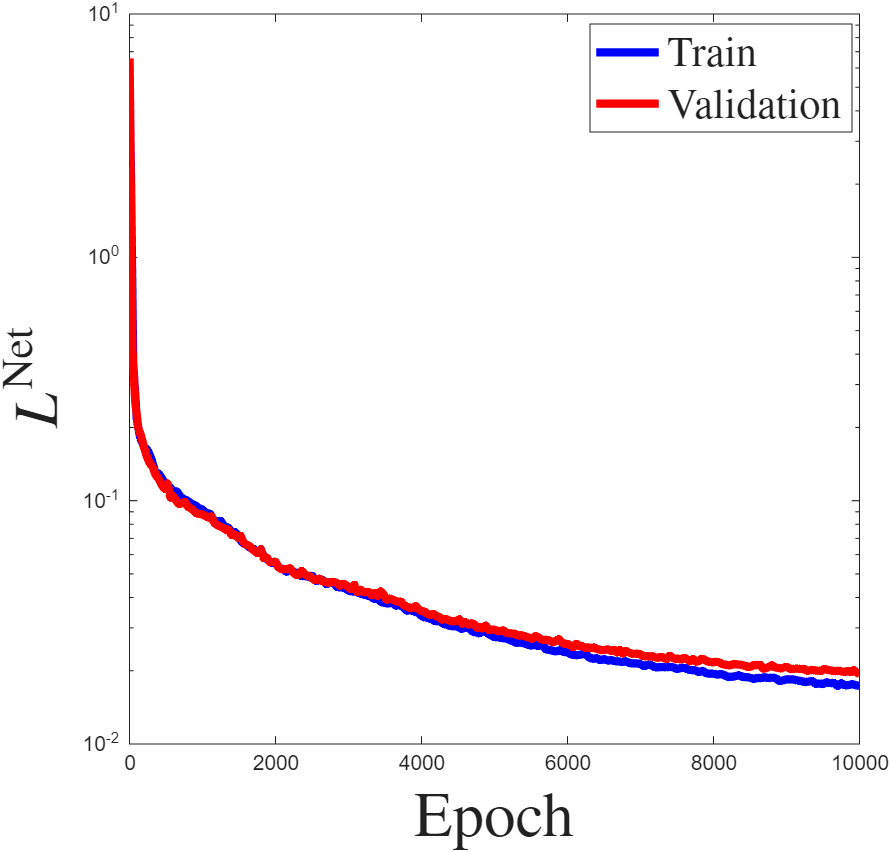}
    \caption{Training and validation losses of the surrogate network for the analytical ellipse parametrization.}
    \label{fig:ellipse-analytical-network-loss}
\end{figure}
\paragraph{Reconstruction.}
We generated two target coefficients by randomly sampling ellipse parameters. In this experiment, the representation map is given explicitly by \eqref{eq:numerics-ellipse-coefficient-field}. We use gradient descent with Armijo line search to minimize \(J^{\mathrm{Net}}(\mu)\), defined in \eqref{eq:numerics-boundary-objective-surrogate}. The optimization is performed in the parameter space \(\IR^5\), and the trained surrogate maps parameter vectors to boundary responses.

Each trial parameter vector is projected onto the coordinate bounds
\((c_x,c_y)\in[-1,1]^2\), \(r_1,r_2\in[0.12,0.45]\), and
\(\alpha\in[-\pi,\pi]\). The containment constraint used during sampling is
not enforced during inversion, so iterates may represent ellipses extending
beyond \(\Omega\).

To initialize the optimization, we randomly sample an ellipse parameter \(\mu^{(0)}\). Figures~\ref{fig:ellipse-sequence-1-analytical} and~\ref{fig:ellipse-sequence-2-analytical} show the reconstruction process for two targets. The figures display selected iterates and the final reconstructions after \(100\) gradient descent iterations. They also show the surrogate objective \(J^{\mathrm{Net}}(\mu)\), which is minimized, the corresponding FEM objective \(J^{\mathrm{FEM}}(\mu)\), and the \(L^2(\Omega)\)-error in the coefficient field along the iterates.

\begin{figure}
    \centering
    \includegraphics[width=\linewidth]{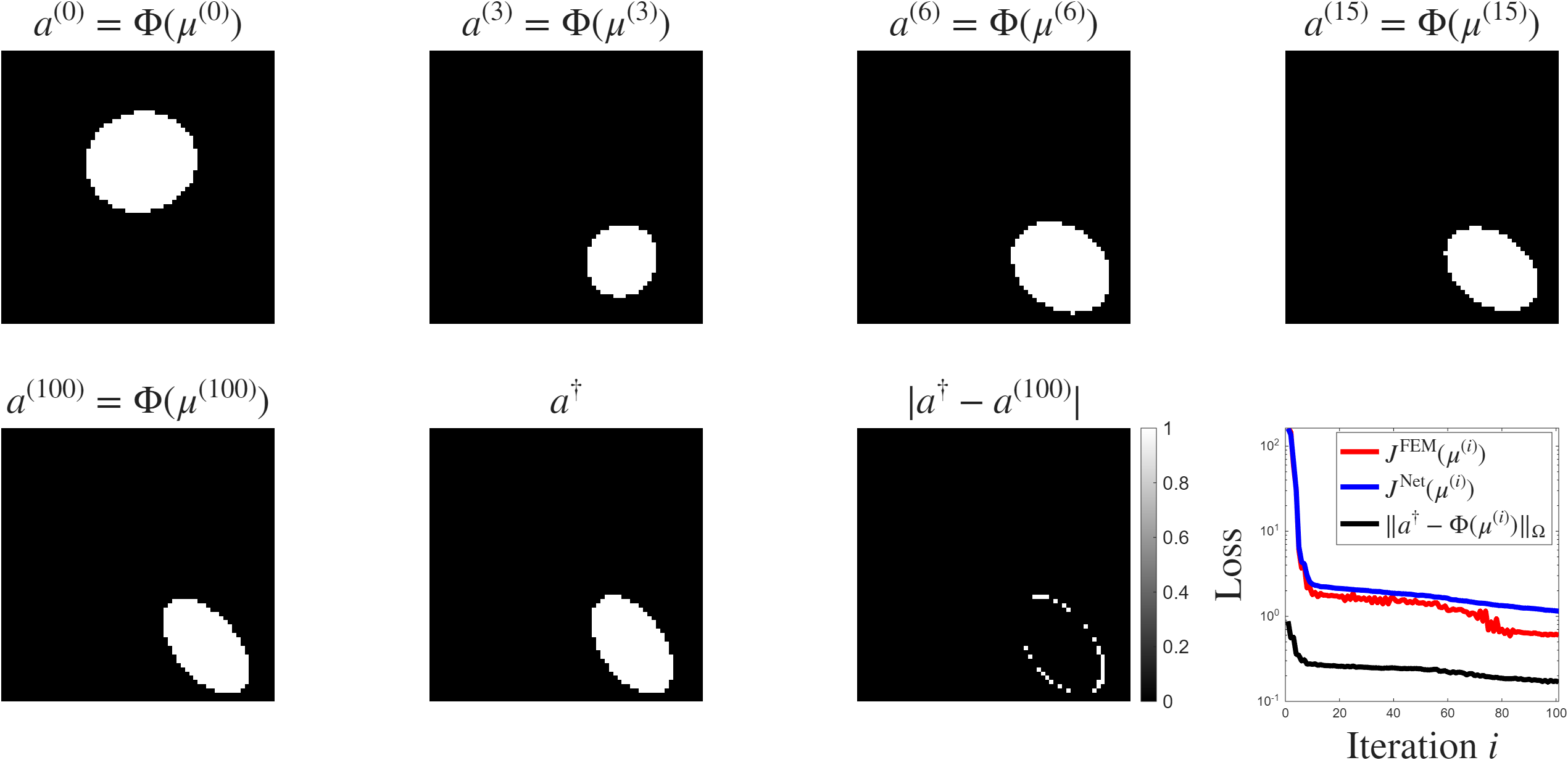}
    \caption{Surrogate-based reconstruction of a target coefficient \(a^\dagger\) using the analytical ellipse parametrization. \emph{Top row:} selected iterates \(a^{(i)}=\Phi(\mu^{(i)})\). \emph{Bottom row:} final reconstruction \(a^{(100)}\), target \(a^\dagger\), pointwise absolute error \(\lvert a^\dagger-a^{(100)}\rvert\), and histories of \(J^{\mathrm{Net}}(\mu^{(i)})\), \(J^{\mathrm{FEM}}(\mu^{(i)})\), and \(\lVert a^\dagger-a^{(i)}\rVert_{L^2(\Omega)}\).}
    \label{fig:ellipse-sequence-1-analytical}
\end{figure}

\begin{figure}
    \centering
    \includegraphics[width=\linewidth]{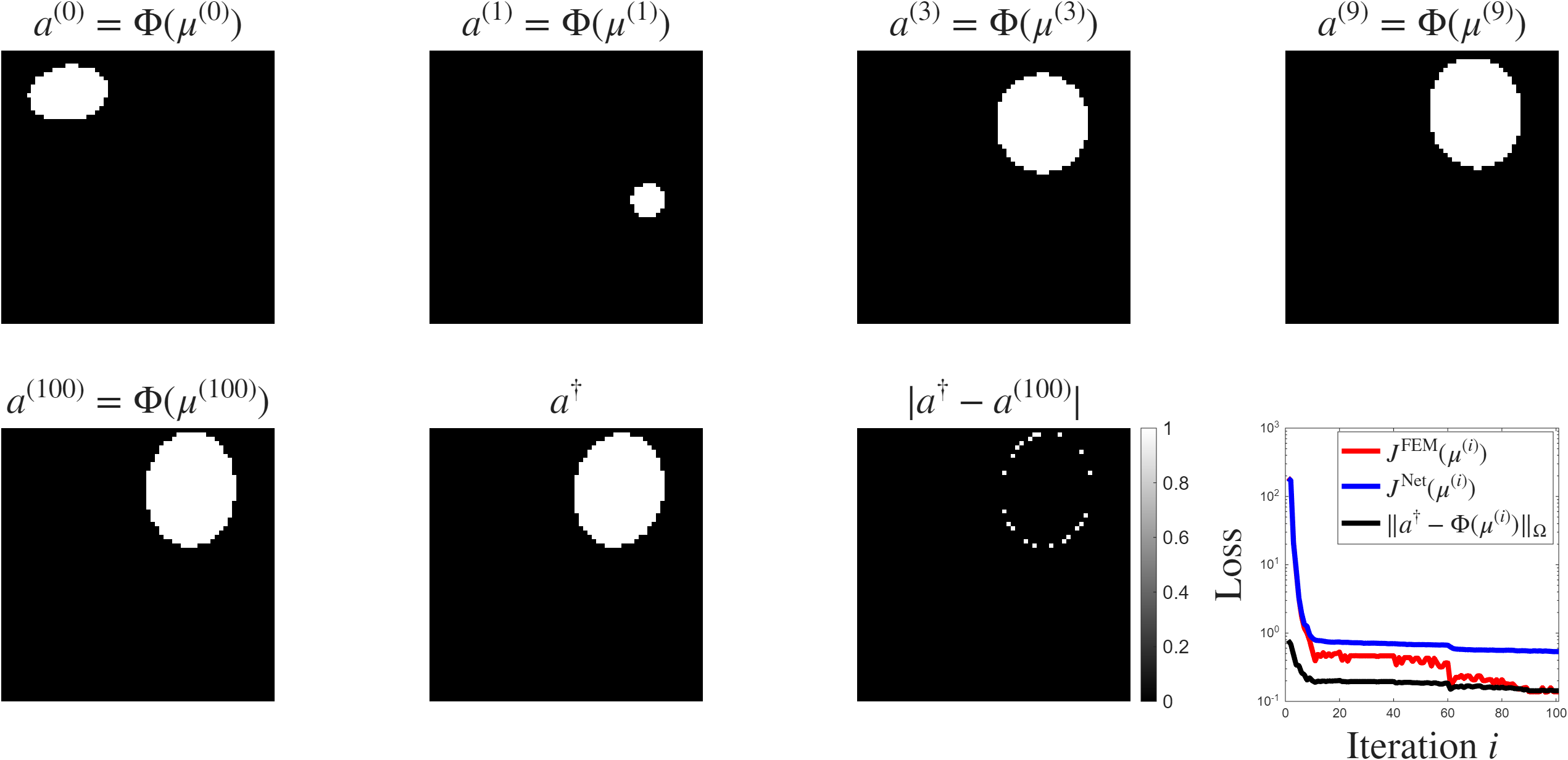}
    \caption{Surrogate-based reconstruction of another target coefficient \(a^\dagger\) using the analytical ellipse parametrization. \emph{Top row:} selected iterates \(a^{(i)}=\Phi(\mu^{(i)})\). \emph{Bottom row:} final reconstruction \(a^{(100)}\), target \(a^\dagger\), pointwise absolute error \(\lvert a^\dagger-a^{(100)}\rvert\), and histories of \(J^{\mathrm{Net}}(\mu^{(i)})\), \(J^{\mathrm{FEM}}(\mu^{(i)})\), and \(\lVert a^\dagger-a^{(i)}\rVert_{L^2(\Omega)}\).}
    \label{fig:ellipse-sequence-2-analytical}
\end{figure}
For the two displayed targets, the surrogate-based reconstructions closely match the target coefficients, with the visible discrepancies primarily confined to pixels near the ellipse boundaries. Along these optimization trajectories, the surrogate and FEM objectives follow similar trends, indicating agreement between the two objectives in the regions visited by the iterates.

\subsubsection{VAE Parametrization}
Next, we consider the setting in which the analytical parametrization of the elliptic inclusions is unavailable. Instead, we assume access only to sample coefficient fields generated from the underlying model. A variational autoencoder is trained on these samples to learn a low-dimensional representation, and the decoder defines the representation map \(\Phi\) in \eqref{eq:numerics-representation-map}. In this experiment, the latent dimension is \(m=6\), and we use the architecture described in Section~\ref{subsec:numerics-vae} with ELU activation functions.

\paragraph{Training Details (VAE).}
The VAE was trained for \(10\,000\) epochs using stochastic gradient descent (SGD) with a learning rate of \(0.1\) and a batch size of \(128\). To improve training stability, we used KL annealing, gradually increasing the weight \(\beta\) from \(10^{-8}\) to \(10^{-3}\) in the variational objective \eqref{eq:numerics-vae-loss}. The edge-aware weight \(\lambda^{\mathrm{edge}}\) in \eqref{eq:numerics-vae-loss} was fixed at \(1.0\) throughout training. The training set consisted of \(9\,000\) samples, with an additional \(1\,000\) samples reserved for validation. Figure~\ref{fig:VAE_ellipse} shows the annealing schedule and the training and validation loss curves.

\begin{figure}
    \centering
    \includegraphics[width=0.32\textwidth]{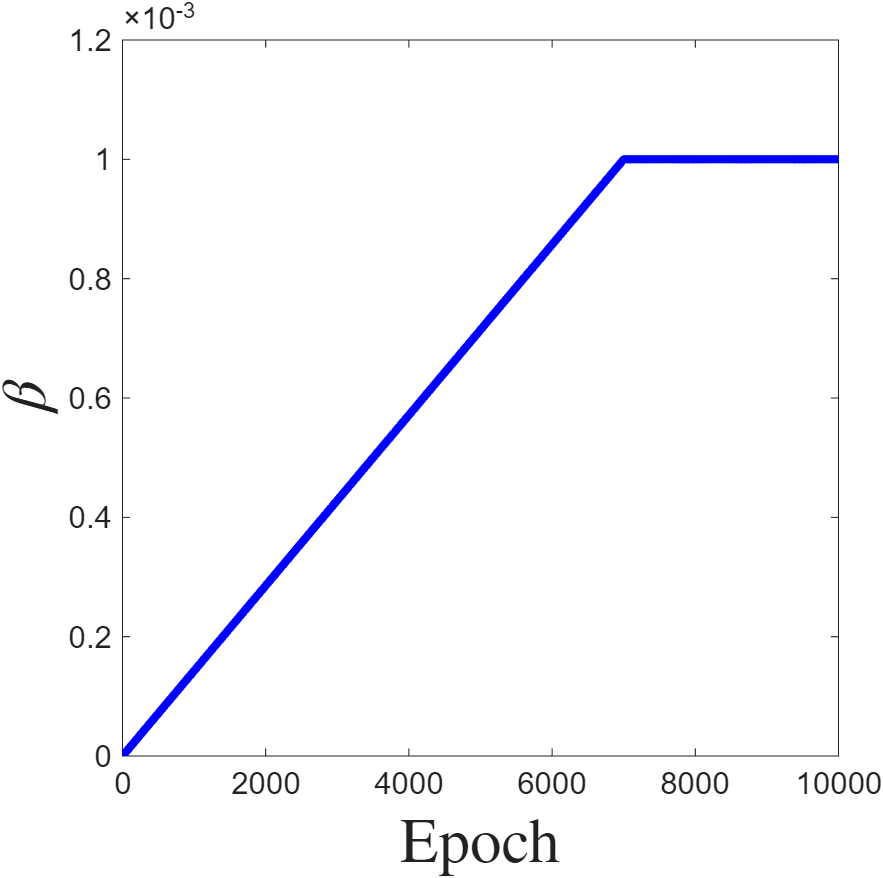}\hfill
    \includegraphics[width=0.32\textwidth]{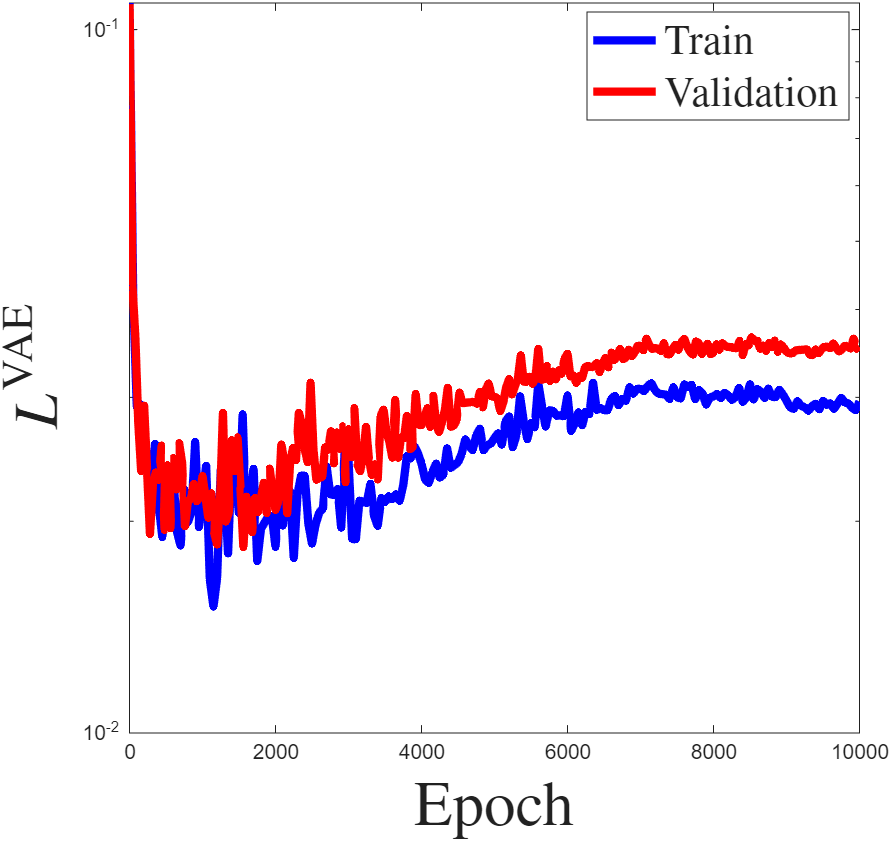}\hfill
    \includegraphics[width=0.32\textwidth]{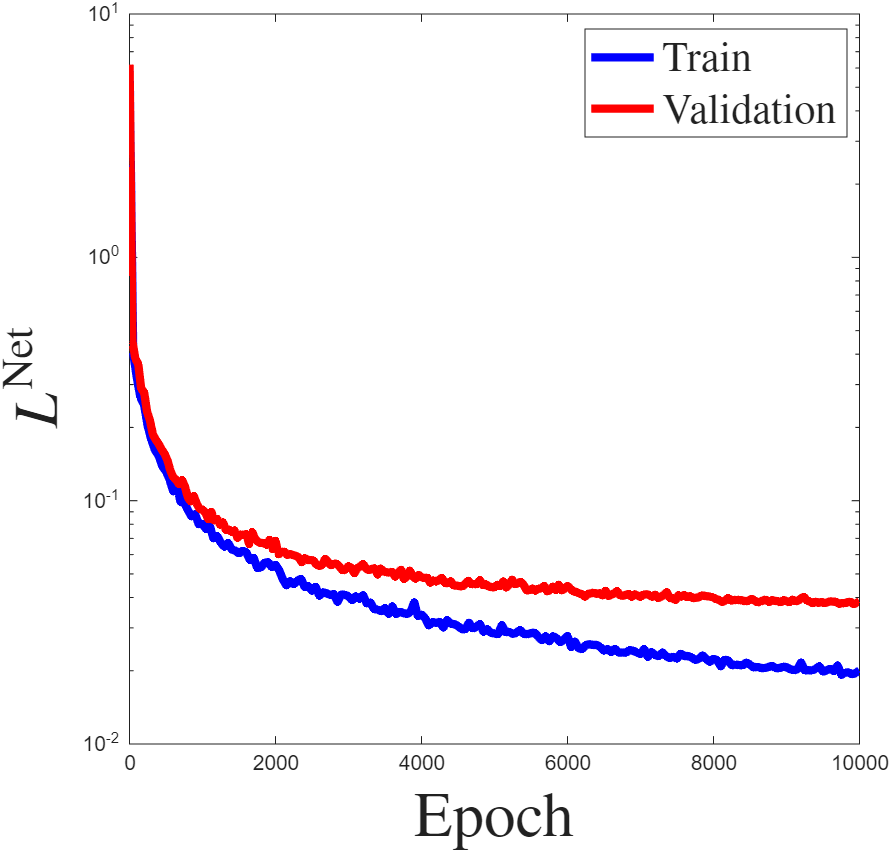}
    \caption{Training of the ellipse representation and surrogate.
    \emph{Left:} KL annealing schedule for the weight \(\beta\).
    \emph{Middle:} training and validation losses of the VAE.
    \emph{Right:} training and validation losses of the surrogate network.}
    \label{fig:VAE_ellipse}
    \label{fig:ellipse-VAE-network-loss}
\end{figure}

Since the VAE learns a probabilistic latent representation of the coefficient family, new coefficient fields can be generated by sampling latent vectors from the standard Gaussian prior and passing them through the decoder. Figure~\ref{fig:numerics-ellipse-samples} shows a collection of fields generated in this manner. The displayed samples exhibit geometric characteristics similar to those of the training data.

We then train a surrogate network to map latent vectors to the boundary responses of the decoded coefficient fields.

\begin{figure}
    \centering
    \includegraphics[width=0.75\linewidth]{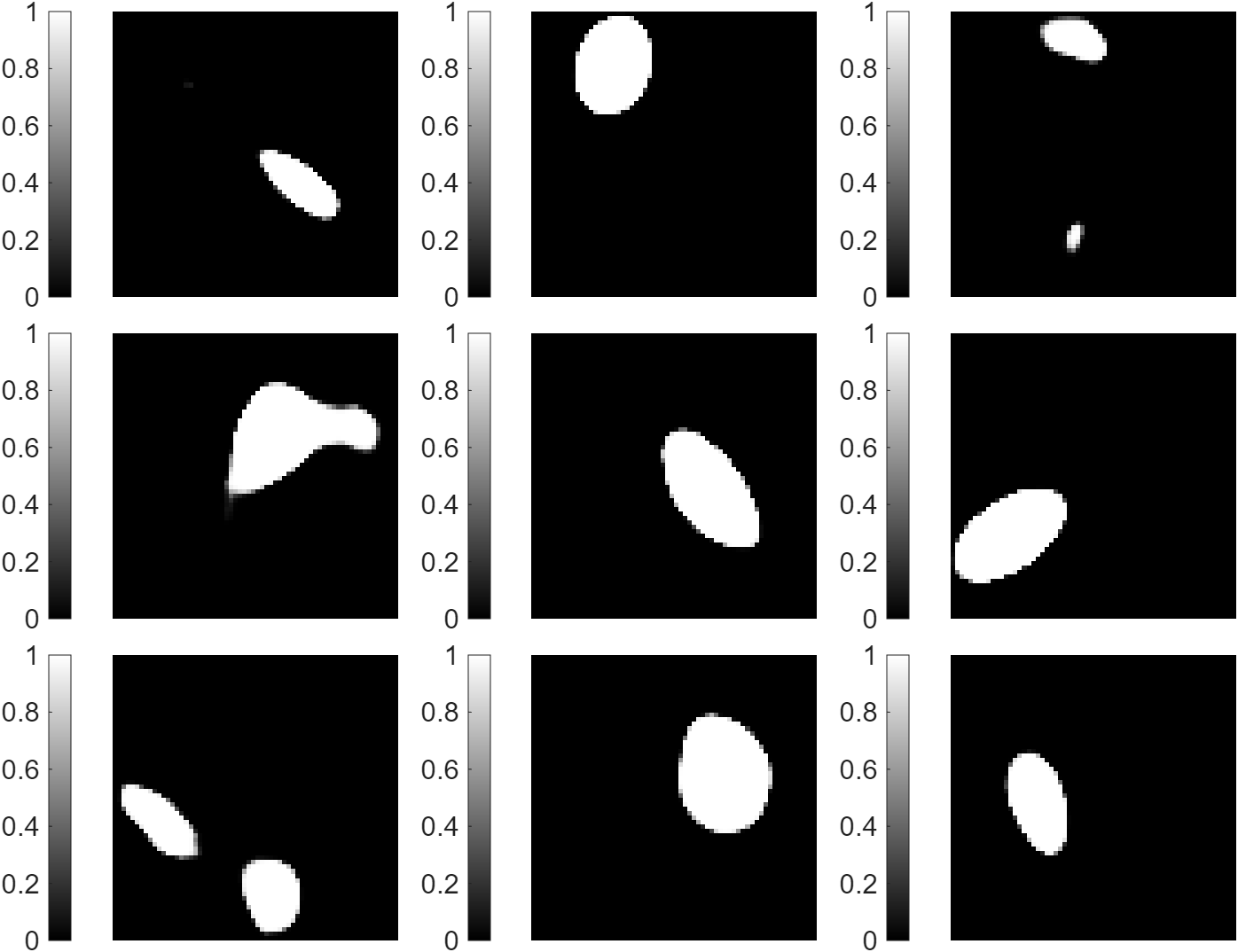}
    \caption{Examples of normalized coefficient contrasts \(\widehat a=a-a_{\min}\) generated by sampling latent vectors from the standard Gaussian prior and decoding them with the trained ellipse VAE.}
    \label{fig:numerics-ellipse-samples}
\end{figure}

\paragraph{Training Details (Surrogate Network).}
We trained the surrogate network using a dataset of \(5\,000\) latent-vector and boundary-response pairs constructed as described in Section~\ref{subsec:numerics-surrogate}, holding out \(500\) pairs for validation. We used the ELU activation function, and \eqref{eq:numerics-overparametrization} gives \(M=18\,000\) hidden units. We trained the network with batches of size \(128\) for \(10\,000\) epochs using the Adam optimizer with an initial learning rate of \(10^{-3}\), reduced by a factor of \(0.75\) every \(2\,000\) epochs. The training and validation loss curves are shown in Figure~\ref{fig:ellipse-VAE-network-loss}.

\paragraph{Reconstruction.}
To illustrate reconstruction in a learned latent space, we reused the two target coefficients from the analytical parametrization experiment and reconstructed them using the trained surrogate network and the VAE. We applied gradient descent with Armijo line search to minimize \(J^{\mathrm{Net}}(z)\), defined in \eqref{eq:numerics-boundary-objective-surrogate}, with \(\Phi\) given by \eqref{eq:numerics-representation-map}. The initial latent vectors were sampled from the standard Gaussian distribution.

Figures~\ref{fig:ellipse-sequence-1} and~\ref{fig:ellipse-sequence-2} show the reconstruction process for two targets, including selected iterates and the final reconstructions after \(100\) gradient descent iterations. For these displayed examples, the reconstructed inclusions closely match the targets, and the visible errors are concentrated near the ellipse boundaries. Unlike the binary analytical parametrization, the learned decoder is not constrained to produce binary pixel values. Its final sigmoid layer and the approximate nature of the learned representation therefore allow intermediate values between \(0\) and \(1\). This spatial transition should be distinguished from the \(C^1\)-dependence of the decoder on the latent variable.

Along the displayed optimization trajectories, the surrogate objective follows the overall trend of the FEM objective. The figures also include the \(L^2(\Omega)\)-error between the target and reconstructed coefficient fields.

\begin{figure}
    \centering
    \includegraphics[width=\linewidth]{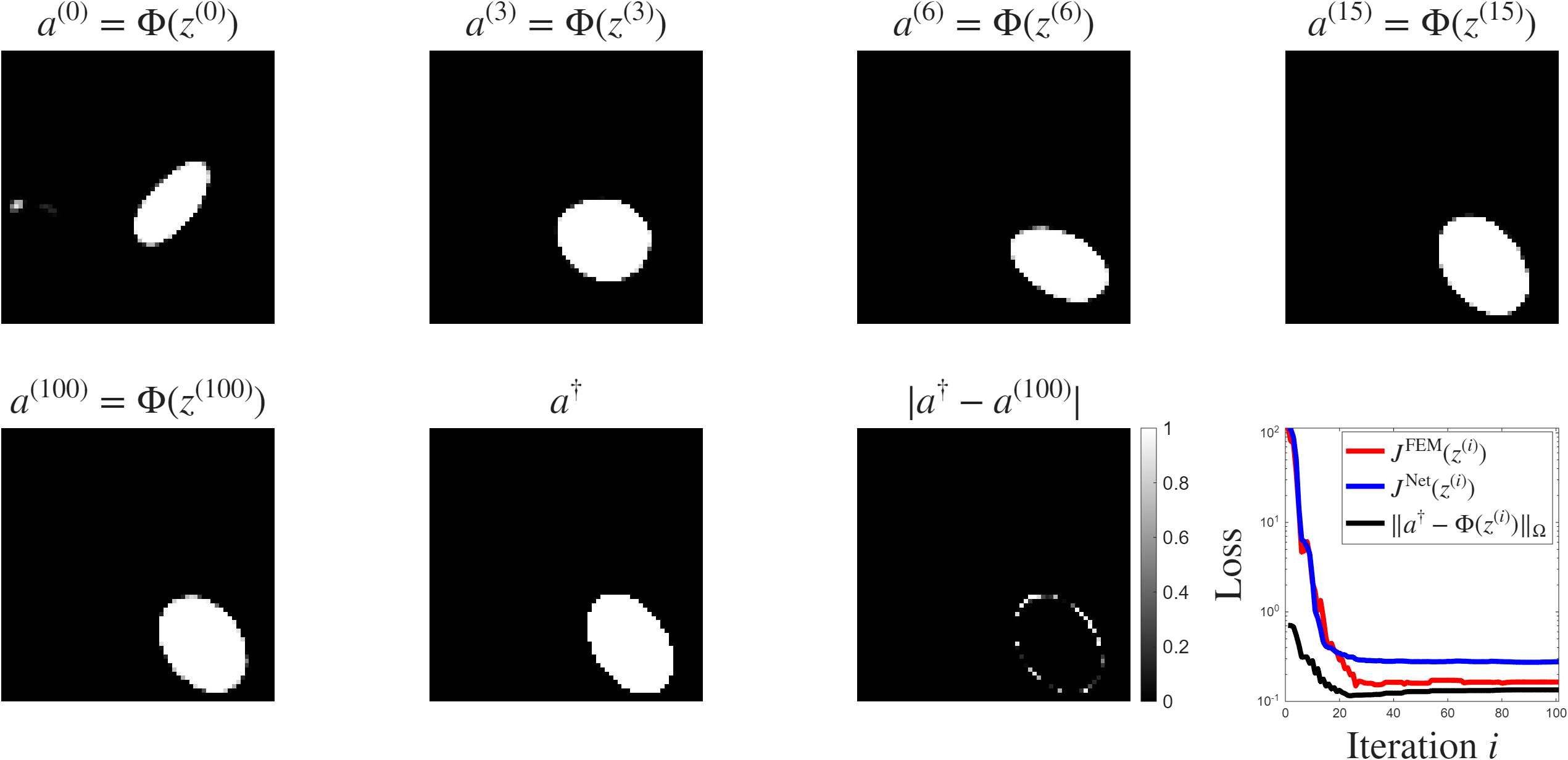}
    \caption{Surrogate-based reconstruction of a target coefficient \(a^\dagger\) using the learned ellipse representation. \emph{Top row:} selected iterates \(a^{(i)}=\Phi(z^{(i)})\). \emph{Bottom row:} final reconstruction \(a^{(100)}\), target \(a^\dagger\), pointwise absolute error \(\lvert a^\dagger-a^{(100)}\rvert\), and histories of \(J^{\mathrm{Net}}(z^{(i)})\), \(J^{\mathrm{FEM}}(z^{(i)})\), and \(\lVert a^\dagger-a^{(i)}\rVert_{L^2(\Omega)}\).}
    \label{fig:ellipse-sequence-1}
\end{figure}

\begin{figure}
    \centering
    \includegraphics[width=\linewidth]{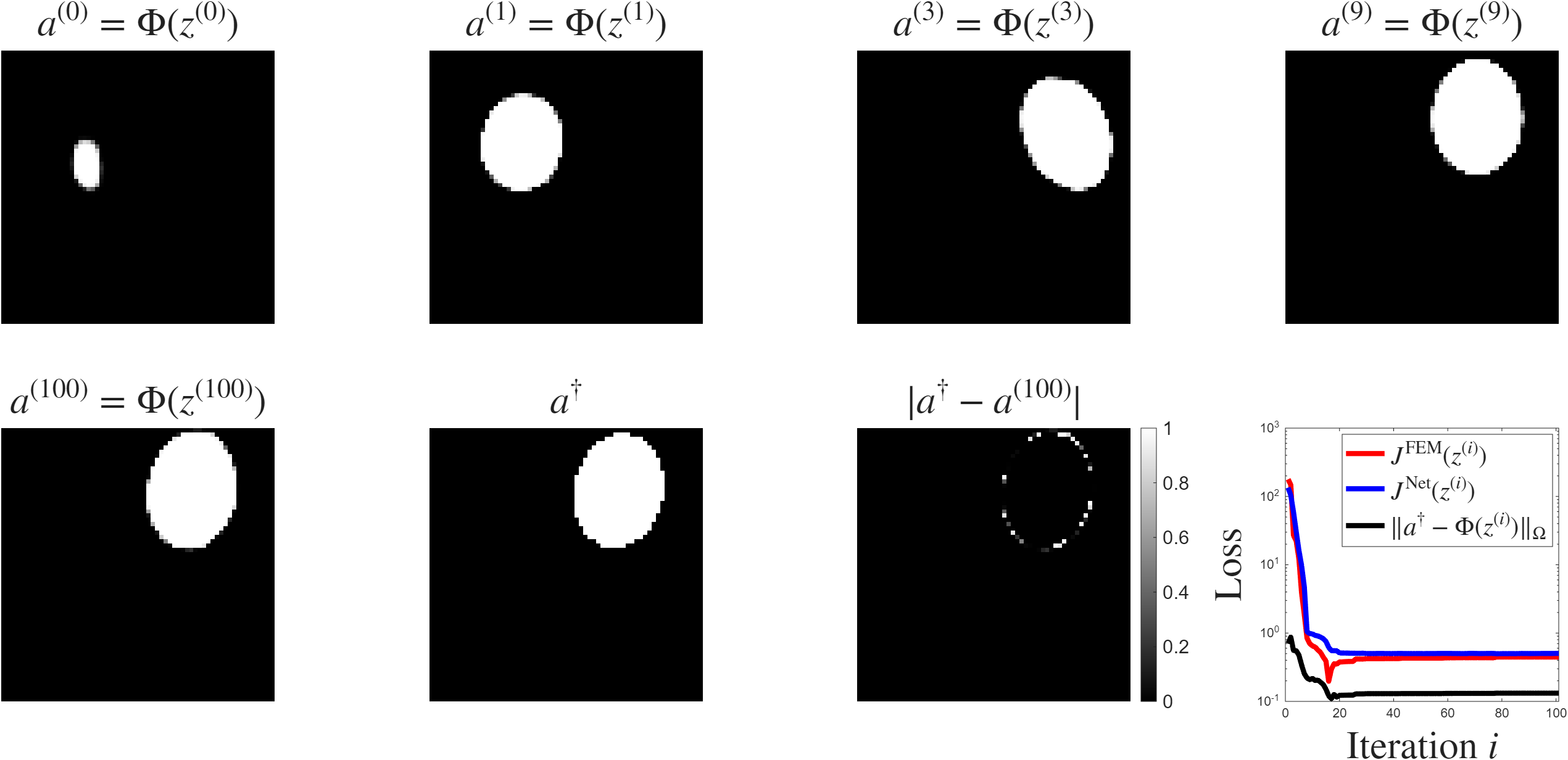}
    \caption{Surrogate-based reconstruction of another target coefficient \(a^\dagger\) using the learned ellipse representation. \emph{Top row:} selected iterates \(a^{(i)}=\Phi(z^{(i)})\). \emph{Bottom row:} final reconstruction \(a^{(100)}\), target \(a^\dagger\), pointwise absolute error \(\lvert a^\dagger-a^{(100)}\rvert\), and histories of \(J^{\mathrm{Net}}(z^{(i)})\), \(J^{\mathrm{FEM}}(z^{(i)})\), and \(\lVert a^\dagger-a^{(i)}\rVert_{L^2(\Omega)}\).}
    \label{fig:ellipse-sequence-2}
\end{figure}

\paragraph{Refined Reconstructions.}
We refine the surrogate-based reconstruction by minimizing \(J^{\mathrm{FEM}}(z)\) in \eqref{eq:numerics-boundary-objective-fem} with a Newton-CG method and Armijo line search, initialized with the recovered latent vector.

Specifically, we first perform \(150\) gradient descent iterations with Armijo line search to minimize \(J^{\mathrm{Net}}(z)\) for a given target coefficient \(a^\dagger\). This yields a latent vector \(z^{(150)}\) and coefficient field \(a^{(150)}=\Phi(z^{(150)})\). Starting from \(z^{(150)}\), we then perform \(50\) Newton-CG iterations with Armijo line search. These iterations remain in the latent space but evaluate the boundary responses with the FEM rather than the surrogate. JAX is used to differentiate through the FEM solver.

Figure~\ref{fig:ellipse-refined-reconstruction} shows the results for four targets. The first column contains the target coefficients, the second contains the final refined reconstructions, and the third contains the loss curves. Note that the targets are generated from the underlying coefficient model, rather than being constructed by sampling latent vectors and decoding them with the trained ellipse VAE.

\begin{figure}
    \centering
    \setlength{\tabcolsep}{2pt} 
    \begin{tabular}{ccc}
    \textbf{Target \(a^\dagger\)} &
    \textbf{\(a^{(200)}=\Phi(z^{(200)})\)} &
    \textbf{Losses}
    \\[0.4em]
    \includegraphics[height=4.5cm]{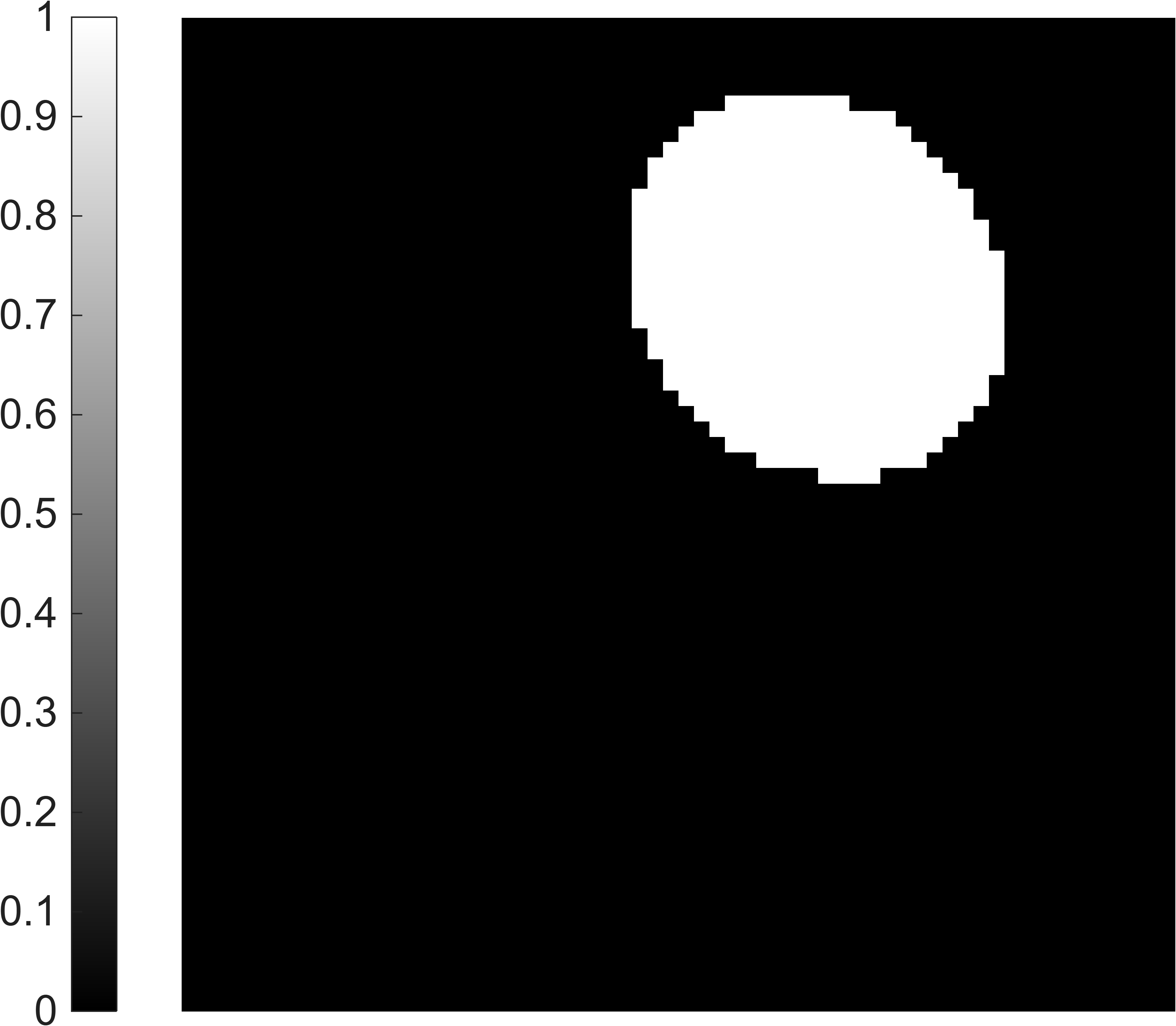} &
    \raisebox{0.07cm}{\includegraphics[height=4.35cm]{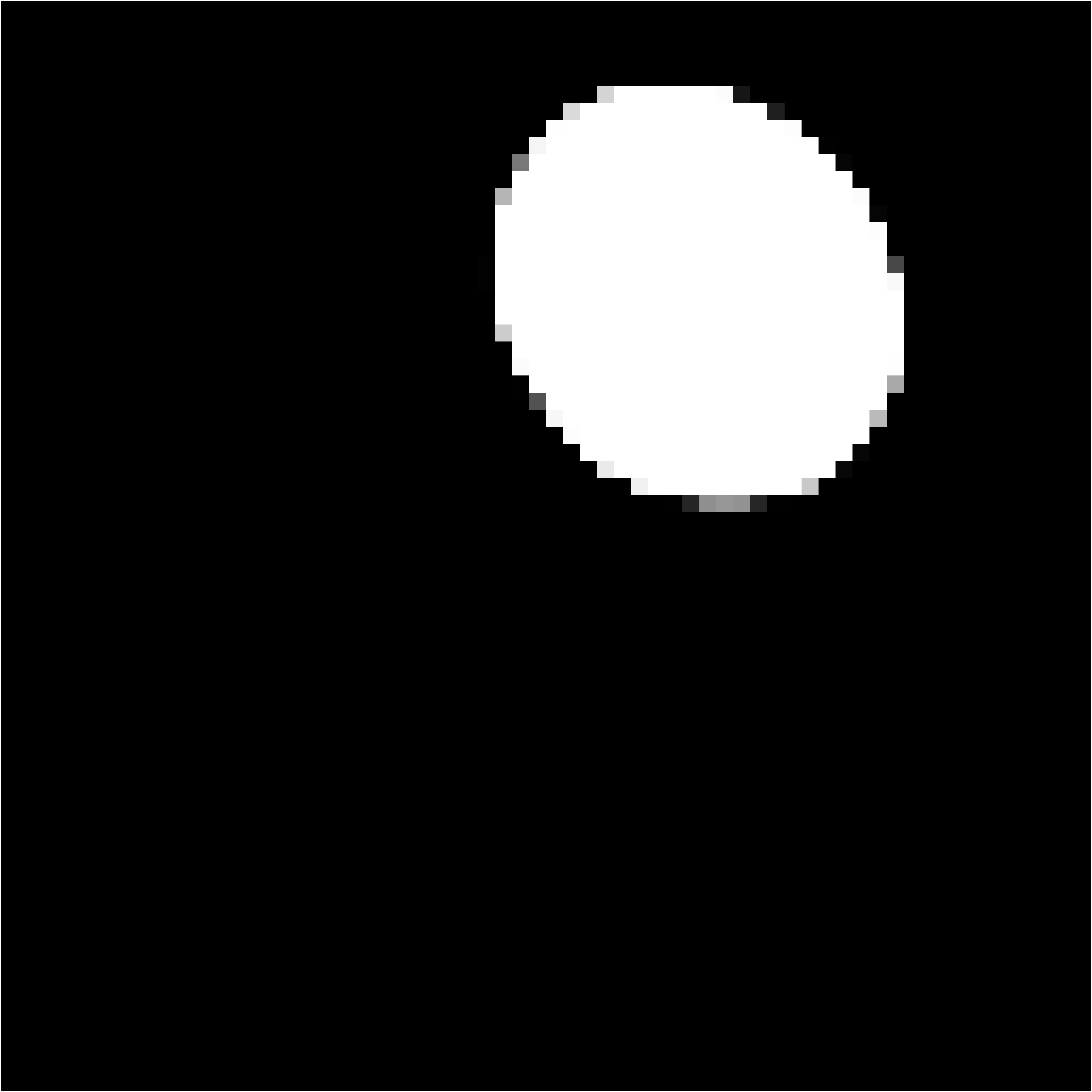}} &
    \includegraphics[height=4.6cm]{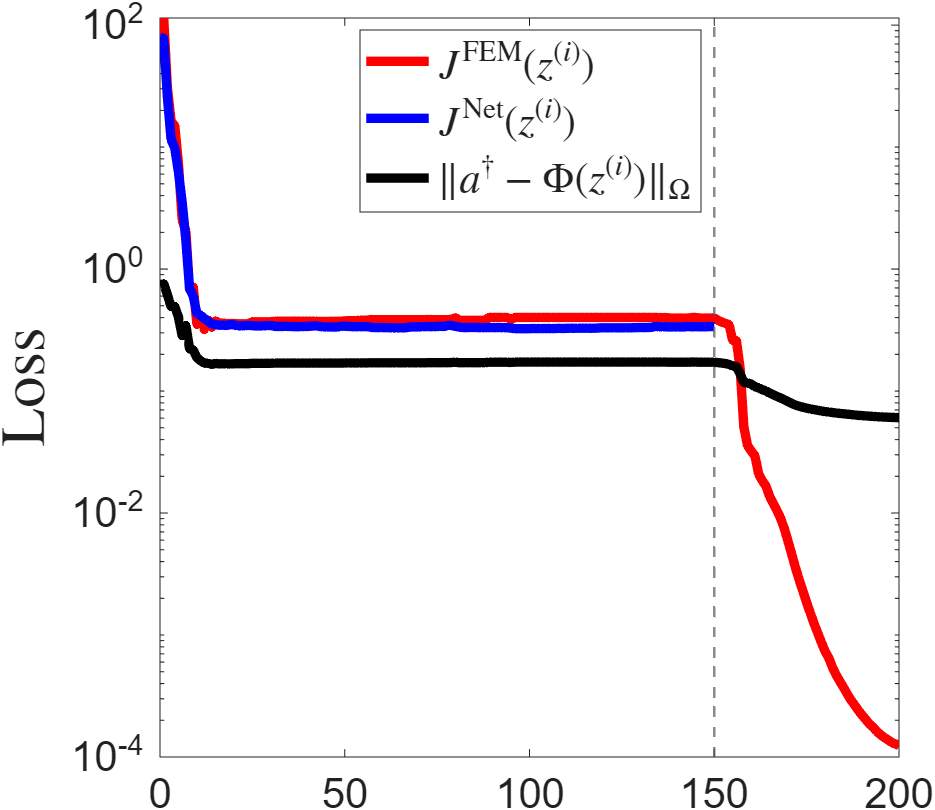}
    \\[0.5em]
    \includegraphics[height=4.5cm]{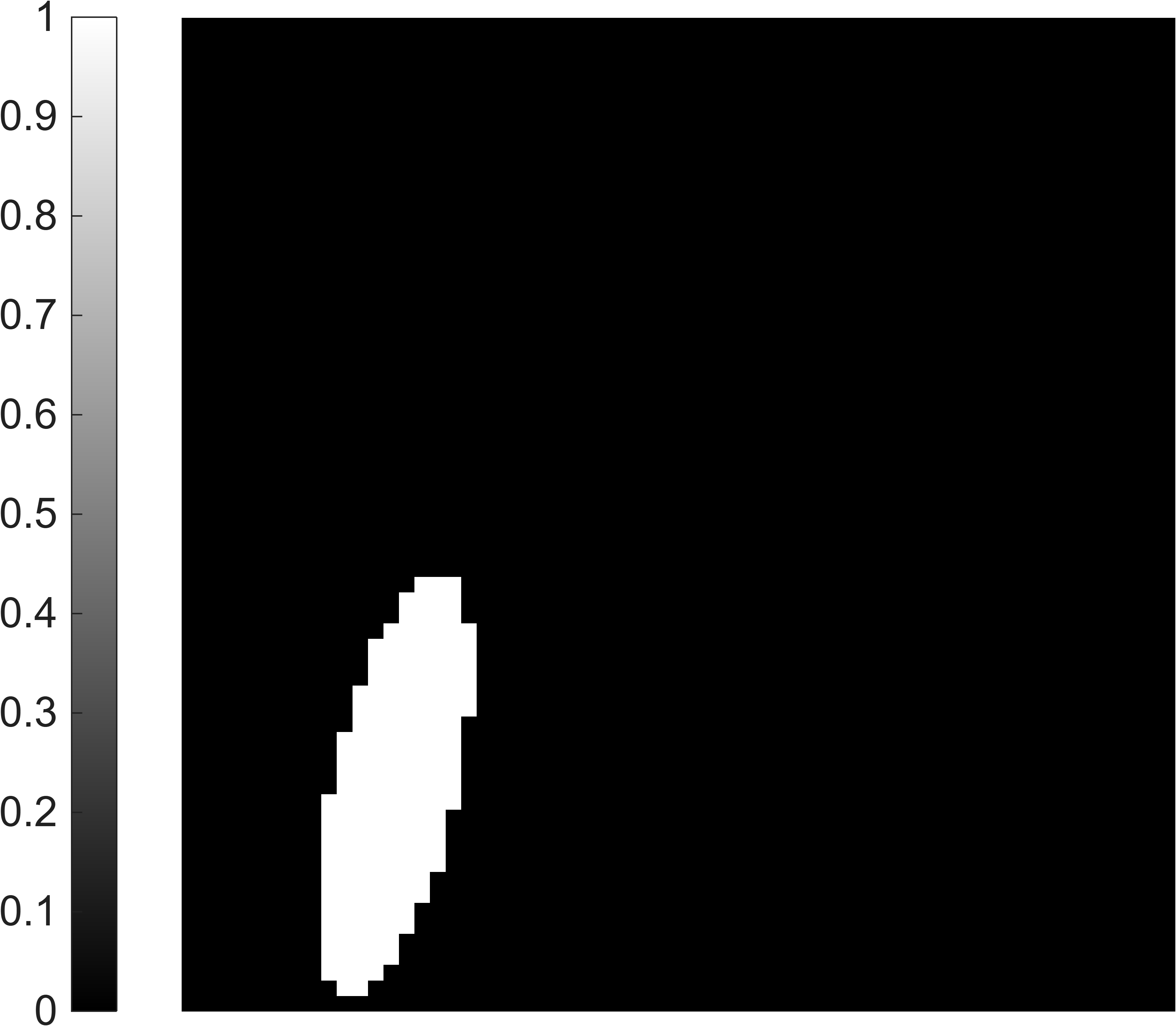} &
    \raisebox{0.07cm}{\includegraphics[height=4.35cm]{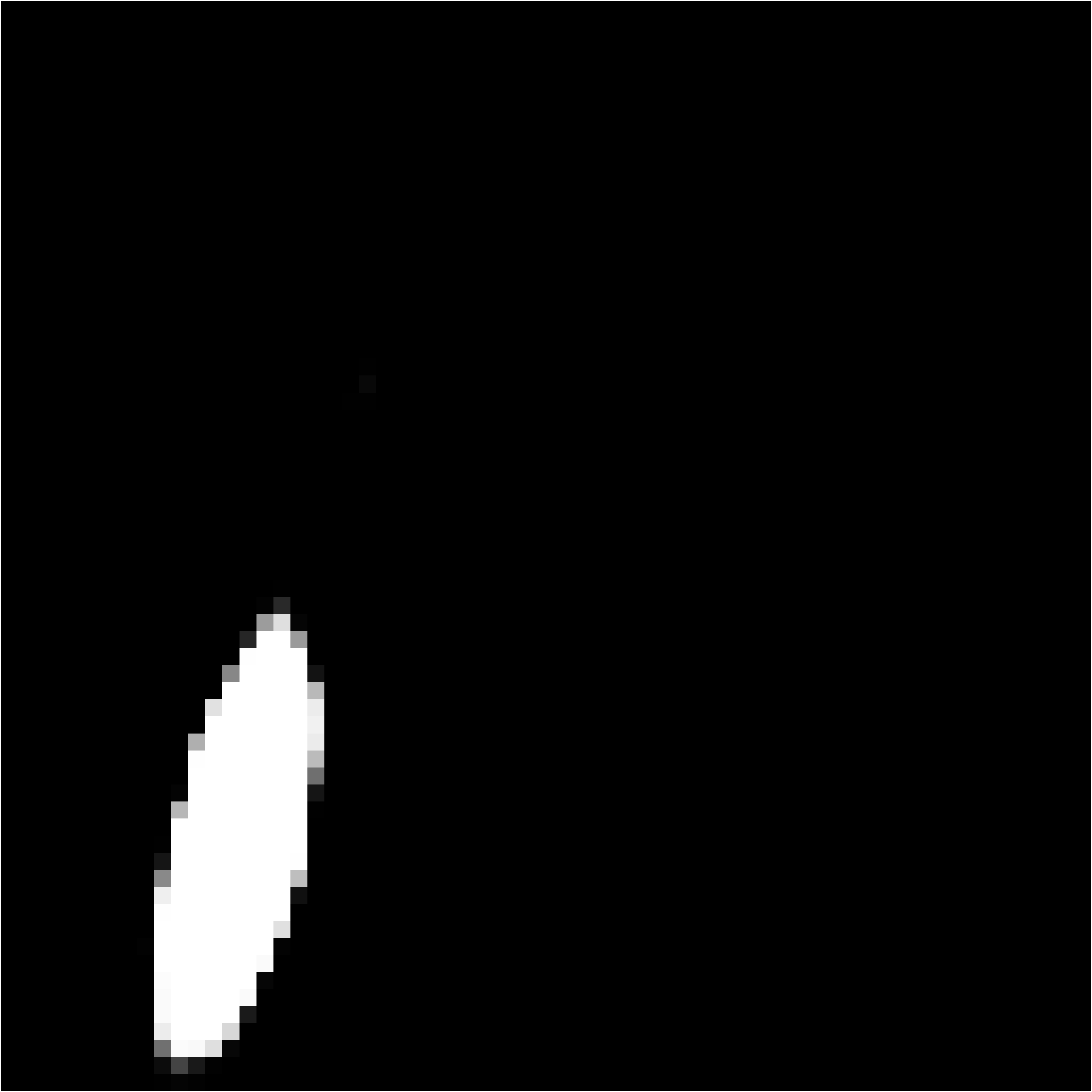}} &
    \includegraphics[height=4.6cm]{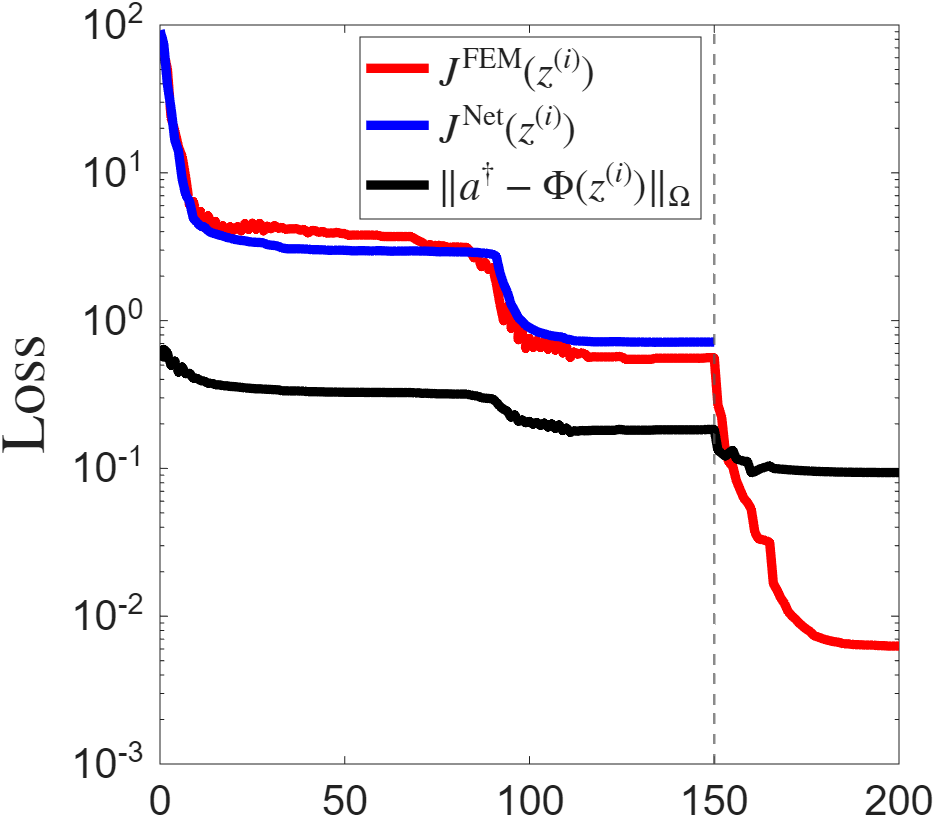}
    \\[0.5em]
    \includegraphics[height=4.5cm]{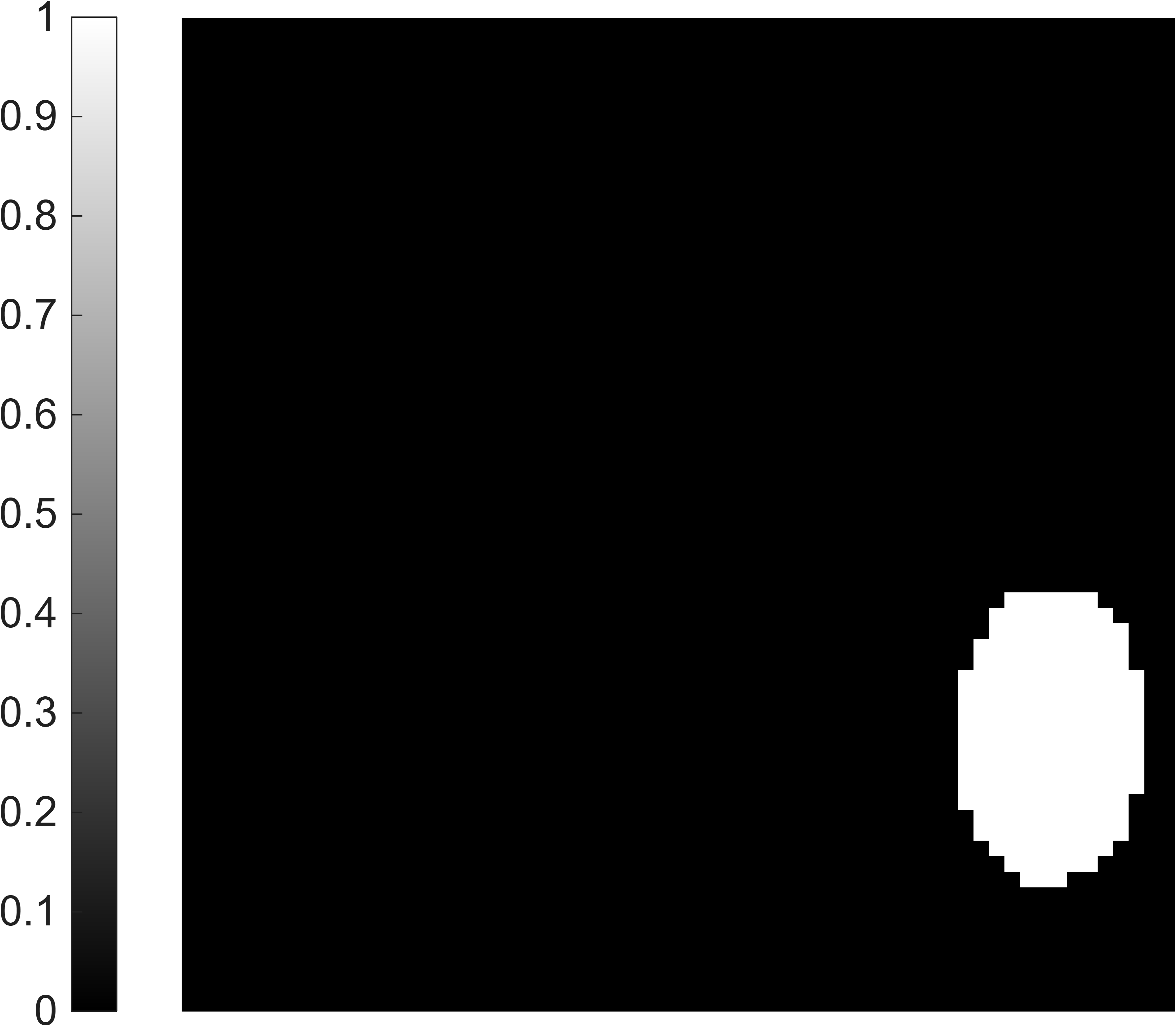} &
    \raisebox{0.07cm}{\includegraphics[height=4.35cm]{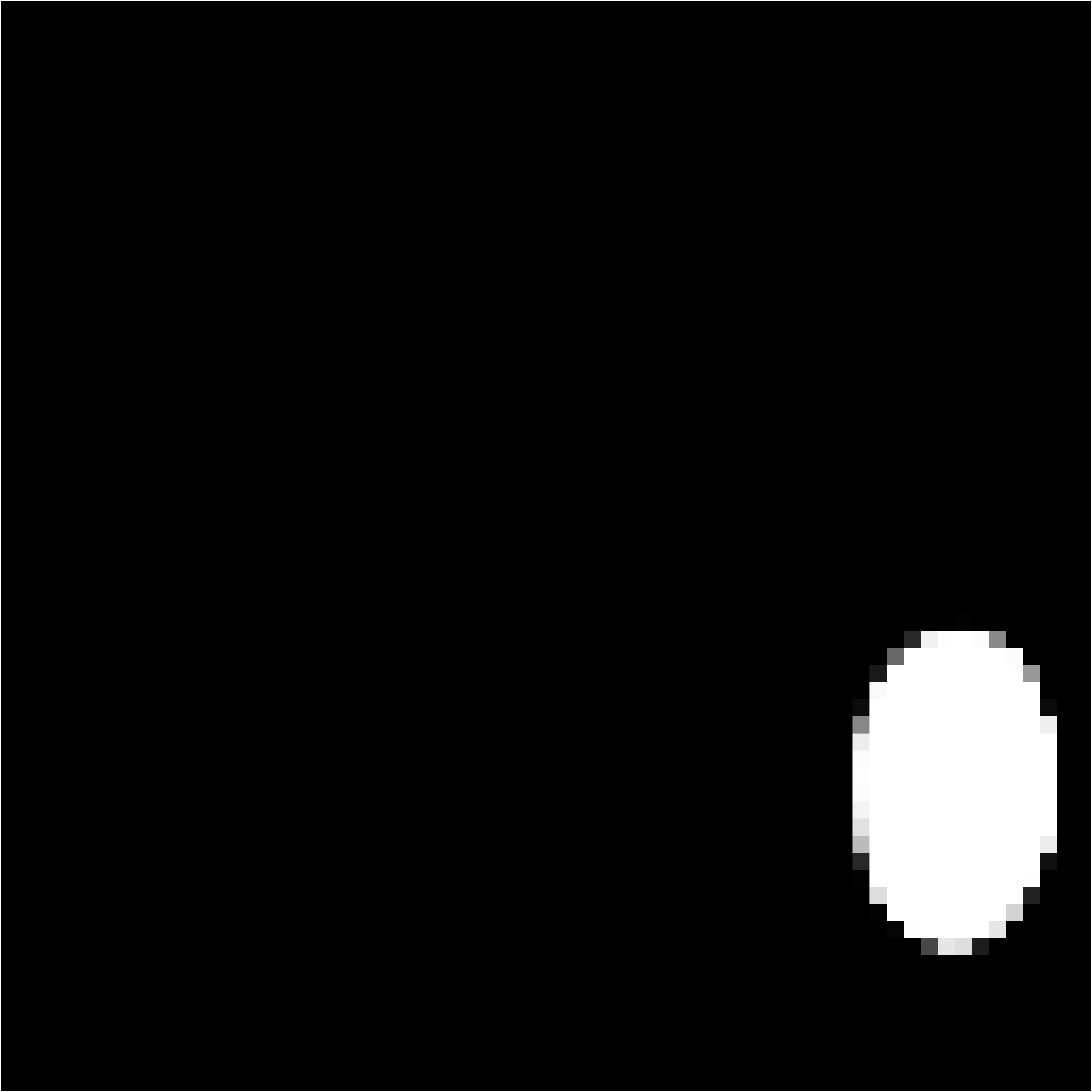}} &
    \includegraphics[height=4.6cm]{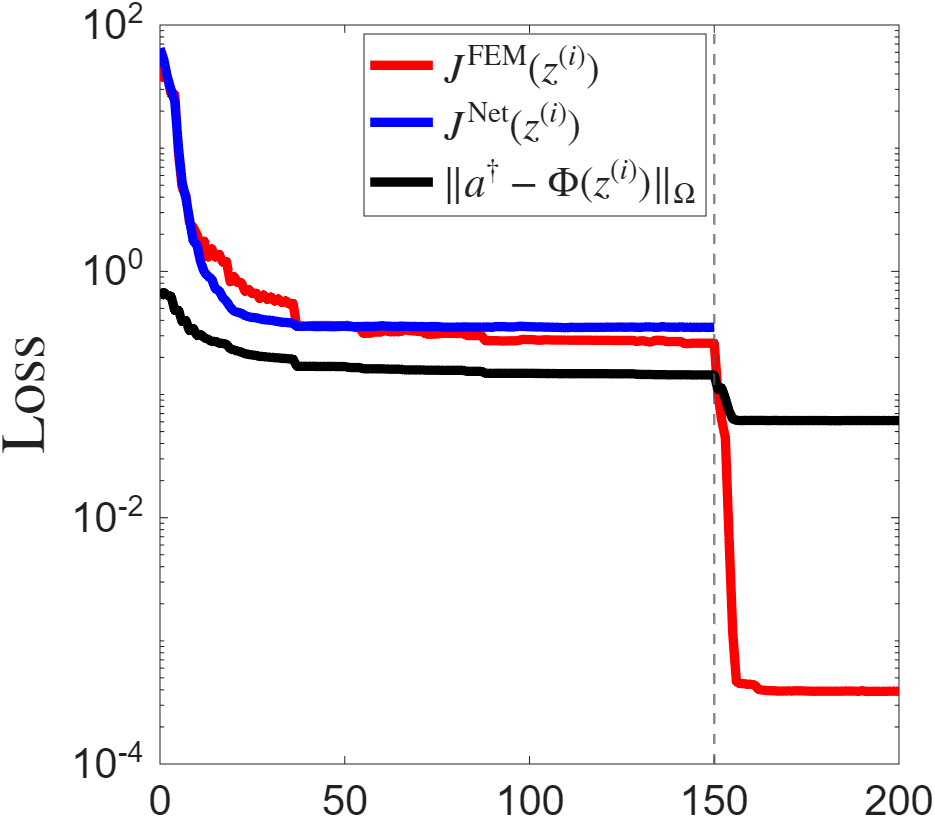}
    \\[0.5em]
    \includegraphics[height=4.5cm]{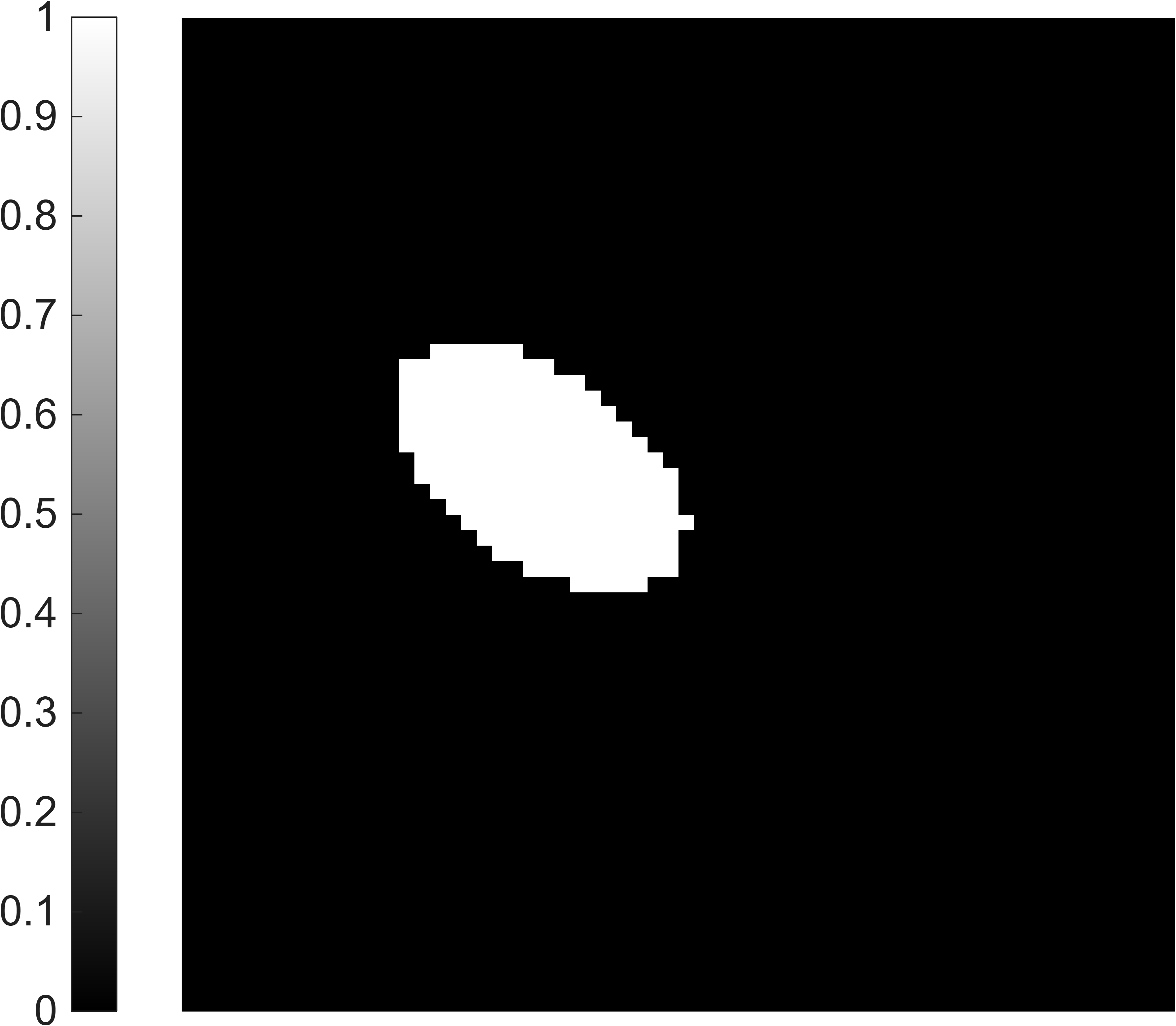} &
    \raisebox{0.07cm}{\includegraphics[height=4.35cm]{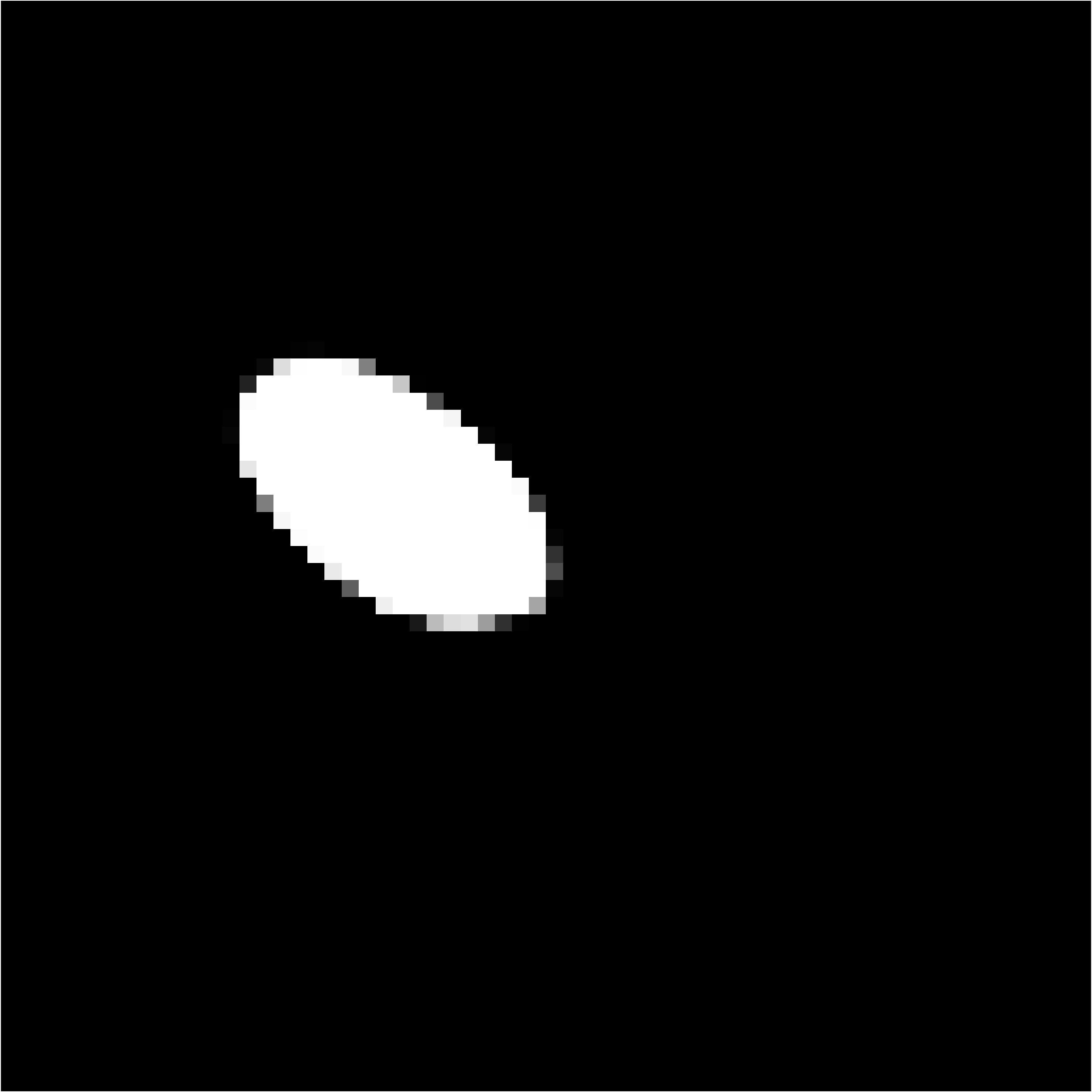}} &
    \includegraphics[height=4.6cm]{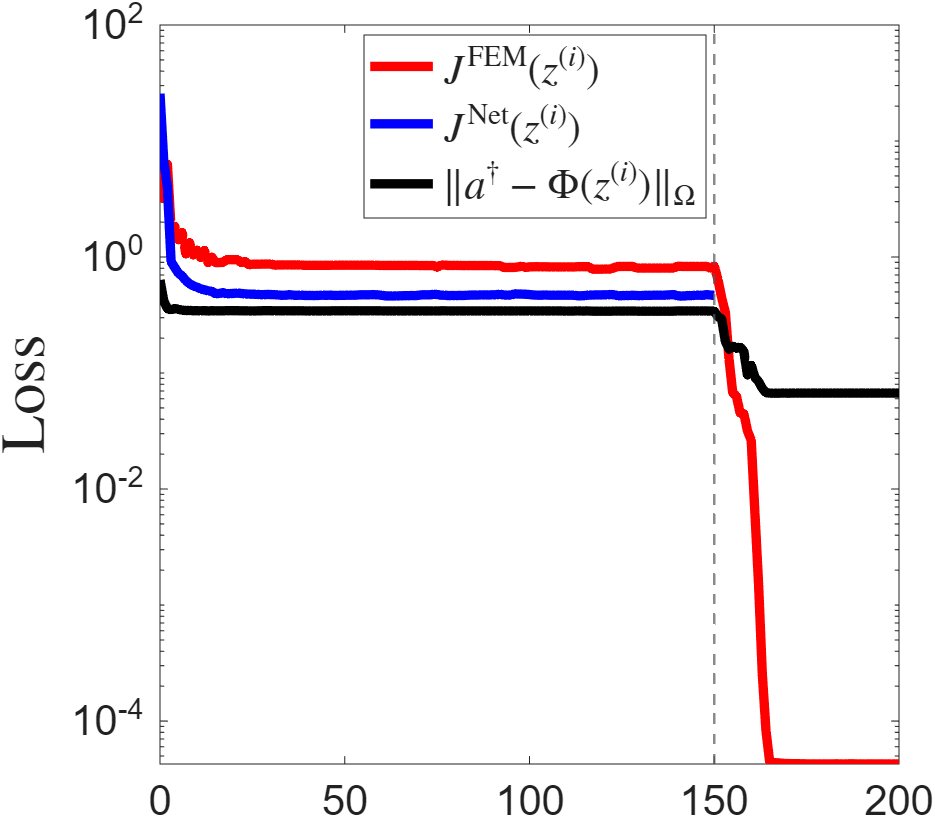}
    \\
    & & \small Iteration \(i\)
    \end{tabular}
    \caption{Refined reconstructions for four target coefficients from the ellipse model. \emph{Columns:} target coefficient \(a^\dagger\), final reconstruction \(a^{(200)}=\Phi(z^{(200)})\), and histories of \(J^{\mathrm{Net}}(z^{(i)})\), \(J^{\mathrm{FEM}}(z^{(i)})\), and \(\lVert a^\dagger-a^{(i)}\rVert_{L^2(\Omega)}\). The vertical dashed line at \(i=150\) marks the transition from surrogate-based gradient descent to FEM-based Newton-CG iterations.}
    \label{fig:ellipse-refined-reconstruction}
\end{figure}

For the displayed targets, the third column of Figure~\ref{fig:ellipse-refined-reconstruction} shows a rapid decrease in \(J^{\mathrm{FEM}}(z)\) after the Newton iterations begin. The corresponding decrease in the \(L^2(\Omega)\)-error shows that the refinement improves both the data fit and the coefficient reconstruction in these examples.

\subsection{Crack-Like Defects in Heterogeneous Backgrounds}
\label{subsec:cracks-with-background-field}
The second coefficient family models localized high-contrast cracks in spatially correlated heterogeneous backgrounds. The background is generated from a stationary
Gaussian random field using a truncated Karhunen--Loève expansion of a
squared-exponential covariance operator. 

\paragraph{Background Field.}
We first define a squared-exponential covariance matrix \(C\in\mathbb R^{p^2\times p^2}\) with entries \(C_{ij,kl}=\exp(-|c_{ij}-c_{kl}|^2/(2\ell^2))\), where the correlation length is \(\ell=1/16\) and \(c_{ij}\) and \(c_{kl}\) are the centers of \(Q_{ij}\) and \(Q_{kl}\), respectively. Let \((\lambda_k,v_k)\) denote the eigenpairs of \(C\), ordered so that \(\lambda_1\geq\lambda_2\geq\cdots\geq0\). We generate a background sample by truncating the Karhunen--Loève expansion after \(100\) terms, followed by rescaling and shifting,
\begin{align}
W
&=0.1+0.1\sum_{k=1}^{100}\sqrt{\lambda_k}\,\xi_k v_k,
\qquad \xi_k \stackrel{\mathrm{i.i.d.}}{\sim}\mathcal{N}(0,1)
\label{eq:numerics-kl-background}
\end{align}
We set \(W_c=\operatorname{clip}(W,0,0.2)\), with clipping applied pixelwise.

\paragraph{Crack Geometry.}
Each crack is represented by a binary mask on the pixel grid. Every mask contains a single crack generated as a random polyline. Starting from a uniformly sampled random point and direction, the crack is grown by successively adding line segments of length \(30\) pixels. At each step, the direction is perturbed by a random angle drawn from \(\mathcal{N}(0,0.25)\).

The segment is stamped onto the mask by setting all pixels within a distance of
\(1.5\) pixels from the segment to \(1\). When the growth process reaches the boundary, it is continued in the opposite direction from the initial point. The process stops when both ends of the polyline have reached the boundary.

Writing \(\chi\) for the resulting binary crack mask, the physical coefficient is
\begin{equation}
a=a_{\min}+\chi+(1-\chi)W_c
\qquad a_{\min}=0.05
\label{eq:numerics-crack-coefficient-field}
\end{equation}
where the operations are pixelwise. Crack pixels replace the background by the
value \(1\) before the positive offset is added; the remaining pixels retain
their background values.

Examples of the resulting coefficient fields are shown in Figure~\ref{fig:numerics-cracks-examples}. Recall that the constant offset \(a_{\min}\) is omitted from the coefficient plots.

\begin{figure}
\centering
\includegraphics[width=0.75\linewidth]{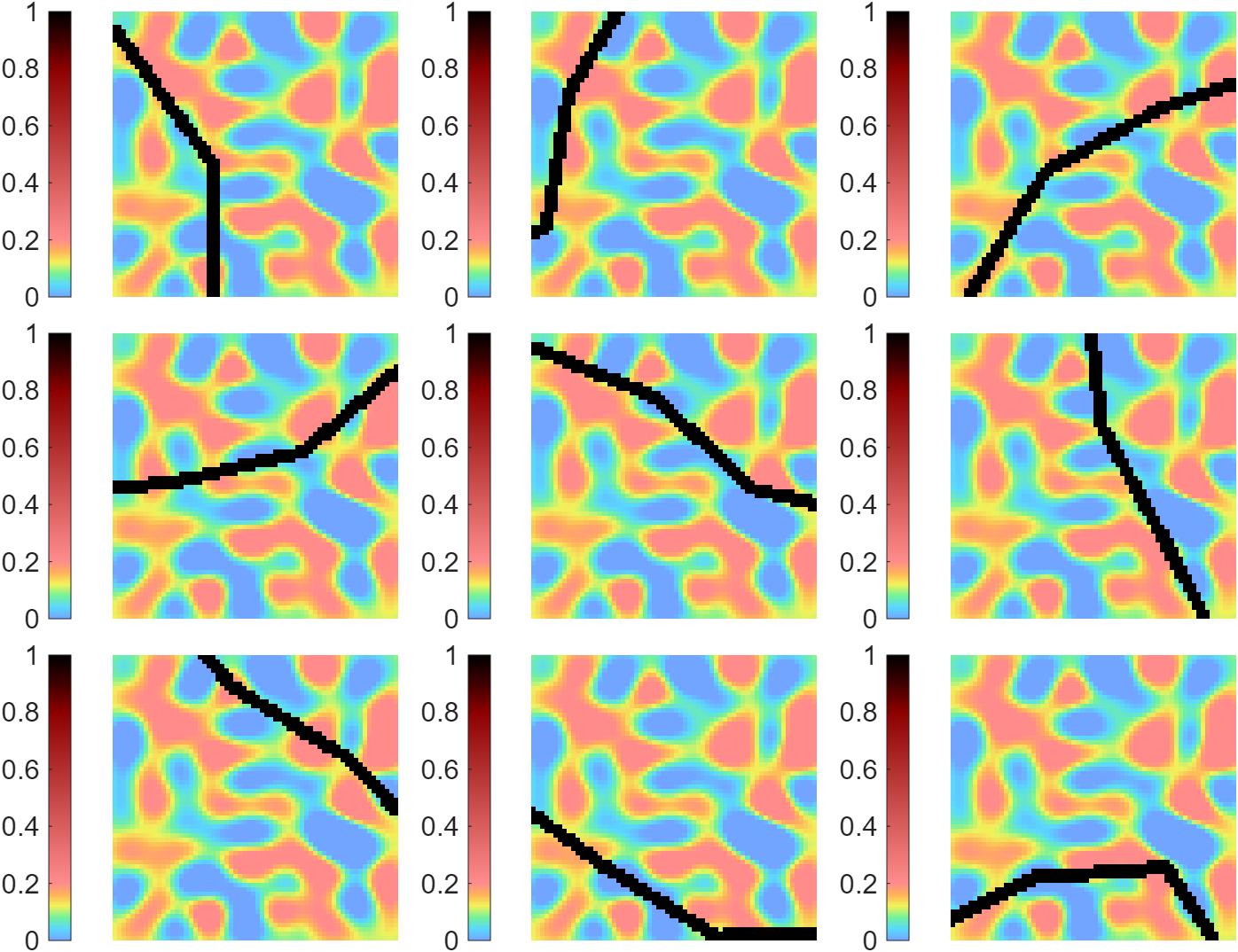}
\caption{Examples of normalized coefficient contrasts \(\widehat a=a-a_{\min}\) from the crack family. Cracks are shown in black, while the colors encode variation in the heterogeneous background.}
\label{fig:numerics-cracks-examples}
\end{figure}

\subsubsection{VAE Parametrization}
For this more complex coefficient family, the representation map is learned using a variational autoencoder. We use the same architecture as in the previous experiment, except that the ELU activation functions are replaced by ReLU activations, which provided better performance in this setting. The latent dimension is increased to \(m=8\). We use \(L=10\) boundary excitations with the Neumann patterns defined by \eqref{eq:numerics-neumann-modes}.

\paragraph{Training Details (VAE).}
The VAE was trained for \(10\,000\) epochs using SGD with a learning rate of \(0.5\) and a batch size of \(128\). We trained it on \(9\,000\) samples, with an additional \(1\,000\) samples reserved for validation. To improve training stability, we used KL annealing, gradually increasing the weight \(\beta\) from \(10^{-8}\) to \(2\times10^{-3}\) in the variational objective \eqref{eq:numerics-vae-loss}. The edge-aware weight \(\lambda^{\mathrm{edge}}\) in \eqref{eq:numerics-vae-loss} was fixed at \(0.1\) throughout training. Figure~\ref{fig:VAE_cracks} shows the annealing schedule together with the training and validation loss curves.

\begin{figure}
    \centering
    \includegraphics[width=0.32\textwidth]{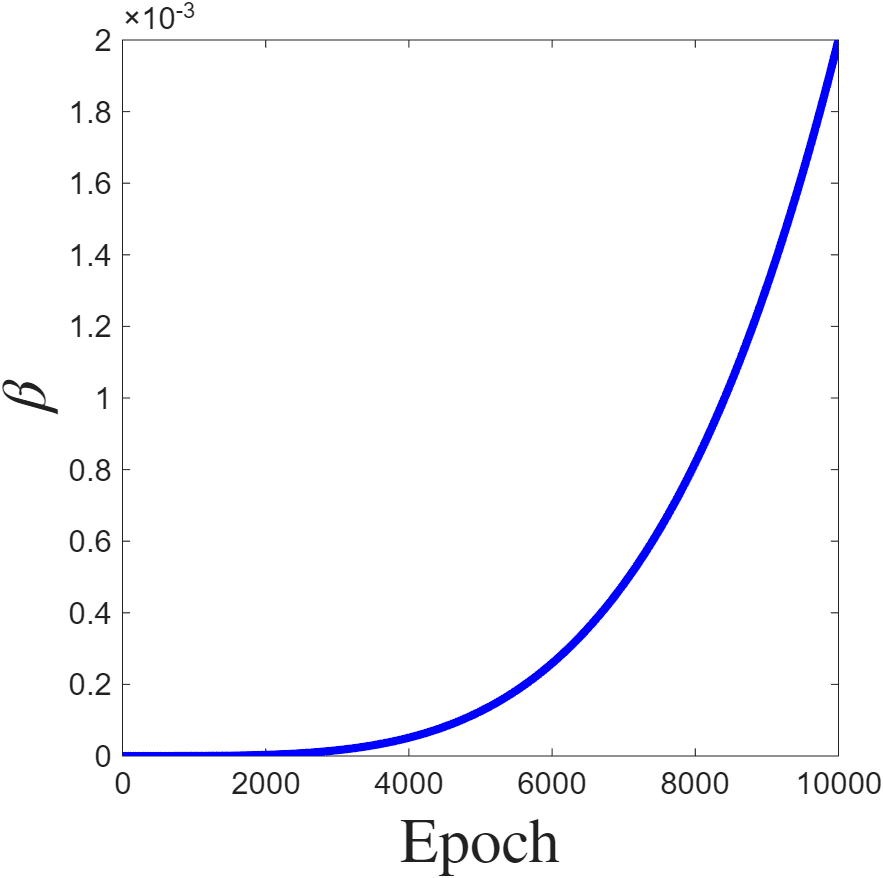}\hfill
    \includegraphics[width=0.32\textwidth]{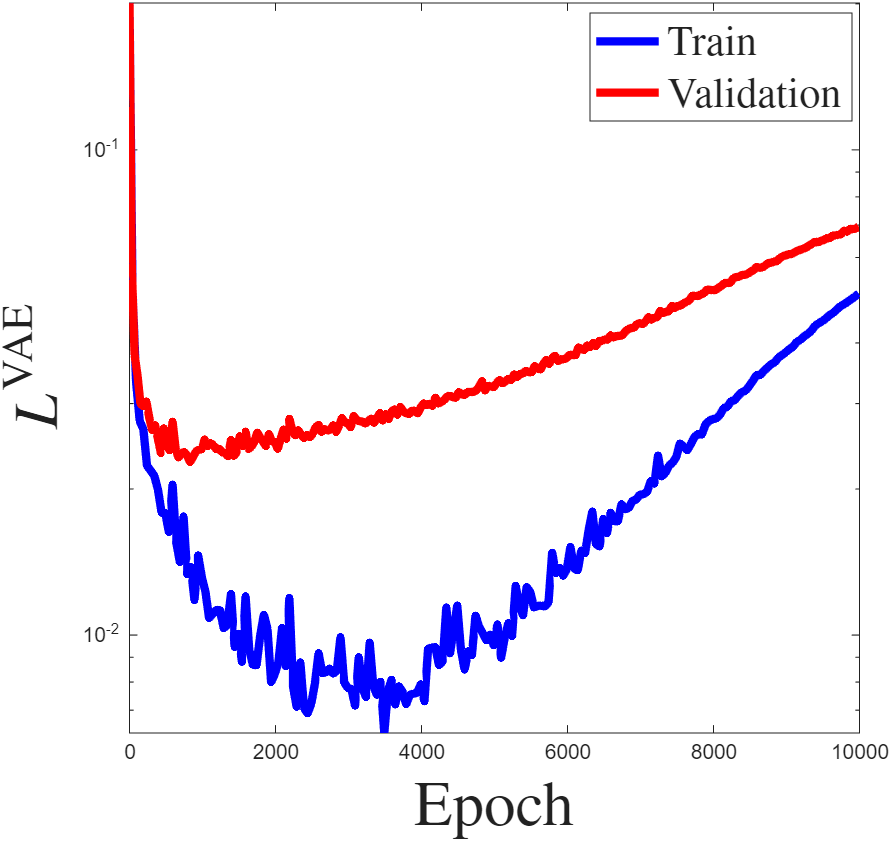}\hfill
    \includegraphics[width=0.32\textwidth]{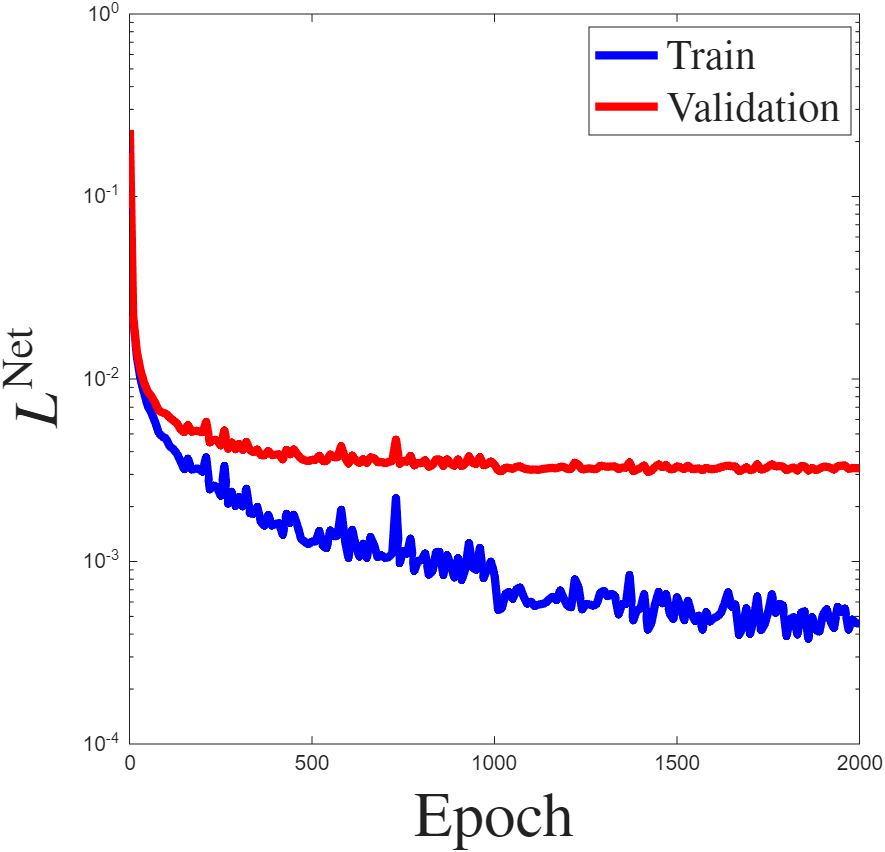}
    \caption{Training of the crack representation and surrogate.
    \emph{Left:} KL annealing schedule for the weight \(\beta\).
    \emph{Middle:} training and validation losses of the VAE.
    \emph{Right:} training and validation losses of the surrogate network.}
    \label{fig:VAE_cracks}
    \label{fig:Network-cracks-loss}
\end{figure}
Figure~\ref{fig:numerics-cracks-samples} presents coefficient fields obtained by sampling latent vectors from the standard Gaussian prior and decoding them with the VAE. The displayed samples reproduce characteristic geometric and background features of the training data.
\begin{figure}
    \centering
    \includegraphics[width=0.75\linewidth]{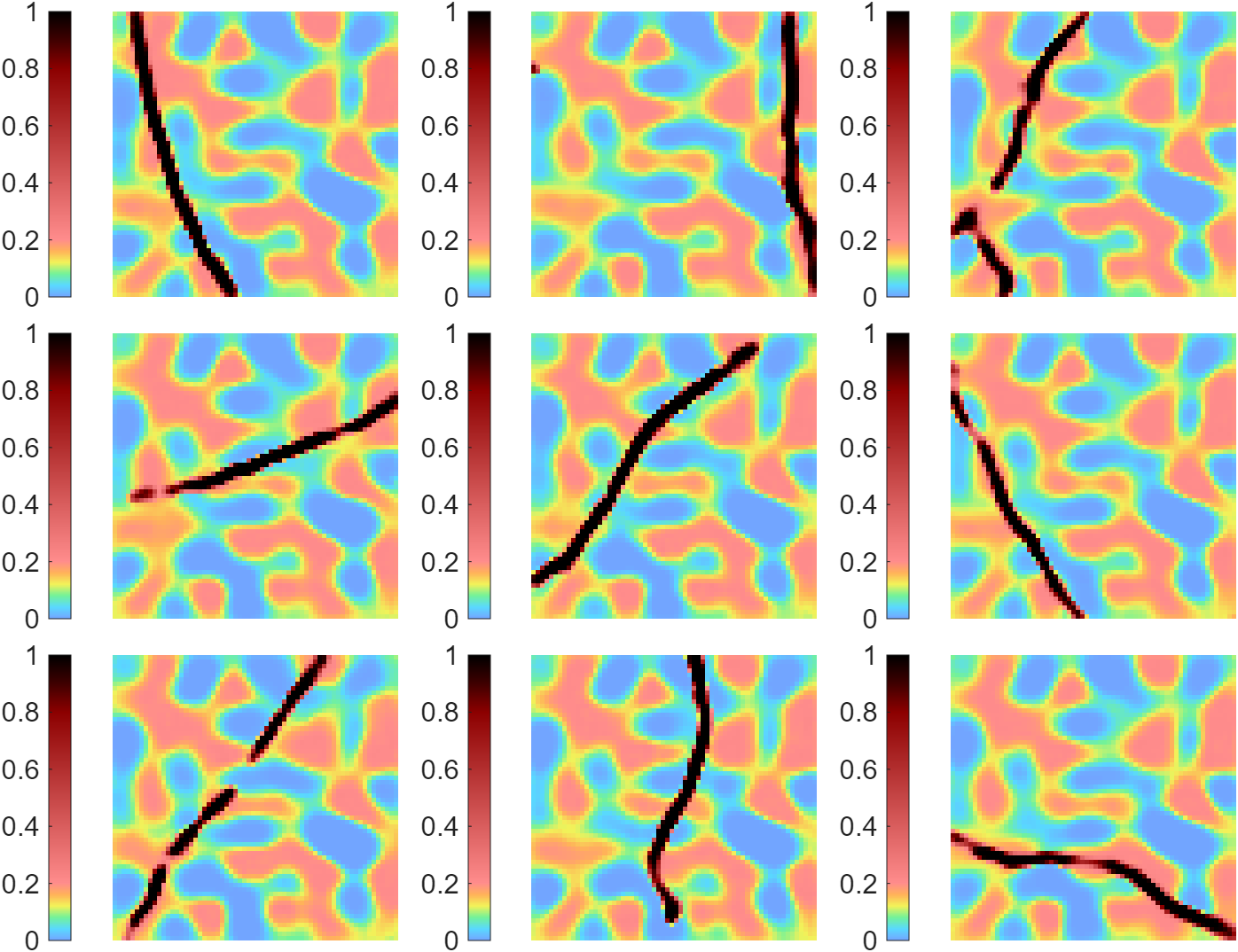}
    \caption{Examples of normalized coefficient contrasts \(\widehat a=a-a_{\min}\) generated by sampling latent vectors from the standard Gaussian prior and decoding them with the trained crack VAE.}
    \label{fig:numerics-cracks-samples}
\end{figure}
Using the latent representation learned by the VAE, we next construct a surrogate network that predicts the boundary responses directly from the latent vectors.

\paragraph{Training Details (Surrogate Network).}
Using \(4\,500\) data pairs for training and \(500\) for validation, we trained the network for \(2\,000\) epochs with batches of size \(128\) and the Adam optimizer. By \eqref{eq:numerics-overparametrization}, the network has \(M=18\,000\) hidden units. The initial learning rate was \(10^{-3}\) and was reduced by a factor of \(0.75\) after \(1\,000\) epochs. We used the ReLU activation function in the surrogate network for this experiment. The loss curves are shown in Figure~\ref{fig:Network-cracks-loss}.

\paragraph{Reconstruction.}
Two target coefficients were drawn from the underlying crack model and reconstructed by minimizing the surrogate objective \(J^{\mathrm{Net}}(z)\) in \eqref{eq:numerics-boundary-objective-surrogate} using gradient descent with Armijo line search. The initial latent vectors were sampled from the standard Gaussian prior. The optimization is performed in the learned latent space \(\IR^8\), and the surrogate maps the latent vectors to the boundary responses.

The reconstruction processes for the two targets are shown in Figures~\ref{fig:cracks-sequence-1} and~\ref{fig:cracks-sequence-2}. Each figure displays the initial guess, selected iterates, and the final reconstruction after \(100\) iterations.

For the displayed targets, the VAE-based representation and surrogate recover the dominant crack geometry and the main features of the background field. The visible reconstruction errors are concentrated near the crack boundaries and in parts of the background. The softened transitions near the cracks are empirical features of the learned reconstruction, rather than a necessary consequence of continuity of the decoder with respect to the latent variable.

Along these optimization trajectories, the surrogate objective follows the overall trend of the FEM objective. We also show the evolution of the \(L^2(\Omega)\)-error between the target and reconstructed coefficient fields.

\begin{figure}
    \centering
    \includegraphics[width=\linewidth]{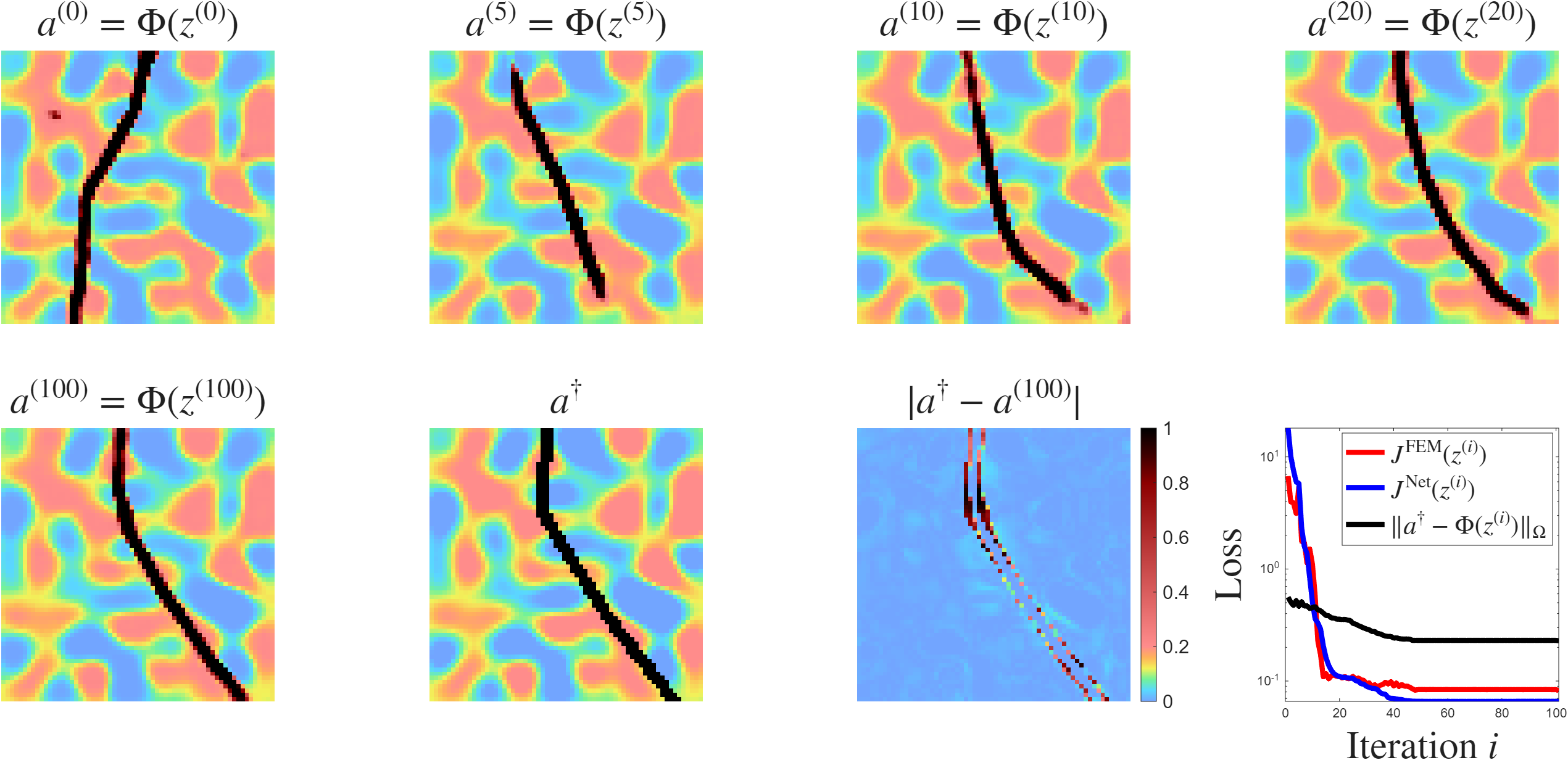}
    \caption{Surrogate-based reconstruction of a target coefficient \(a^\dagger\) from the crack family. \emph{Top row:} selected iterates \(a^{(i)}=\Phi(z^{(i)})\). \emph{Bottom row:} final reconstruction \(a^{(100)}\), target \(a^\dagger\), pointwise absolute error \(\lvert a^\dagger-a^{(100)}\rvert\), and histories of \(J^{\mathrm{Net}}(z^{(i)})\), \(J^{\mathrm{FEM}}(z^{(i)})\), and \(\lVert a^\dagger-a^{(i)}\rVert_{L^2(\Omega)}\).}
    \label{fig:cracks-sequence-1}
\end{figure}

\begin{figure}
    \centering
    \includegraphics[width=\linewidth]{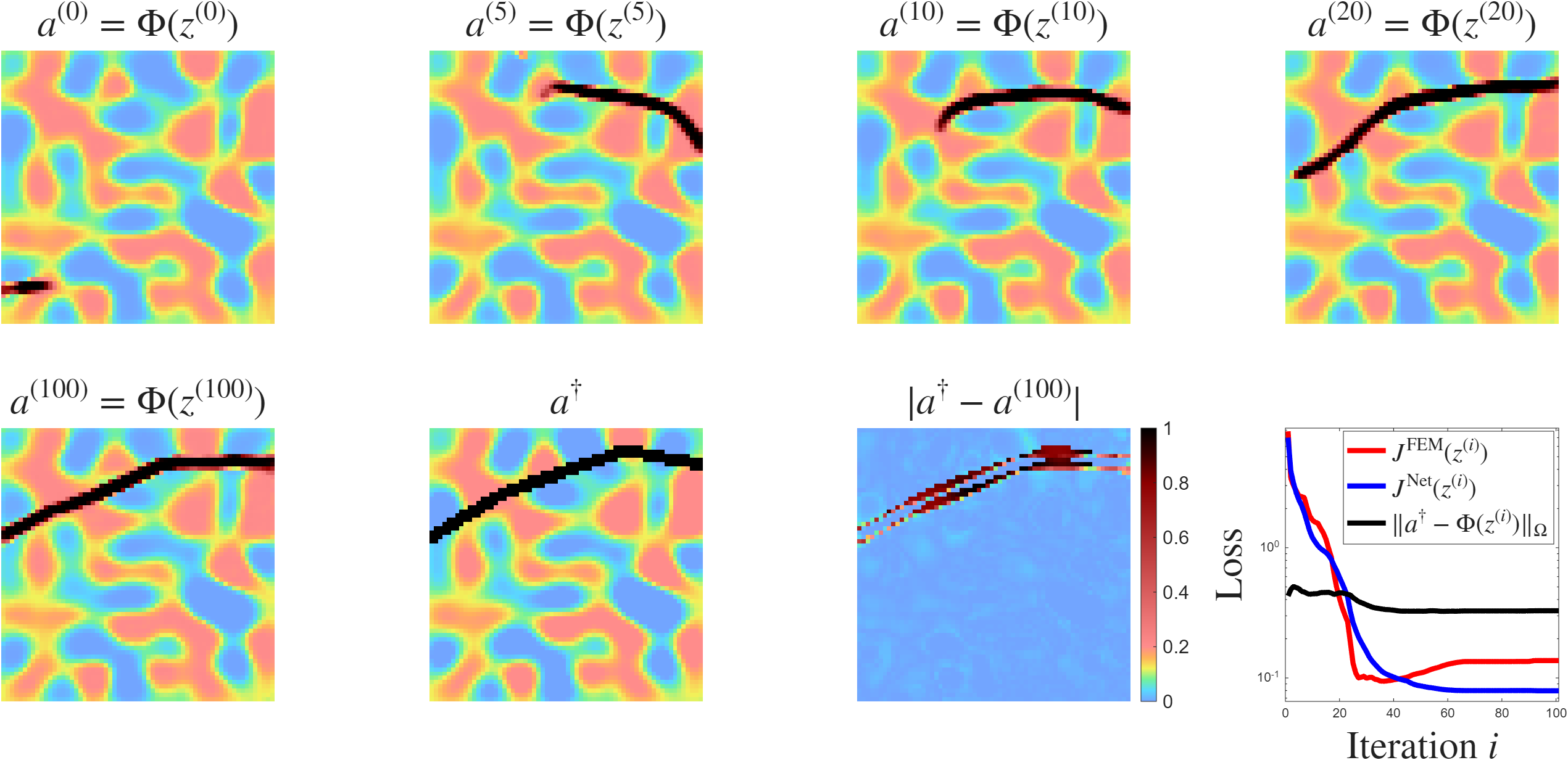}
    \caption{Surrogate-based reconstruction of another target coefficient \(a^\dagger\) from the crack family. \emph{Top row:} selected iterates \(a^{(i)}=\Phi(z^{(i)})\). \emph{Bottom row:} final reconstruction \(a^{(100)}\), target \(a^\dagger\), pointwise absolute error \(\lvert a^\dagger-a^{(100)}\rvert\), and histories of \(J^{\mathrm{Net}}(z^{(i)})\), \(J^{\mathrm{FEM}}(z^{(i)})\), and \(\lVert a^\dagger-a^{(i)}\rVert_{L^2(\Omega)}\).}
    \label{fig:cracks-sequence-2}
\end{figure}

\paragraph{Refined Reconstructions.}
We further refine the reconstruction by using the recovered latent vector as the initial guess for minimizing the FEM objective \(J^{\mathrm{FEM}}(z)\). Starting from the latent vector obtained after \(150\) iterations applied to \(J^{\mathrm{Net}}(z)\), we perform an additional \(50\) Newton-CG iterations with Armijo line search using the FEM model. Figure~\ref{fig:cracks-refined-reconstruction} shows the results for four targets. The columns contain the target coefficients, the refined reconstructions, and the corresponding loss curves, respectively.

    \begin{figure}
        \centering
        \setlength{\tabcolsep}{2pt} 
        \begin{tabular}{ccc}
        \textbf{Target \(a^\dagger\)} &
        \textbf{\(a^{(200)}=\Phi(z^{(200)})\)} &
        \textbf{Losses}
        \\[0.4em]
        \includegraphics[height=4.5cm]{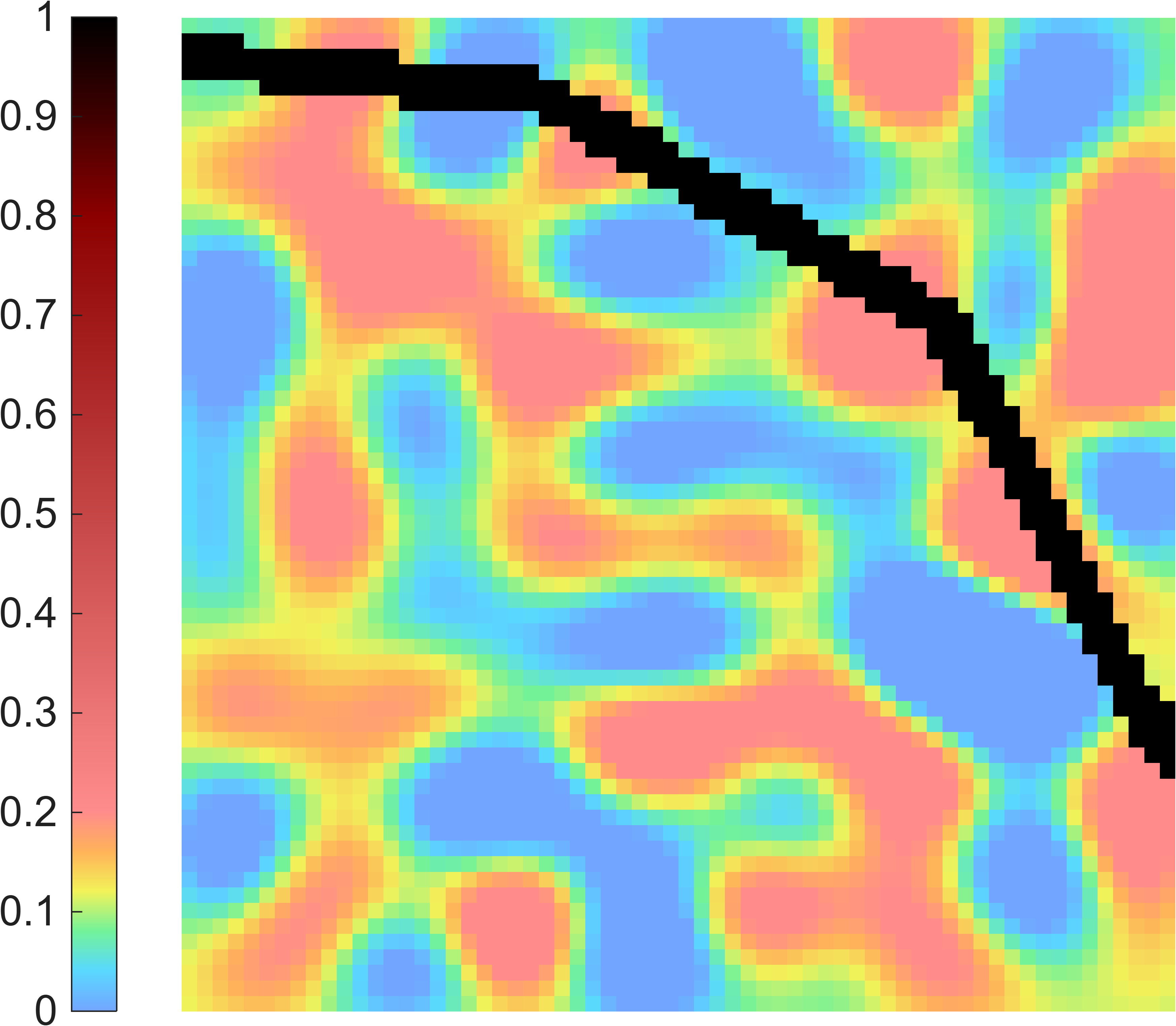} &
        \raisebox{0.07cm}{\includegraphics[height=4.35cm]{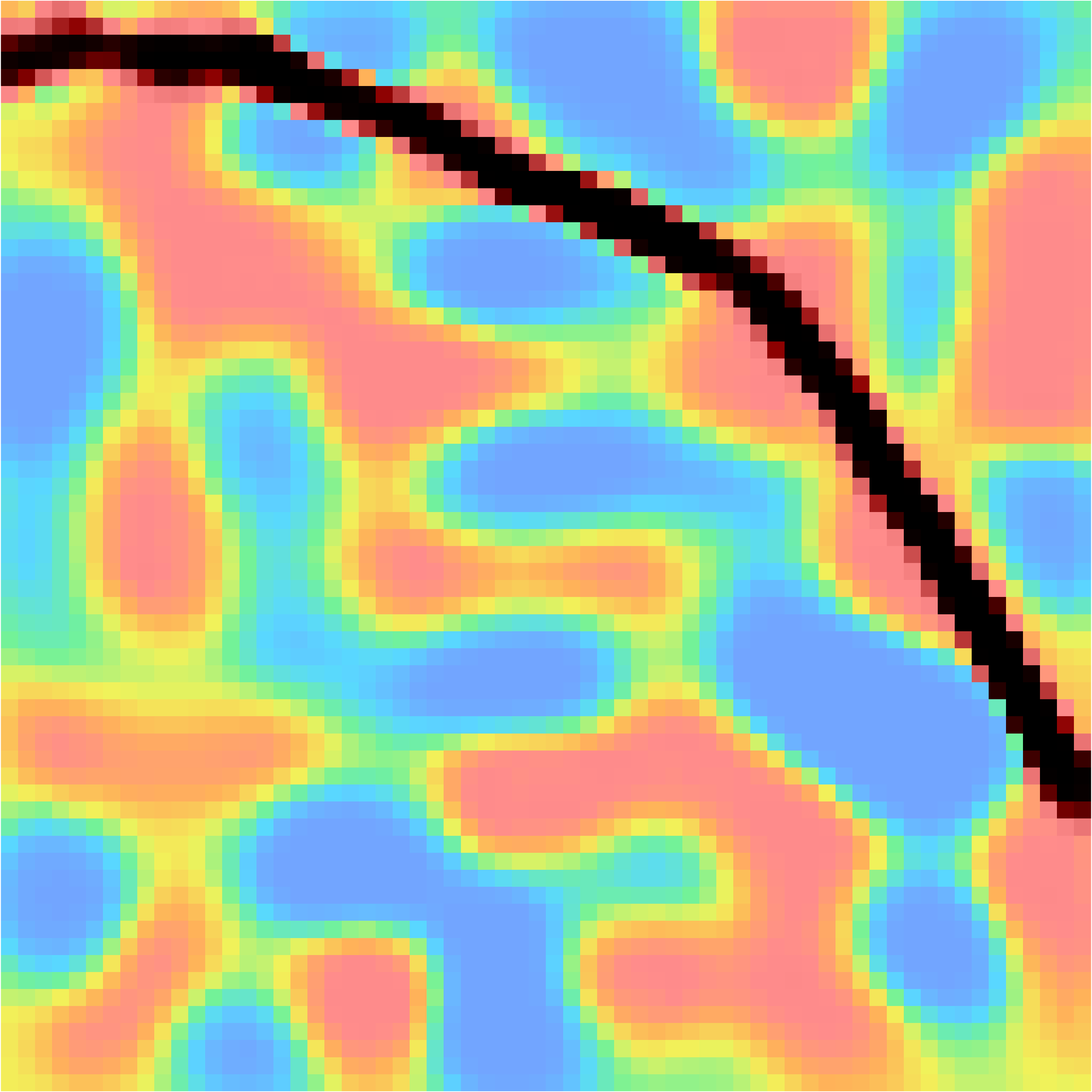}} &
        \includegraphics[height=4.6cm]{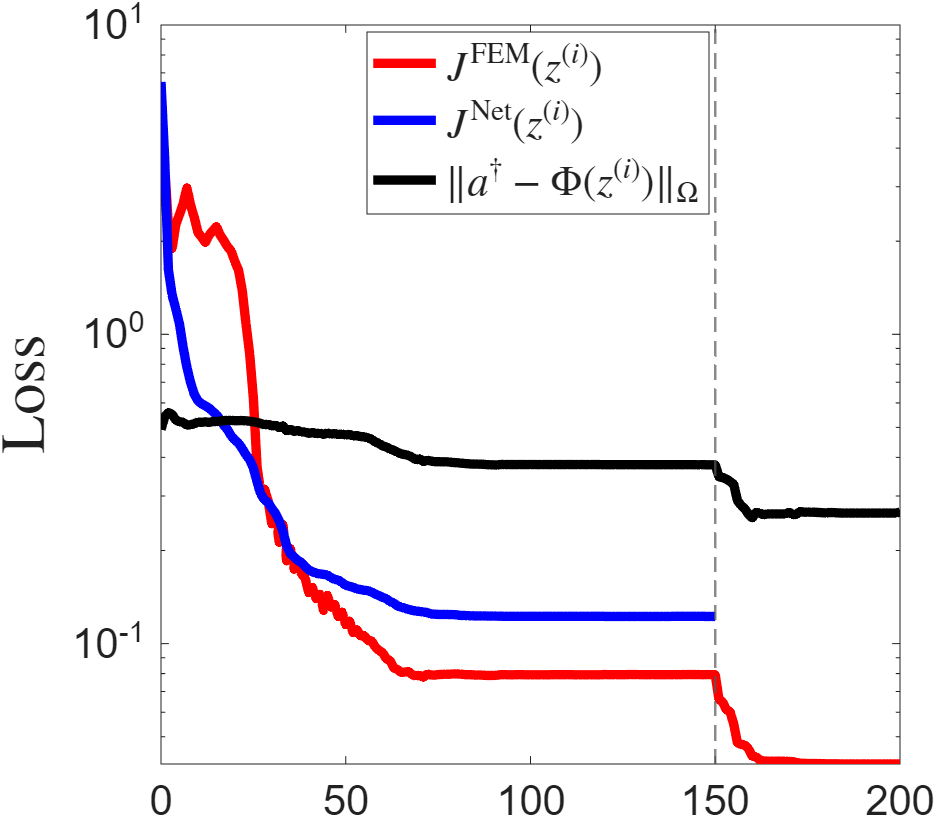}
        \\[0.5em]
        \includegraphics[height=4.5cm]{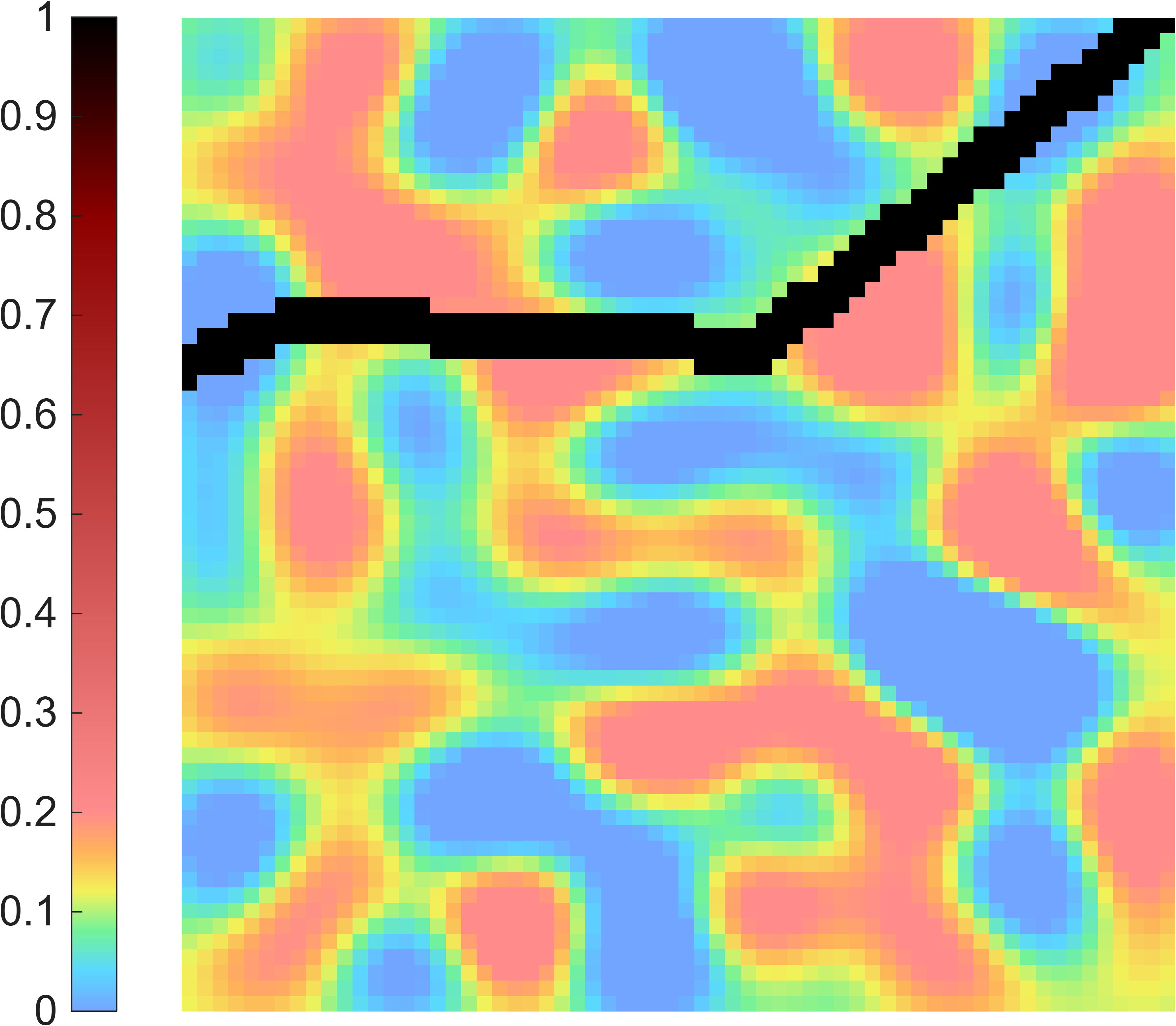} &
        \raisebox{0.07cm}{\includegraphics[height=4.35cm]{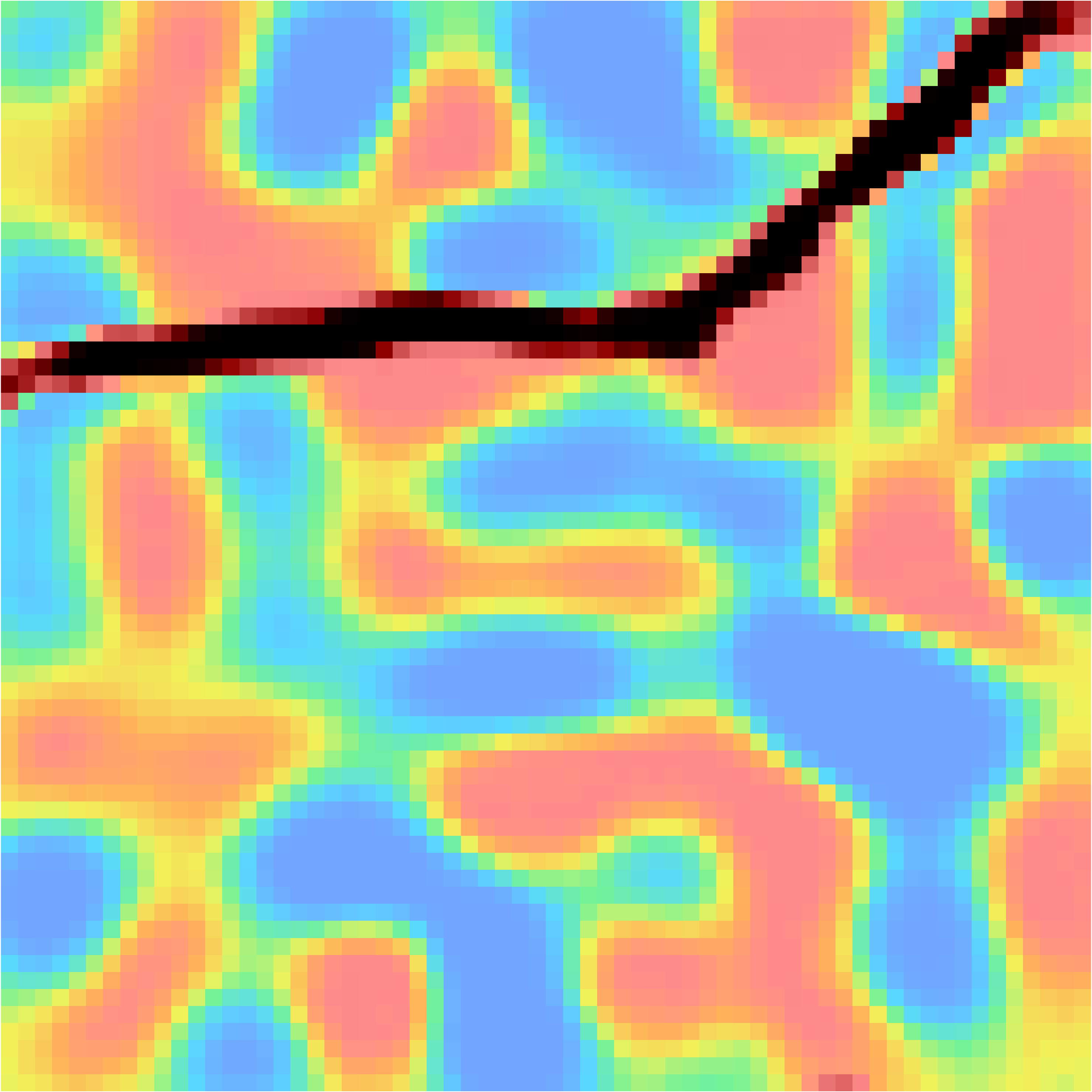}} &
        \includegraphics[height=4.6cm]{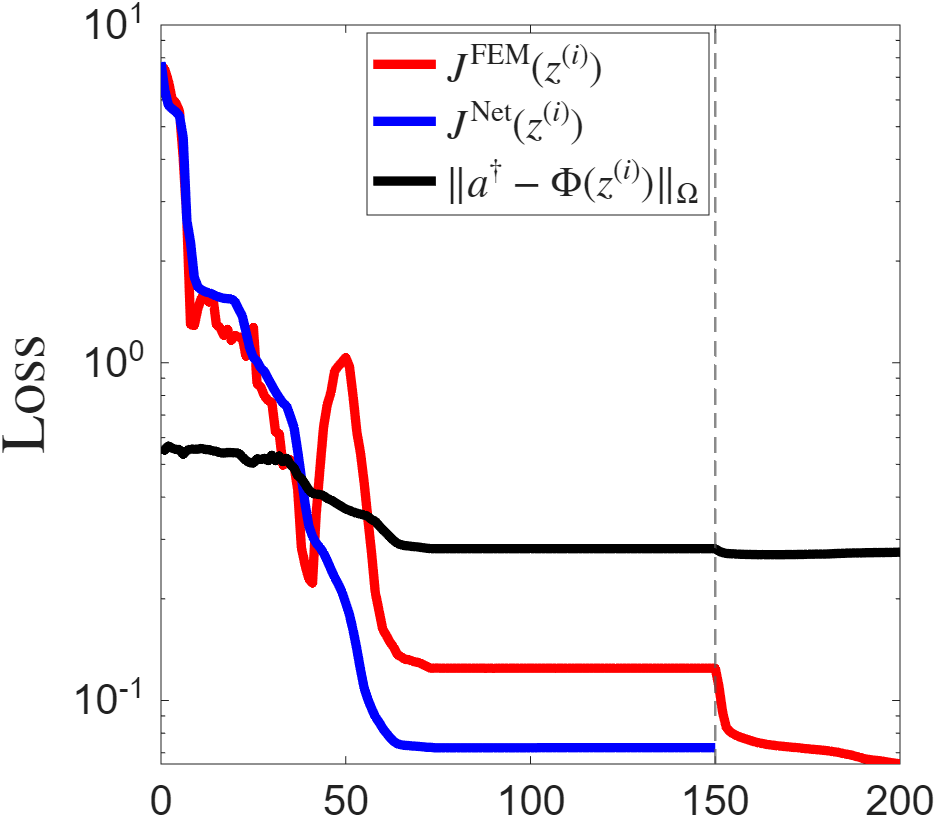}
        \\[0.5em]
        \includegraphics[height=4.5cm]{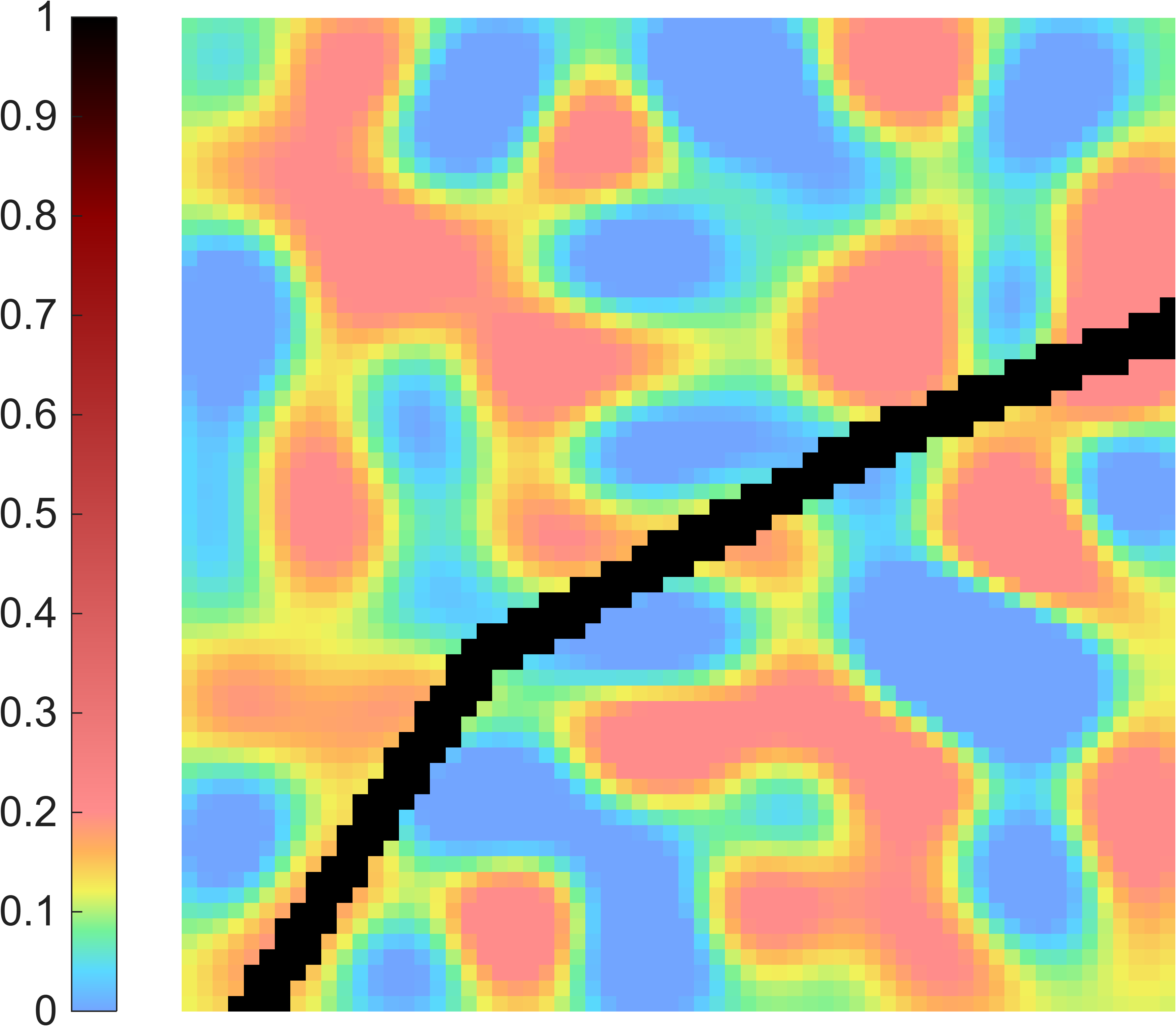} &
        \raisebox{0.07cm}{\includegraphics[height=4.35cm]{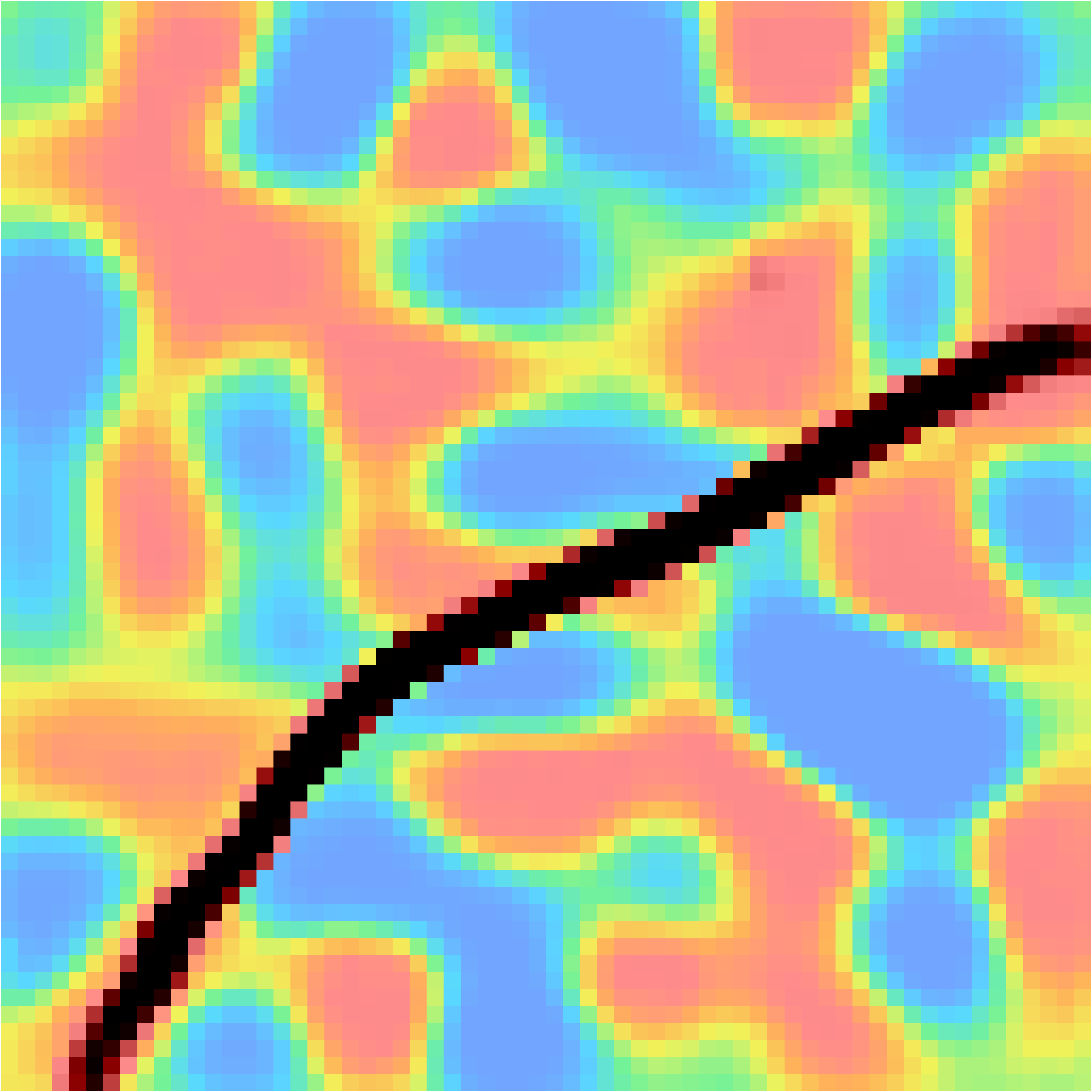}} &
        \includegraphics[height=4.6cm]{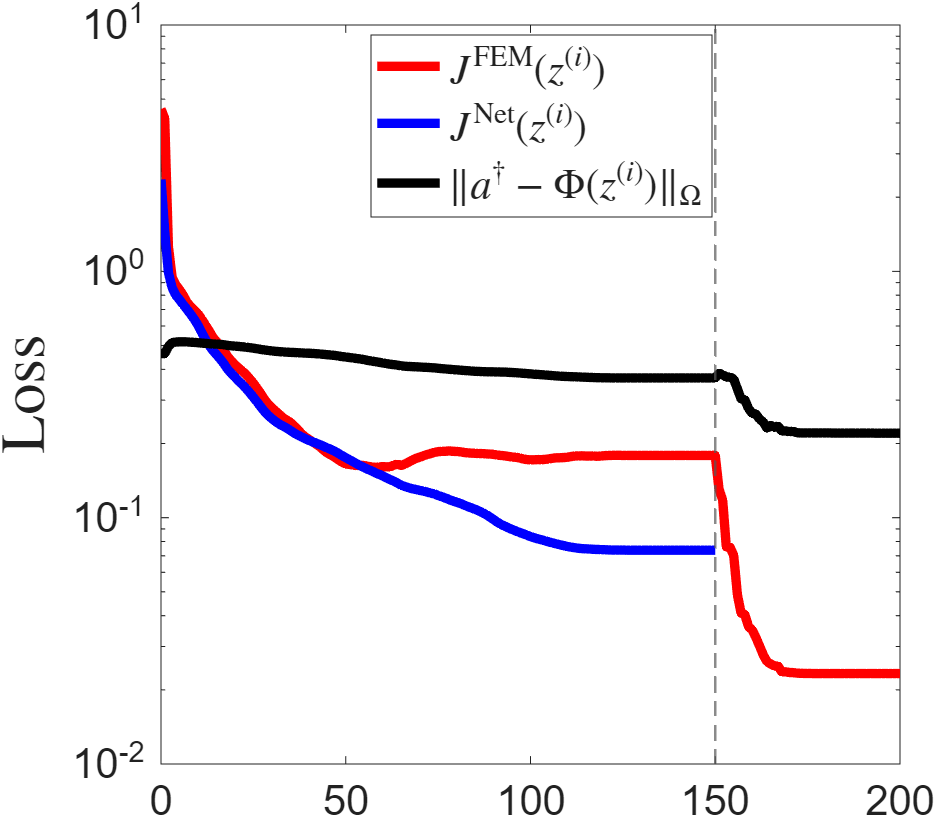}
        \\[0.5em]
        \includegraphics[height=4.5cm]{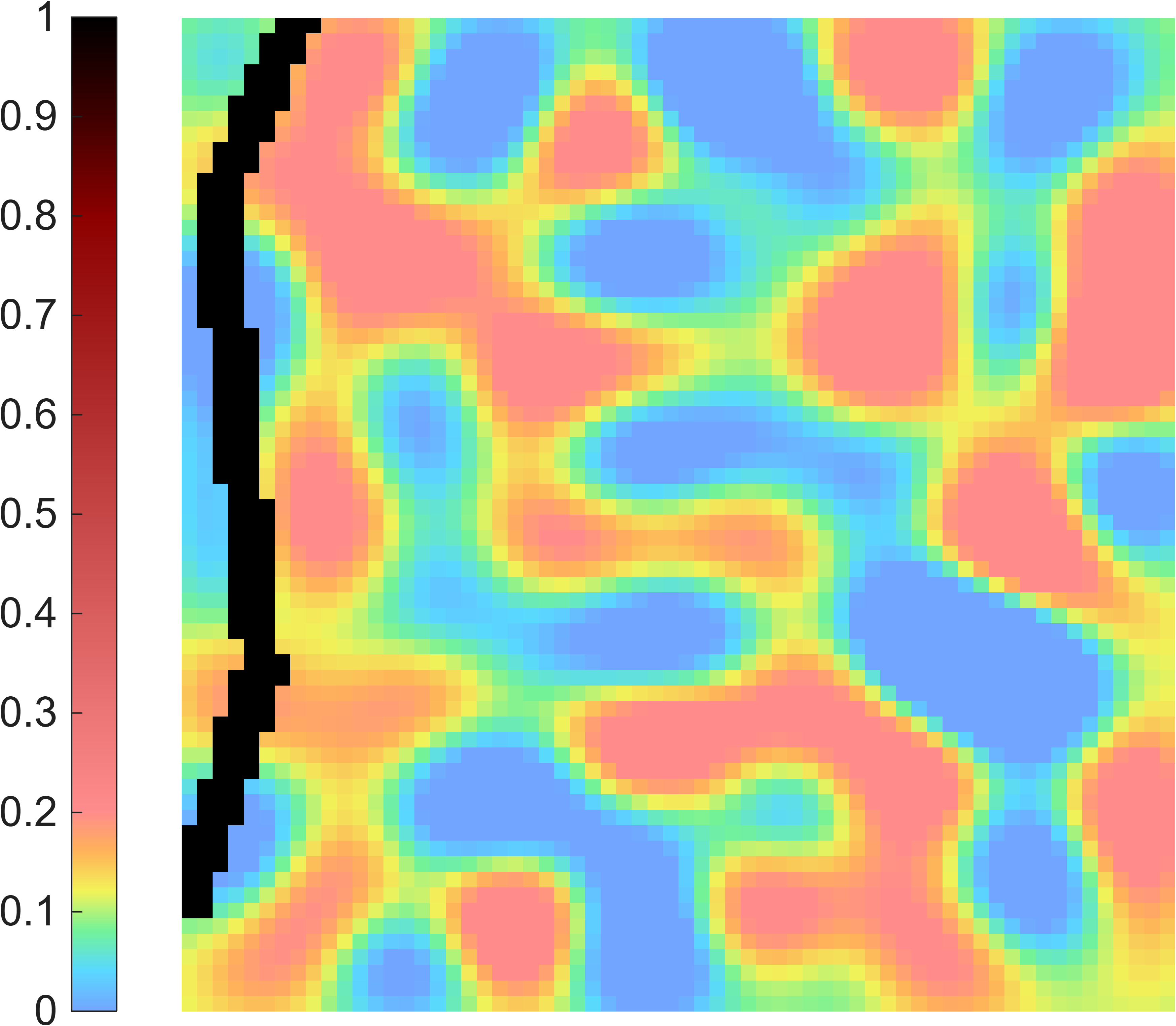} &
        \raisebox{0.07cm}{\includegraphics[height=4.35cm]{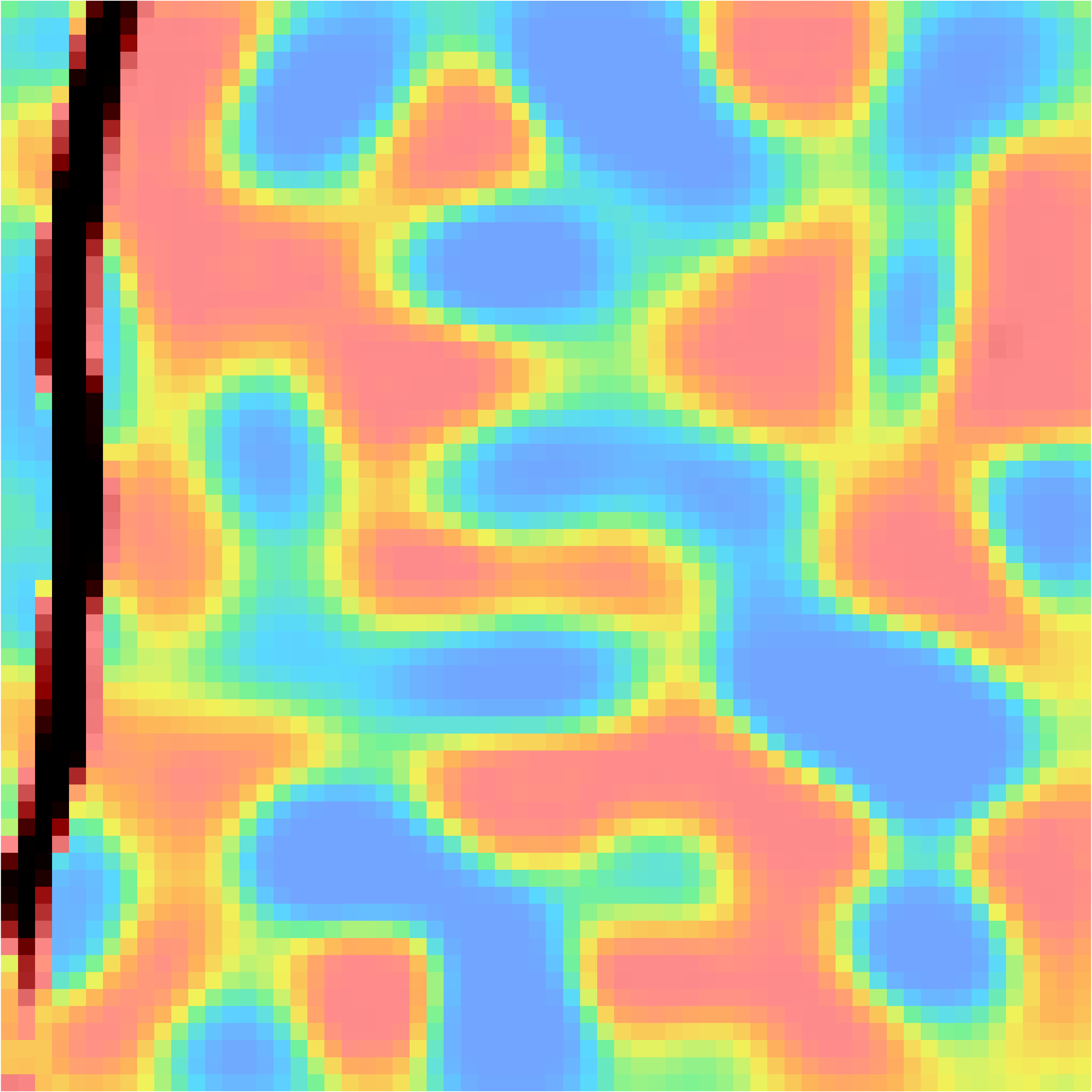}} &
        \includegraphics[height=4.6cm]{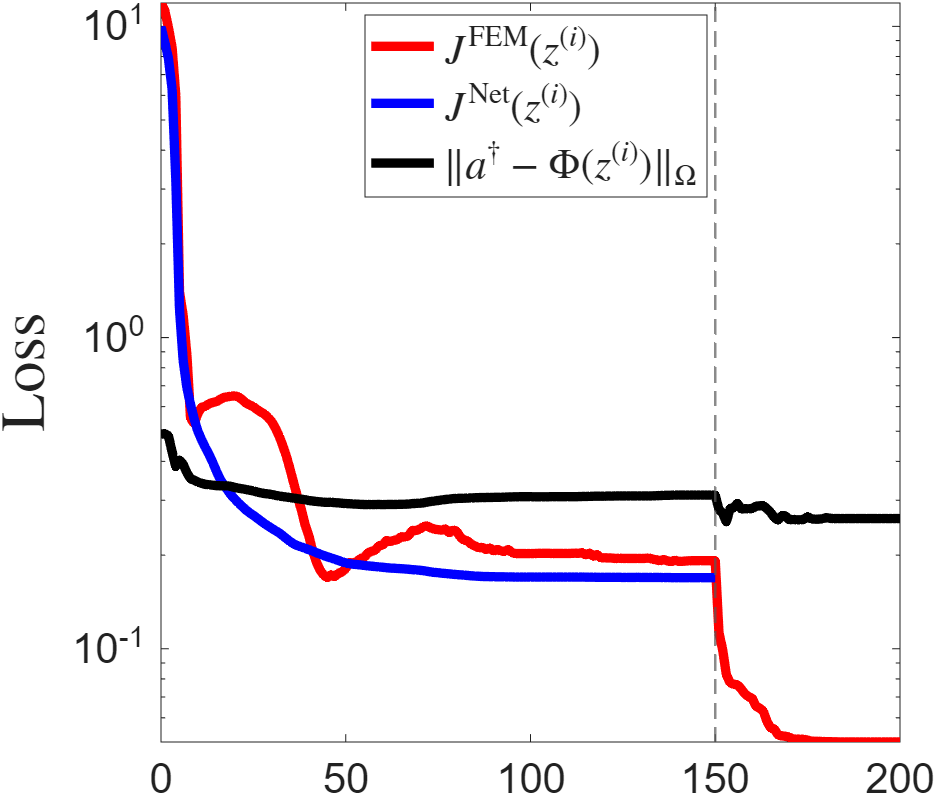}
        \\
        & & \small Iteration \(i\)
        \end{tabular}
        \caption{Refined reconstructions for four target coefficients from the crack model. \emph{Columns:} target coefficient \(a^\dagger\), final reconstruction \(a^{(200)}=\Phi(z^{(200)})\), and histories of \(J^{\mathrm{Net}}(z^{(i)})\), \(J^{\mathrm{FEM}}(z^{(i)})\), and \(\lVert a^\dagger-a^{(i)}\rVert_{L^2(\Omega)}\). The vertical dashed line at \(i=150\) marks the transition from surrogate-based gradient descent to FEM-based Newton-CG iterations.}
        \label{fig:cracks-refined-reconstruction}
    \end{figure}
    For the displayed targets, the loss curves in the third column of Figure~\ref{fig:cracks-refined-reconstruction} show that the Newton refinement yields a rapid decrease in the FEM objective. The accompanying reduction in the \(L^2(\Omega)\)-error shows that the refinement improves both the boundary-data fit and the coefficient reconstruction in these examples.

    \subsection{Computational Cost}
    \label{sec:numerics-computational-cost}
    
    We compare the cost of evaluating the neural surrogate and the finite element model, and report network training and reconstruction times. The GPU computations used an NVIDIA RTX PRO 2000 Blackwell Generation GPU.
    
    \paragraph{Forward Evaluation Time.}
    The reported times are averages over \(100\) evaluations. Each evaluation maps either a latent vector \(z\) or a coefficient field \(a\) to the boundary responses for \(L=10\) prescribed excitations.
    
    The surrogate evaluates the network on a latent vector. The complete FEM pipeline first decodes the latent vector into a coefficient field and then solves the finite element problems. We also measure FEM evaluation time with the coefficient field supplied directly, excluding decoder cost. This comparison uses the network trained on the crack dataset. Table~\ref{tab:forward-evaluation-times} reports the CPU and GPU evaluation times.
    
    \begin{table}
        \centering
        \begin{tabular}{@{}lcc@{}}
            \toprule
            Method & CPU & GPU \\
            \midrule
            Surrogate network
            & \(4.101\,\mathrm{ms}\)
            & \(2.547\,\mathrm{ms}\) \\
            Decoder + FEM
            & \(445.889\,\mathrm{ms}\)
            & \(26.117\,\mathrm{ms}\) \\
            FEM without decoder
            & \(438.631\,\mathrm{ms}\)
            & \(24.521\,\mathrm{ms}\) \\
            \bottomrule
        \end{tabular}
        \caption{Average CPU and GPU wall-clock times per forward evaluation, computed from \(100\) evaluations. Each evaluation produces the boundary responses for \(L=10\) boundary excitations.}
        \label{tab:forward-evaluation-times}
    \end{table}
    
    In this implementation, the surrogate evaluation is approximately \(109\) times faster than the complete decoder--FEM pipeline on the CPU. On the GPU, the corresponding speedup is approximately a factor of \(10\). The measured cost of evaluating the decoder is small compared with the subsequent FEM solve. Thus, for this implementation and discretization, the finite element solves dominate the cost of the physics-based forward map.

    \paragraph{Offline Training.}
    The computational savings obtained from the surrogate require an initial offline stage in which the VAE and surrogate network are trained. For the crack experiment, the training times are reported in Table~\ref{tab:training-times}. The surrogate dataset contained \(5\,000\) generated pairs, of which \(4\,500\) were used for training and \(500\) for validation; training required approximately \(9\) minutes and \(32\) seconds. The VAE used \(9\,000\) training fields and \(1\,000\) validation fields and required approximately \(2\) hours and \(34\) minutes. Both networks were trained on the GPU.
    
    \begin{table} 
        \centering 
        \begin{tabular}{@{}lccc@{}}
        \toprule
        Model & Dataset split & Number of epochs & Training time \\
        \midrule
        Surrogate network & \(4\,500+500\) & \(2\,000\) & \(9\,\mathrm{min}\ 32\,\mathrm{s}\) \\
        VAE & \(9\,000+1\,000\) & \(10\,000\) & \(2\,\mathrm{h}\ 34\,\mathrm{min}\) \\
        \bottomrule
        \end{tabular}
         \caption{Training times for the surrogate network and VAE used in the crack experiment. The dataset split is reported as training plus validation samples. All training was performed on the GPU.}
         \label{tab:training-times} 
        \end{table}

    The longer VAE training time reflects, among other factors, its larger number of epochs and more complex architecture. We have not systematically studied the number of epochs required for a prescribed reconstruction quality. The VAE loss function also changes as the KL-divergence weight is annealed, whereas the surrogate loss function remains fixed throughout training.

    \paragraph{Reconstruction Time.}
    We finally report the wall-clock times for the inverse reconstruction procedure on the GPU. Since the number of line-search steps and the behavior of the optimization depend on both the target coefficient and the initialization, these timings should be interpreted as representative values rather than fixed costs.
    
    In the representative surrogate-only run, \(100\) gradient descent iterations with Armijo line search require approximately \(5\) seconds after offline training.
    
    For the refined reconstruction pipeline, we first perform \(150\) iterations using the surrogate objective, followed by \(50\) Newton-CG iterations using the full FEM objective. In a representative run, the surrogate stage required approximately \(8\) seconds, whereas the subsequent FEM-based Newton-CG refinement required approximately \(132\) seconds, giving a total reconstruction time of approximately \(140\) seconds. Thus, in the refined pipeline, only a small fraction of the total computational time is spent in the surrogate-based initialization, while the FEM-based refinement accounts for the dominant part of the runtime. The reconstruction times are summarized in Table~\ref{tab:reconstruction-times}.
    
    \begin{table}
        \centering
        \small
        \begin{tabular}{@{}lcccc@{}}
            \toprule
            \begin{tabular}{@{}c@{}}
                Reconstruction \\
                method
            \end{tabular}
            &
            \begin{tabular}{@{}c@{}}
                Number of \\
                iterations
            \end{tabular}
            &
            \begin{tabular}{@{}c@{}}
                Surrogate \\
                stage
            \end{tabular}
            &
            \begin{tabular}{@{}c@{}}
                FEM \\
                refinement
            \end{tabular}
            &
            \begin{tabular}{@{}c@{}}
                Total \\
                time
            \end{tabular}
            \\
            \midrule
            Surrogate only
            & \(100\) GD
            & \(5\,\mathrm{s}\)
            & --
            & \(5\,\mathrm{s}\)
            \\
            
            Refined pipeline
            & \(150\) GD \(+\) \(50\) Newton-CG
            & \(8\,\mathrm{s}\)
            & \(132\,\mathrm{s}\)
            & \(140\,\mathrm{s}\)
            \\
            \bottomrule
        \end{tabular}
        \caption{Representative GPU wall-clock times for the two reconstruction procedures.
        Here, GD denotes gradient descent with Armijo line search applied to the surrogate objective; the refined pipeline then uses FEM-based Newton-CG iterations. The reported times
        depend on the target coefficient and the initial latent vector, and should
        therefore be interpreted as representative values.}
        \label{tab:reconstruction-times}
    \end{table}

    The timing results illustrate the lower forward-evaluation cost of the surrogate in this implementation. Constructing the learned representation and surrogate requires an offline investment, after which the reported surrogate-only reconstruction takes a few seconds. The refined pipeline uses the surrogate to obtain an initial reconstruction before applying the more expensive FEM model.

%% file: sec-conclusions-rev0.tex
\section{Conclusions}
\label{sec:conclusions}

We have analyzed latent reconstruction of a leading elliptic coefficient from
finitely many Neumann excitations and their full Dirichlet traces. If the
linearized boundary map is injective on the \(m\)-dimensional latent tangent
space, at most \(m\) excitations can be selected to obtain local Lipschitz
stability. Convergence of the finite element coefficient sensitivities
preserves this stability on sufficiently fine meshes, with a stability
constant and neighborhood independent of the mesh size. The argument also applies to
the boundary \(L^2\)-norm.

The surrogate analysis distinguishes two approximation requirements.
Uniform value accuracy gives a coefficient-error estimate in terms of
representation error, data noise, finite element error, and surrogate error
for reconstructions satisfying residual comparison in the local stability set.
Representation error is measured relative to a comparison coefficient in that set.
Derivative accuracy further ensures that the surrogate itself remains locally injective and Lipschitz
stable. The results are local, and the selected tests and stability neighborhood
may depend on the reference coefficient. The analysis does not provide a
convergence guarantee for the numerical optimizer.

The numerical experiments illustrate the framework on analytical and learned
families of elliptic inclusions and crack-like coefficients. For the displayed
targets, latent optimization recovers the principal geometric features, and
FEM-based refinement in the same latent space further reduces the boundary
misfit and coefficient error. In this implementation, the trained surrogate
is substantially cheaper to evaluate than the finite element model. These two-dimensional
experiments illustrate the computational method; they do not verify the
separation or surrogate approximation hypotheses. In particular, the moving
binary-interface and ReLU representations fall outside the global \(C^1\)
assumptions.

%% file: sec-appendix-rev0.tex
\appendix

\section{CGO Verification of the Separation Condition}
\label{sec:cgo-verification-of-the-separation-condition}

In the continuous injectivity argument, the richness of the admissible boundary experiments enters through the separation condition \eqref{eq:separation-condition-tangent-space}. The next proposition verifies this condition in dimensions \(d\ge 3\) using standard complex geometrical optics (CGO) solutions, following \cite{SU87}. The corresponding Lipschitz-boundary setting was considered in \cite{Ale90}; for the two-dimensional uniqueness result, see \cite{Nachman1996}. The proposition supplies an independent sufficient criterion for the abstract separation condition. Its smoothness and dimensional assumptions are not imposed on the two-dimensional numerical experiments.

\begin{prop}[CGO verification of the separation condition]
\label{prop:cgo-verification-of-the-separation-condition}
Assume that \(d\ge 3\), that \(\Omega\subset\mathbb R^d\) is bounded, connected, and has smooth boundary, and that
\begin{equation}
a^\dagger\in C^\infty(\overline\Omega),
\qquad
0<c_0\le a^\dagger(x)
\quad \forall x\in\overline\Omega
\label{eq:cgo-reference-coefficient-assumptions}
\end{equation}
Assume also that the admissible Neumann test class is the full compatible class,
\begin{equation}
\mathcal G=H^{-1/2}_\diamond(\partial\Omega)
\label{eq:cgo-full-neumann-test-class}
\end{equation}
Let \(\mathcal U(a^\dagger)\) be defined by \eqref{eq:background-state-family}. If \(\delta a\in L^\infty(\Omega)\) satisfies
\begin{equation}
\int_\Omega \delta a\,\nabla u\cdot\nabla v\,dx=0
\quad \forall u,v\in \mathcal U(a^\dagger)
\label{eq:cgo-orthogonality-assumption}
\end{equation}
then
\begin{equation}
\delta a=0
\quad \text{a.e. in } \Omega
\label{eq:cgo-orthogonality-conclusion}
\end{equation}
Consequently, \eqref{eq:separation-condition-tangent-space} holds for every latent tangent space \(T^\Phi_{z^\dagger}\subset L^\infty(\Omega)\) associated with \(a^\dagger=\Phi(z^\dagger)\).
\end{prop}

\begin{proof}
\proofparagraph{Liouville Transform.} Set
\begin{equation}
\gamma=a^\dagger,
\qquad
s=\gamma^{1/2},
\qquad
b=\log s,
\qquad
\beta=\delta a/\gamma
\label{eq:cgo-liouville-quantities}
\end{equation}
We extend \(\gamma\) to a smooth positive function on \(\mathbb R^d\), still denoted by \(\gamma\), such that
\begin{equation}
\gamma=1
\quad \text{outside a ball } B\supset\overline\Omega
\label{eq:cgo-extension-normalization}
\end{equation}
After decreasing \(c_0\) if necessary, we may also assume that
\(\gamma\geq c_0>0\) throughout \(\mathbb R^d\).
The corresponding Schrödinger potential is
\begin{equation}
q=s^{-1}\Delta s=\Delta b+|\nabla b|^2
\label{eq:cgo-potential-q}
\end{equation}
and hence \(q\in C_c^\infty(\mathbb R^d)\).
Let \(u=s^{-1}v\). A direct computation gives
\begin{align}
\nabla\cdot(\gamma\nabla u)
&=\nabla\cdot\bigl(s^2\nabla(s^{-1}v)\bigr)
\label{eq:cgo-liouville-divergence-start}\\
&=s\Delta v-v\Delta s
\label{eq:cgo-liouville-divergence-middle}\\
&=s(\Delta v-qv)
\label{eq:cgo-liouville-divergence-end}
\end{align}
Thus \(u\) solves the conductivity equation if and only if \(v=su\) solves
\begin{equation}
(-\Delta+q)v=0
\label{eq:cgo-schroedinger-equation}
\end{equation}
This is the Liouville form used below.

\proofparagraph{The Distribution Associated with the Perturbation.}
Extend \(\beta\) by zero outside \(\Omega\) and define
\begin{equation}
\beta_0=\mathbf 1_\Omega\beta
\label{eq:cgo-beta-zero-definition}
\end{equation}
We introduce the compactly supported distribution
\begin{equation}
T_\beta=\frac12\bigl(\Delta\beta_0+2\nabla b\cdot\nabla\beta_0\bigr)
\quad \text{in } \mathcal D'(\mathbb R^d)
\label{eq:cgo-distribution-definition}
\end{equation}
Equivalently, for every \(\phi\in C_c^\infty(\mathbb R^d)\),
\begin{equation}
\langle T_\beta,\phi\rangle
=\frac12\int_\Omega \beta\bigl(\Delta\phi-2\nabla b\cdot\nabla\phi-2(\Delta b)\phi\bigr)\,dx
\label{eq:cgo-distribution-action}
\end{equation}
We use this form of the distributional definition in the product calculation below.

\proofparagraph{CGO Product Identity.}
Let \(v_1\) and \(v_2\) solve \eqref{eq:cgo-schroedinger-equation} in a neighborhood of \(\overline\Omega\), and set
\begin{equation}
u_j=s^{-1}v_j,
\quad j=1,2
\label{eq:cgo-conductivity-from-schroedinger}
\end{equation}
Then the Liouville transform gives
\begin{equation}
\nabla\cdot(\gamma\nabla u_j)=0
\quad \text{in } \Omega,\quad j=1,2
\label{eq:cgo-background-equations-from-cgo}
\end{equation}
We claim that
\begin{equation}
\langle T_\beta,v_1v_2\rangle
=\int_\Omega \delta a\,\nabla u_1\cdot\nabla u_2\,dx
\label{eq:cgo-product-identity}
\end{equation}
Indeed, applying \eqref{eq:cgo-distribution-action} with \(\phi=v_1v_2\) gives
\begin{equation}
\langle T_\beta,v_1v_2\rangle
=\frac12\int_\Omega \beta\bigl(\Delta(v_1v_2)-2\nabla b\cdot\nabla(v_1v_2)-2(\Delta b)v_1v_2\bigr)\,dx
\label{eq:cgo-action-on-product}
\end{equation}
Since \(v_j\) solve \eqref{eq:cgo-schroedinger-equation}, we have
\begin{equation}
\Delta(v_1v_2)=2qv_1v_2+2\nabla v_1\cdot\nabla v_2
\label{eq:cgo-laplacian-product}
\end{equation}
Using \eqref{eq:cgo-potential-q}, we obtain
\begin{align}
\langle T_\beta,v_1v_2\rangle
&=\int_\Omega \beta\Bigl(\nabla v_1\cdot\nabla v_2
-\nabla b\cdot(v_2\nabla v_1+v_1\nabla v_2)
+|\nabla b|^2v_1v_2\Bigr)\,dx
\label{eq:cgo-action-expanded}\\
&=\int_\Omega \beta(\nabla v_1-v_1\nabla b)\cdot(\nabla v_2-v_2\nabla b)\,dx
\label{eq:cgo-action-factorized}
\end{align}
On the other hand,
\begin{align}
\nabla u_j
&=e^{-b}(\nabla v_j-v_j\nabla b),
\quad j=1,2
\label{eq:cgo-gradient-u-j}\\
\delta a
&=e^{2b}\beta
\label{eq:cgo-delta-a-beta}
\end{align}
Therefore
\begin{align}
\int_\Omega \delta a\,\nabla u_1\cdot\nabla u_2\,dx
&=\int_\Omega \beta(\nabla v_1-v_1\nabla b)\cdot(\nabla v_2-v_2\nabla b)\,dx
\label{eq:cgo-energy-factorized}\\
&=\langle T_\beta,v_1v_2\rangle
\label{eq:cgo-energy-equals-distribution}
\end{align}
This proves \eqref{eq:cgo-product-identity}.

\proofparagraph{Admissibility of the CGO Boundary Data.}
Let \(v\) be any of the Schrödinger solutions used above and let \(u=s^{-1}v\). Its conormal trace is
\begin{equation}
g_N=\gamma\partial_n u|_{\partial\Omega}
\label{eq:cgo-neumann-trace}
\end{equation}
Since \(u\) solves the homogeneous conductivity equation, the compatibility condition follows from
\begin{equation}
\langle g_N,1\rangle_{\partial\Omega}
=\int_{\partial\Omega}\gamma\partial_n u\,dS
=\int_\Omega \nabla\cdot(\gamma\nabla u)\,dx
=0
\label{eq:cgo-boundary-compatibility}
\end{equation}
Thus \(g_N\in H^{-1/2}_\diamond(\partial\Omega)=\mathcal G\). Subtracting
the boundary mean
\(
|\partial\Omega|^{-1}\int_{\partial\Omega}u\,dS
\)
from \(u\) places the solution in \(V\) and leaves its gradient unchanged.
Hence the CGO-generated conductivity solutions are admissible elements of
\(\mathcal U(a^\dagger)\). Although the CGO solutions are complex-valued, the bilinear
identity follows by applying the real identity to real and imaginary parts.
Therefore \eqref{eq:cgo-orthogonality-assumption} and
\eqref{eq:cgo-product-identity} imply
\begin{equation}
\langle T_\beta,v_1v_2\rangle=0
\label{eq:cgo-product-pairing-vanishes}
\end{equation}
for all CGO pairs \(v_1,v_2\) constructed below.

\proofparagraph{Passage to Fourier Modes.}
Fix \(k\in\mathbb R^d\setminus\{0\}\). Since \(d\ge 3\), choose \(\theta,\eta\in\mathbb R^d\) such that
\begin{equation}
|\theta|=|\eta|=1,
\qquad
\theta\cdot\eta=0,
\qquad
\theta\cdot k=0,
\qquad
\eta\cdot k=0
\label{eq:cgo-theta-eta-choice}
\end{equation}
For \(\tau>|k|/2\), define
\begin{equation}
\lambda_\tau=\bigl(\tau^2-|k|^2/4\bigr)^{1/2}
\label{eq:cgo-lambda-tau}
\end{equation}
and
\begin{equation}
\rho_1
=\tau\theta+i\bigl(k/2+\lambda_\tau\eta\bigr),
\qquad
\rho_2
=-\tau\theta+i\bigl(k/2-\lambda_\tau\eta\bigr)
\label{eq:cgo-rho-definitions}
\end{equation}
Then
\begin{align}
\rho_1+\rho_2
&=ik
\label{eq:cgo-rho-sum}\\
\rho_j\cdot\rho_j
&=0,
\quad j=1,2
\label{eq:cgo-rho-null}
\end{align}
The standard CGO construction for the smooth compactly supported potential \(q\) gives solutions of \eqref{eq:cgo-schroedinger-equation} of the form
\begin{equation}
v_j(x)=e^{\rho_j\cdot x}\bigl(1+r_j(x)\bigr),
\quad j=1,2
\label{eq:cgo-solutions}
\end{equation}
To justify the regularity needed below, fix \(-1<\delta<0\) and let
\(G_\rho\) denote the Faddeev inverse of \(-\Delta-2\rho\cdot\nabla\).
We use the standard weighted Sobolev norm
\begin{equation}
\|f\|_{H^m_\sigma}^2
=\sum_{|\alpha|\leq m}\|\langle x\rangle^\sigma D^\alpha f\|_{L^2}^2
\label{eq:cgo-weighted-sobolev-norm}
\end{equation}
The weighted \(L^2\) estimate in \cite[Corollary~2.2]{SU87}, applied after
commuting \(G_\rho\) with derivatives, gives, for every integer \(m\geq0\),
\begin{equation}
\|G_\rho f\|_{H^m_\delta(\mathbb R^d)}
\leq \frac{C_m}{|\rho|}\|f\|_{H^m_{\delta+1}(\mathbb R^d)}
\label{eq:cgo-faddeev-weighted-estimate}
\end{equation}
The remainders satisfy \(r_j=-G_{\rho_j}\bigl(q(1+r_j)\bigr)\).
Since \(q\in C_c^\infty(\mathbb R^d)\), multiplication by \(q\) maps
\(H^m_\delta\) boundedly into \(H^m_{\delta+1}\). Using
\(|\rho_j|=\sqrt{2}\tau\), the usual contraction argument applied to
\eqref{eq:cgo-faddeev-weighted-estimate} gives, for sufficiently large \(\tau\),
\begin{equation}
\|r_j\|_{H^m_\delta(\mathbb R^d)}\leq C_m\tau^{-1},
\quad j=1,2
\label{eq:cgo-remainder-weighted-estimate}
\end{equation}
Choosing \(m>d/2+2\) and using local Sobolev embedding on \(B\), we obtain
\begin{equation}
\|r_j\|_{C^2(B)}\to0
\quad \text{as } \tau\to\infty,\quad j=1,2
\label{eq:cgo-remainder-convergence}
\end{equation}
Consequently,
\begin{align}
v_1(x)v_2(x)
&=e^{(\rho_1+\rho_2)\cdot x}\bigl(1+r_1(x)\bigr)\bigl(1+r_2(x)\bigr)
\label{eq:cgo-product-before-sum}\\
&=e^{ik\cdot x}\bigl(1+r_1(x)\bigr)\bigl(1+r_2(x)\bigr)
\label{eq:cgo-product-after-sum}
\end{align}
Hence
\begin{equation}
\|v_1v_2-e^{ik\cdot x}\|_{C^2(B)}\to0
\quad \text{as } \tau\to\infty
\label{eq:cgo-product-convergence}
\end{equation}
Since \(T_\beta\) is compactly supported in \(\overline\Omega\) and has order at most two, \eqref{eq:cgo-product-convergence} gives
\begin{equation}
\langle T_\beta,v_1v_2\rangle\to \langle T_\beta,e^{ik\cdot x}\rangle
\quad \text{as } \tau\to\infty
\label{eq:cgo-pairing-convergence}
\end{equation}
Combining this with \eqref{eq:cgo-product-pairing-vanishes}, we obtain
\begin{equation}
\langle T_\beta,e^{ik\cdot x}\rangle=0
\quad \forall k\in\mathbb R^d\setminus\{0\}
\label{eq:cgo-fourier-nonzero-modes}
\end{equation}
Since \(T_\beta\) has compact support, its Fourier transform is continuous, and therefore
\begin{equation}
\widehat{T_\beta}(k)=0
\quad \forall k\in\mathbb R^d
\label{eq:cgo-fourier-transform-zero}
\end{equation}
It follows that
\begin{equation}
T_\beta=0
\quad \text{in } \mathcal D'(\mathbb R^d)
\label{eq:cgo-distribution-zero}
\end{equation}
This is the distributional information extracted from the CGO products.

\proofparagraph{Conclusion.}
By \eqref{eq:cgo-distribution-definition} and \eqref{eq:cgo-distribution-zero},
\begin{equation}
\Delta\beta_0+2\nabla b\cdot\nabla\beta_0=0
\quad \text{in } \mathcal D'(\mathbb R^d)
\label{eq:cgo-beta-equation-nondivergence}
\end{equation}
Since \(\gamma=e^{2b}\), this is equivalent to the divergence-form equation
\begin{equation}
\nabla\cdot(\gamma\nabla\beta_0)=0
\quad \text{in } \mathcal D'(\mathbb R^d)
\label{eq:cgo-beta-equation-divergence}
\end{equation}
Elliptic hypoellipticity for this smooth uniformly elliptic operator gives
\(\beta_0\in C^\infty(\mathbb R^d)\). Since \(\beta_0\) has compact support,
testing \eqref{eq:cgo-beta-equation-divergence} with \(\beta_0\) yields
\begin{equation}
0
=\int_{\mathbb R^d}\beta_0\nabla\cdot(\gamma\nabla\beta_0)\,dx
=-\int_{\mathbb R^d}\gamma|\nabla\beta_0|^2\,dx
\label{eq:cgo-energy-zero-end}
\end{equation}
Because \(\gamma\ge c_0>0\), \eqref{eq:cgo-energy-zero-end} implies
\begin{equation}
\nabla\beta_0=0
\quad \text{in } \mathbb R^d
\label{eq:cgo-gradient-beta-zero}
\end{equation}
Since \(\beta_0\) has compact support, this constant must be zero:
\begin{equation}
\beta_0=0
\quad \text{in } \mathbb R^d
\label{eq:cgo-beta-zero}
\end{equation}
Finally, \(\delta a=\gamma\beta\) in \(\Omega\), and therefore
\begin{equation}
\delta a=0
\quad \text{a.e. in } \Omega
\label{eq:cgo-delta-a-zero-final}
\end{equation}
This proves \eqref{eq:cgo-orthogonality-conclusion}, and hence the separation condition \eqref{eq:separation-condition-tangent-space}.
\end{proof}

\begin{rem}[Admissible subclasses]
\label{rem:cgo-admissible-subclasses}
The proof only uses that \(\mathcal G\) contains the real and imaginary parts
of the conormal traces of the conductivity solutions obtained from the CGO
solutions. Hence the full-data
assumption~\eqref{eq:cgo-full-neumann-test-class} may be replaced by this
smaller richness condition.
\end{rem}